\documentclass{amsart}
\usepackage{amssymb}
\usepackage{amscd}
\usepackage{verbatim}
\usepackage{epsfig}
\begin{document}
\newcommand{\pa}{\partial}
\newcommand{\CI}{C^\infty}
\newcommand{\dCI}{\dot C^\infty}
\newcommand{\supp}{\operatorname{supp}}
\renewcommand{\Box}{\square}
\newcommand{\ep}{\epsilon}
\newcommand{\Ell}{\operatorname{Ell}}
\newcommand{\WF}{\operatorname{WF}}
\newcommand{\WFb}{\operatorname{WF}_{\bl}}
\newcommand{\WFsc}{\operatorname{WF}_{\scl}}
\newcommand{\diag}{\mathrm{diag}}
\newcommand{\sign}{\operatorname{sign}}
\newcommand{\Ker}{\operatorname{Ker}}
\newcommand{\Ran}{\operatorname{Ran}}
\newcommand{\Span}{\operatorname{Span}}
\newcommand{\sH}{\mathsf{H}}
\newcommand{\sk}{\mathsf{k}}
\newcommand{\codim}{\operatorname{codim}}
\newcommand{\Id}{\operatorname{Id}}
\newcommand{\cl}{{\mathrm{cl}}}
\newcommand{\piece}{{\mathrm{piece}}}
\newcommand{\bl}{{\mathrm b}}
\newcommand{\scl}{{\mathrm{sc}}}
\newcommand{\Psib}{\Psi_\bl}
\newcommand{\Psibc}{\Psi_{\mathrm{bc}}}
\newcommand{\Psibcc}{\Psi_{\mathrm{bcc}}}
\newcommand{\Psisc}{\Psi_\scl}
\newcommand{\Diff}{\mathrm{Diff}}
\newcommand{\Diffsc}{\Diff_\scl}
\newcommand{\BB}{\mathbb{B}}
\newcommand{\RR}{\mathbb{R}}
\newcommand{\Cx}{\mathbb{C}}
\newcommand{\NN}{\mathbb{N}}
\newcommand{\sphere}{\mathbb{S}}
\newcommand{\codimY}{k}
\newcommand{\dimX}{n}
\newcommand{\cO}{\mathcal O}
\newcommand{\cS}{\mathcal S}
\newcommand{\cP}{\mathcal P}
\newcommand{\cF}{\mathcal F}
\newcommand{\cL}{\mathcal L}
\newcommand{\cH}{\mathcal H}
\newcommand{\cG}{\mathcal G}
\newcommand{\cU}{\mathcal U}
\newcommand{\cM}{\mathcal M}
\newcommand{\cT}{\mathcal T}
\newcommand{\loc}{{\mathrm{loc}}}
\newcommand{\comp}{{\mathrm{comp}}}
\newcommand{\Tb}{{}^{\bl}T}
\newcommand{\Sb}{{}^{\bl}S}
\newcommand{\Tsc}{{}^{\scl}T}
\newcommand{\Ssc}{{}^{\scl}S}
\newcommand{\Vf}{\mathcal V}
\newcommand{\Vb}{{\mathcal V}_{\bl}}
\newcommand{\Vsc}{{\mathcal V}_{\scl}}
\newcommand{\Lambdasc}{{}^{\scl}\Lambda}
\newcommand{\etat}{\tilde\eta}
\newcommand{\Hsc}{H_{\scl}}
\newcommand{\Hscloc}{H_{\scl,\loc}}
\newcommand{\Hscd}{\dot H_{\scl}}
\newcommand{\Hscb}{\bar H_{\scl}}
\newcommand{\ff}{{\mathrm{ff}}}
\newcommand{\inter}{{\mathrm{int}}}
\newcommand{\Sym}{\mathrm{Sym}}
\newcommand{\be}[1]{\begin{equation}\label{#1}}
\newcommand{\ee}{\end{equation}}

\newcommand{\foliation}{x}
\newcommand{\loccoord}{y}
\newcommand{\Foliation}{\mathsf{X}}
\newcommand{\Loccoord}{\mathsf{Y}}

\setcounter{secnumdepth}{3}
\newtheorem{lemma}{Lemma}[section]
\newtheorem{prop}[lemma]{Proposition}
\newtheorem{thm}[lemma]{Theorem}
\newtheorem{cor}[lemma]{Corollary}
\newtheorem{result}[lemma]{Result}
\newtheorem*{thm*}{Theorem}
\newtheorem*{prop*}{Proposition}
\newtheorem*{cor*}{Corollary}
\newtheorem*{conj*}{Conjecture}
\numberwithin{equation}{section}
\theoremstyle{remark}
\newtheorem{rem}[lemma]{Remark}
\newtheorem*{rem*}{Remark}
\theoremstyle{definition}
\newtheorem{Def}[lemma]{Definition}
\newtheorem*{Def*}{Definition}

\newcommand{\mar}[1]{{\marginpar{\sffamily{\scriptsize #1}}}}
\newcommand\av[1]{\mar{AV:#1}} 
\newcommand\gu[1]{\mar{GU:#1}} 
\newcommand\ps[1]{\mar{PS:#1}} 

\title{Inverting the local geodesic X-ray transform on tensors}
\author[Plamen Stefanov, Gunther Uhlmann and Andras Vasy]{Plamen
  Stefanov, Gunther Uhlmann and Andr\'as Vasy}
\date{October 18, 2014}
\address{Department of Mathematics, Purdue University, West Lafayette,
IN 47907-1395, U.S.A.}
\email{stefanov@math.purdue.edu}
\address{Department of Mathematics, University of Washington, 
Seattle, WA 98195-4350, U.S.A.}
\email{gunther@math.washington.edu}
\address{Department of Mathematics, Stanford University, Stanford, CA
94305-2125, U.S.A.}
\email{andras@math.stanford.edu}
\thanks{The authors were partially supported by the National Science Foundation under
grant DMS-1301646 (P.S.), CMG-1025259 (G.U.\ and A.V.) and DMS-1265958 (G.U.) and
DMS-1068742 and DMS-1361432 (A.V.).}
\subjclass{53C65, 35R30, 35S05, 53C21}

\begin{abstract}
We prove the local invertibility, up to potential fields,  and stability of the geodesic X-ray transform on tensor fields of order $1$ and $2$ near a strictly convex boundary point, on manifolds with boundary of dimension $n\ge3$. We also present an inversion formula. Under the condition that the manifold can be foliated with a continuous family of strictly convex surfaces, we prove a global result which also implies a lens rigidity result near such a metric.  The class of manifolds  satisfying the foliation condition includes manifolds with no focal points, and does not exclude existence of conjugate points. 
\end{abstract} 

\maketitle

\section{Introduction}
Let $(M,g)$ be a compact Riemannian  manifold with boundary. The X-ray transform of symmetric covector fields of order $m$ is given by 
\be{1}
If(\gamma) = \int \langle f(\gamma(t)), \dot\gamma^m(t) \rangle \, d t,
\ee
where, in local coordinates, $\langle f, v^m\rangle  =f_{i_1\dots i_m} v^{i_1}\dots v^{i_m}$, and $\gamma$ runs over all  (finite length) geodesics with endpoints on $\partial M$. When $m=0$, we integrate functions; when $m=1$, $f$ is a covector field, in local coordinates,  $f_j dx^j$; when $m=2$, $f$ is a symmetric 2-tensor field $f_{ij}dx^i dx^j$, etc. The problem is of interest by itself but it also appears as a linearization  of boundary and lens rigidity problems, see, e.g., \cite{Sh-UW,Sh-book, SU-Duke, SU-JAMS, SU-lens,Croke04b,CrokeH02,Croke_scatteringrigidity}. Indeed, when $m=0$, $f$ can be interpreted as  the infinitesimal difference of two conformal factors, and when $m=2$, $f_{ij}$ can be thought of as an infinitesimal difference of two metrics.  The $m=1$ problem arises as a linearization of recovery a velocity fields from the time of fly. The $m=4$ problem appears in linearized elasticity. 

The problem we study is the invertibility of $I$. It is well known that   \textit{potential} vector fields,  i.e., $f$ which are a symmetric differential $d^sv$ of a symmetric field of order $m-1$ vanishing on $\partial M$ (when $m\ge1$), are in the kernel of $I$. When $m=0$, there are no potential fields; when $m=1$, potential fields are just ordinary differentials $d v$ of functions vanishing at the boundary; for $m=2$, potential fields are given by $d^sv = \frac12(v_{i,j} +v_{j,i})$, with $v$ one form, $v=0$ on $\partial M$; etc.  The natural invertibility question is then whether $If=0$ implies that $f$ is potential; we call that property \textit{s-injectivity}  below.  

This problem has been studied extensively for \textit{simple manifolds}, i.e., when $\partial M$ is strictly  convex and any two points are connected by a unique minimizing geodesic smoothly depending on the endpoints. 
For simple metrics, in case of functions ($m=0$), uniqueness and a non-sharp stability estimate was established in \cite{Mu1,Mu2,BGerver} using the energy method initiated by Mukhometov, and for $m=1$, in \cite{AnikonovR}. Sharp stability follows from \cite{SU-Duke}. 
The case $m\ge2$ is harder with less complete results and the $m=2$ one already contains all the difficulties. 
In two dimensions, uniqueness for simple metrics and $m=2$ has been proven in \cite{Sh-2D} following the boundary rigidity proof in \cite{PestovU}. For any $m$, this was done in \cite{PaternainSU_13}.

In dimensions $n\ge3$, the problem still remains open for $m\ge2$. Under an explicit upper bound of the curvature, uniqueness and a non-sharp stability was proved by Sharafutdinov, see \cite{Sh-book, Sh-UW} and the references there, using a suitable version of the  energy method developed in \cite{PestovSh}. Convexity of $\partial M$ is not essential for those kind of results and the curvature assumption can be replaced by an  assumption stronger than requiring no conjugate points,  see \cite{Sh-sibir, Dairbekov}.  This still does not answer the uniqueness question for metrics without conjugate points however. The first and the second author proved in \cite{SU-Duke, SU-JAMS}, using microlocal and analytic microlocal techniques, that for simple metrics, the problem  is Fredholm (modulo potential fields) with a finitely dimensional smooth kernel. For analytic simple metrics, there is uniqueness; and in fact, the uniqueness extends to an open and dense set of simple metrics in $C^k$, $k\gg1$. Moreover, there is a sharp stability  $L^2(M)\to H^1(\tilde M)$  estimate for $f\mapsto I^*If$, where $\tilde M$ is some extension of $M$, see \cite{S-AIP}. We study the $m=2$ case there for simplicity of the exposition but the methods extend to any $m\ge2$. 

The reason why $m\ge2$ is harder than the $m=1$ and the $m=0$ cases can be seen from the analysis in \cite{SU-Duke, SU-JAMS}. When $m=0$, the presence of the boundary $\partial M$ is not essential --- we can extend $(M,g)$ to a complete $(\tilde M,\tilde g)$ and just restrict $I$ to functions  supported in a fixed compact set. When $f$ is an one-form ($m=1$), we have to deal with non-uniqueness due to exact one-forms but then the symmetric differential is $d^s$ just the ordinary one $d$. When $n\ge2$, $d^s$ is an elliptic operator but recovery of $df$ from $d^sf$  is not a local operator.  One way to deal with the non-uniqueness due to potential fields is to project on solenoidal ones (orthogonal to the potential fields). This involves solving an elliptic boundary value problem and the presence of the boundary $\partial M$ becomes an essential factor. The standard pseudo-differential calculus is not suited naturally to work on manifolds with boundary. 

In \cite{SU-AJM}, the first two authors study manifolds with possible conjugate points of dimension $n\ge3$. The geodesic manifold (when it is a smooth manifold) has dimension $2n-2$ which exceeds $n$ when $n\ge3$. We restrict $I$ there to an open set $\Gamma$ of geodesics. Assuming that $\Gamma$ consists of geodesics without conjugate points so that the conormal bundle $\{T^*\gamma|\; \gamma\in\Gamma\}$ covers $T^*M\setminus 0$, we show uniqueness and stability for analytic metrics, and moreover for an open and dense set of such metrics. In this case, even though conjugate points are allowed, the analysis is done on the geodesics in $\Gamma$ assumed to have no such  points. 

A significant progress is done in the recent work \cite{UV:local}, where the second and the third author prove the following local result: if $\partial M$ is strictly convex at $p\in\partial M$ and $n\ge3$, then $If$, acting on functions ($m=0$), known for all geodesics close enough to the tangent ones to $\partial M$ at $p$, determine $f$ near $p$ in a stable way. The new idea in \cite{UV:local} was to introduce an artificial boundary near $p$ cutting off a small part of $M$ including $p$ and to apply the scattering calculus in the new domain $\Omega_c$, treating the artificial boundary as infinity, see Figure~\ref{fig:convex-1}. Then $\Omega_c$ is small enough, then a suitable ``filtered'' backprojection operator is not only Fredholm, but also invertible. We use this idea in the present work, as well. The authors used this linear results in a recent work \cite{SUV_localrigidity} to prove local boundary and lens rigidity near a convex boundary point. 

The purpose of this paper is to invert the geodesic X-ray transform
$f\mapsto If$ on one forms and symmetric 2-tensors ($m=1$ and $m=2$) for $n\geq 3$ near a strictly convex
boundary point. 
We give a local recovery procedure for $f$ on
suitable open sets $\Omega\subset M$ from the knowledge of
$If(\gamma)$ for $\Omega$-local geodesics $\gamma$, i.e.\ $\gamma$
contained in $\Omega$ with endpoints on $\pa M\cap\Omega$. More precisely,
there is an obstacle to the inversion explained above: one-forms or tensors which are potential, i.e.\ of the form $d^s v$,
where $v$ is scalar or a one-form, vanishing at $\pa M\cap\Omega$, have vanishing 
integrals along all the geodesics with endpoints there, so one may always add a potential (exact) form or a potential two-tensor to $f$ and obtain the same localized transform $If$. Our result is
thus the local recovery 
of $f$ from $If$ {\em up to this gauge freedom}; in a stable way. Further,
under an additional global convex foliation assumption we
also give a global counterpart to this result.

We now state our main results more concretely. Let $\rho$ be a local boundary defining function, so that $\rho\ge0$ in $M$. 
 It is convenient to also consider a manifold without
boundary $(\tilde M,g)$ extending $M$.
First, as in \cite{UV:local}, the
main local result is obtained for sufficiently small regions
$\Omega=\Omega_c=\{x\geq 0,\ \rho\geq 0\}$, $x=x_c$; see
Figure~\ref{fig:convex-1}.
Here $x=0$ is an `artificial
boundary' which is strictly concave as viewed from the region $\Omega$
between
it and the actual boundary $\pa M$; this (rather than $\pa M$) is the boundary
that plays a role in the analysis below.

We set this up in the same way as
in \cite{UV:local} by considering a function $\tilde x$ with
strictly concave level sets from the super-level set side for levels
$c$, $|c|<c_0$, and letting
$$
x_c=\tilde x+c,\ \Omega_c=\{x_c\geq 0,\ \rho\geq 0\}.
$$
(A convenient normalization is that there is a point $p\in\pa M$ such
that $\tilde x(p)=0$ and such that $d\tilde x(p)=-d\rho(p)$; then one
can take e.g.\ $\tilde x(z)=-\rho(z)-\ep|z-p|^2$ for small $\ep>0$,
which localizes in a lens shaped region near $p$, or indeed $\tilde
x=-\rho$ which only localizes near $\pa\Omega$.)
Here the
requirement on $\tilde x$ is, if we assume that $M$ is compact, that
there is a continuous function $F$ such that $F(0)=0$ and such that
$$
\Omega_c\subset\{\tilde x<-c+F(c)\},
$$
i.e.\ as $c\to 0$, $\Omega_c$ is a thinner and thinner shell in terms
of $\tilde x$. As in \cite{UV:local}, our constructions are
{\em uniform} in $c$ for $|c|<c_0$. We drop the subscript $c$ from
$\Omega_c$, i.e.\ simply write $\Omega$, again as in
\cite{UV:local}, to avoid overburdening the notation.

\begin{figure}[ht]
\includegraphics[width=60mm]{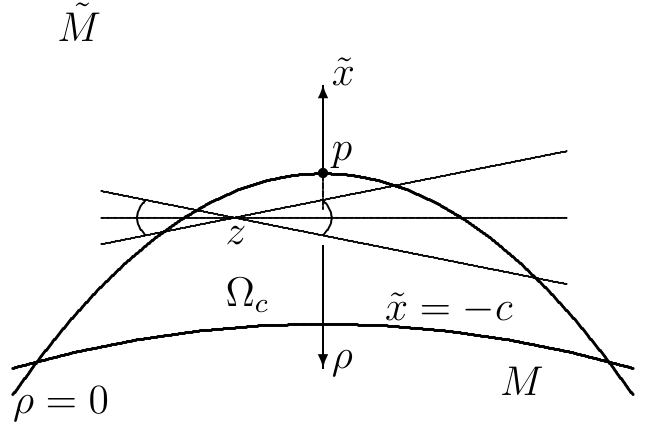}
\caption{The functions $\rho$ and $\tilde x$ when the background is
  flat space $\tilde M$. The intersection of
 $\rho\geq 0$ and $x_c> 0$ (where $x_c=\tilde x+c$, so this is the region $\tilde x>-c$) is the lens
 shaped region $O_p$. Note that, as viewed from the superlevel sets,
 thus from $O_p$, $\tilde
 x$ has concave level sets. At the point $z$, $L$ integrates over
 geodesics in the indicated small angle. As $z$ moves to the
 artificial boundary $x_c=0$, the angle of this cone shrinks like $C
 x_c$ so that in the limit the geodesics taken into account become
 tangent to $x_c=0$.}\label{fig:convex-1}
\end{figure}

A weaker version, in terms of function spaces, of the main local
theorem, presented in Corollaries~\ref{cor:local-linear-one-form}-\ref{cor:local-linear-2-tensor}, is then the following. The
notation here is that local spaces mean that the condition is
satisfied on compact subsets of $\Omega\setminus\{x=0\}$, i.e.\ the
conclusions are not stated uniformly up to the artificial boundary
(but are uniform up to the original boundary); this is due to our
efforts to minimize the analytic and geometric background in the
introduction. The dot denotes supported distributions in the sense of
H\"ormander relative to the actual boundary $\rho=0$, i.e.\ distributions in $x>0$ (within the extension
$\tilde M$) whose support lies in
$\rho\geq 0$, i.e.\ for $\dot H^1$, this is the $H^1_0$ space.

\begin{thm}\label{thm:local-linear-intro-1}(See Corollaries~\ref{cor:local-linear-one-form}-\ref{cor:local-linear-2-tensor}.)
With $\Omega=\Omega_c$ as above, there is $c_0>0$ such that for $c\in(0,c_0)$,
if $f\in L^2(\Omega)$ then $f=u+d^s v$,
where $v\in\dot H^1_{\loc}(\Omega\setminus\{x=0\})$, while $u\in
L_\loc^2(\Omega\setminus\{x=0\})$ can be stably determined from
$If$ restricted to $\Omega$-local geodesics in the following sense. 
There is a
continuous map
$If\mapsto u$,  where for $s\geq 0$, $f$ in $H^s(\Omega)$, the $H^{s-1}$ norm of $u$
restricted to any compact subset of $\Omega\setminus\{x=0\}$
is controlled by the $H^s$ norm of $If$ restricted to the set of $\Omega$-local geodesics.

Replacing $\Omega_c=\{\tilde x>-c\}\cap M$ by
$\Omega_{\tau,c}=\{\tau>\tilde x>-c+\tau\}\cap M$, $c$ can be taken
uniform in $\tau$ for $\tau$ in a compact set on which the strict
concavity assumption on level sets of $\tilde x$ holds.
\end{thm}

The uniqueness part of the theorem generalizes  Helgason's type of support theorems for tensors fields for analytic metrics \cite{Venky09,SV,BQ1}. In those works however, analyticity plays a crucial role and the proof is a form of a microlocal analytic continuation. In contrast, no analyticity is assumed here. 

As in \cite{UV:local}, this theorem can be applied in a
manner to obtain a global conclusion. To state this, assume
that $\tilde x$ is a globally defined function with level sets $\Sigma_t$
which are strictly concave from the super-level set for $t\in (-T,0]$,
with $\tilde x\leq 0$ on the manifold with boundary $M$.
Then we have:

\begin{thm}\label{thm:global-intro}(See Theorem~\ref{thm:global}.)
Suppose $M$ is compact.
Then the geodesic X-ray transform is injective and stable modulo potentials on the
restriction of one-forms and symmetric 2-tensors $f$ to $\tilde x^{-1}((-T,0])$ in
the following sense. For all $\tau>-T$ there is $v\in \dot H^1_\loc(\tilde
x^{-1}((\tau,0]))$ such that $f-d^sv\in L^2_\loc(\tilde
x^{-1}((\tau,0]))$ can be
stably recovered from $If$ in the sense that for $s\geq 0$ and $f\in
H^s$ locally on $\tilde
x^{-1}((\tau,0])$, the
$H^{s-1}$ norm of $v$ restricted to compact subsets of $\tilde
x^{-1}((\tau,0])$
is controlled by the $H^s$ norm of $If$ on local geodesics.
\end{thm}

\begin{rem}
This theorem, combined with Theorem~2 in \cite{SU-lens} (with a minor change --- the no-conjugate condition there is only needed to guarantee a stability estimate, and we have it in our situation), implies a local, in terms of a perturbation of the metric, lens rigidity uniqueness result near metric satisfying the foliation condition. 
\end{rem}

Manifolds satisfying the foliation condition include manifolds without focal points \cite{RanjanS02}. Subdomains $M$ of $\mathbf{R}^n$ with the metric $c^{-2}(r)d x^2$, $r=|x|$ satisfying the Herglotz \cite{Herglotz} and  Wiechert and Zoeppritz \cite{WZ} condition $\frac{d}{dr}\frac{r}{c(r)}>0$ on $M$ satisfy it as well since then the Euclidean spheres $|x|=r$ form a strictly convex foliation. Conjugate points in that case may exist, and small perturbations of such metrics satisfy the condition, as well. We can also formulate semi-global results: if we can foliate $M\setminus K$ with $K\subset M$ compact, then we can recover $f$ up to a potential field there in a stable way, with stability degenerating near $\partial M$.  This can be considered as a linearized model of the seismology problem for anisotropic speeds of propagation.  One such example is  metrics $c^{-2}(r)d x^2$ (and close to them) for which $\frac{d}{dr}\frac{r}{c(r)}>0$ holds for $a\le r\le b$ and $M\subset \{|x|\le b\}$. Then $f$ can be stably recovered for $|x|>a$ up to a potential field.

Similarly to our work \cite{SUV_localrigidity}, this
paper, and its methods, will have applications to the boundary
rigidity problem; in this case without the conformal class
restriction. This paper is forthcoming.

The plan of the paper is the following. In Section~\ref{sec:idea} we
sketch the idea of the proof, and state the main technical result. In
Section~\ref{sec:ellipticity} we show the ellipticity of the modified
version of $LI$, modified by the addition of gauge terms. This
essentially proves the main result {\em if} one can satisfy the gauge
condition. In Section~\ref{sec:gauge} we analyze the gauge condition
and complete the proof of our main results.

\section{The idea of the proof and the scattering algebra}\label{sec:idea}
We now explain the basic ideas of the paper.

The usual approach in dealing with the gauge freedom is to add a gauge condition,
which typically, see e.g.\ the work of the first two authors
\cite{SU-JAMS}, is of the {\em solenoidal gauge
  condition} form, $\delta^s_g f=0$, where
$\delta^s_g $ is the adjoint of $d^s$ with respect to the Riemannian
metric on $M$. Notice that actually the particular choice of the
adjoint is irrelevant; once one recovers $f$ in one gauge, one could
always express it in terms of another gauge, e.g.\ in this case
relative to a {\em different} Riemannian metric.

In order to motivate
our gauge condition, we need to recall the method introduced by the
last two authors in \cite{UV:local} to analyze the geodesic
X-ray transform on functions: the underlying analysis strongly
suggests the form the gauge condition should take.

As in \cite{UV:local} we consider an operator $L$ that
integrates over geodesics in a small cone at each point, now multiplying
with a one form or symmetric 2-tensor, in the direction of the
geodesic, mapping (locally defined) functions on the space of
geodesics to (locally defined) one forms or tensors. The choice of
the operator, or more concretely the angle, plays a big role; we
choose it to be comparable to the distance to the artificial boundary,
$x=0$. In this case $LI$ ends up being in Melrose's scattering
pseudodifferential algebra, at least once conjugated by an exponential
weight. (The effect of this weight is that we get exponentially weak
estimates as we approach the artificial boundary.) The main analytic problem
one faces then is that, corresponding to the gauge freedom mentioned above, $LI$ is not elliptic, unlike in the scalar
(function) setting.

Concretely $L$ is defined as follows. Near $\pa\Omega$, one can use
coordinates $(x,y)$, with $x=x_c=\tilde x+c$ as before, $y$
coordinates on $\pa\Omega$. Correspondingly, elements of $T_pM$ can be
written as $\lambda\,\pa_x+\eta\,\pa_y$. The unit speed geodesics which
are close to being tangential to level sets of $\tilde x$ (with the
tangential ones being given by $\lambda=0$) through a
point $p$ can be parameterized by say $(\lambda,\omega)$ (with the
actual unit speed being a positive multiple of this) where
$\omega$ is unit length with respect to say a Euclidean metric. The
concavity of the level sets of $\tilde x$, as viewed from the
super-level sets, means that $\frac{d^2}{dt^2}\tilde x\circ\gamma$ is
bounded below by a positive constant along geodesics in $\Omega_c$, as
long as $c$ is small, which in turn means that, for sufficiently small
$C_1>0$, geodesics with
$|\lambda|<C_1\sqrt{x}$ indeed remain in $x\geq 0$ (as long as they
are in $M$). Thus, if $If$ is known along $\Omega$-local geodesics, it
is known for geodesics $(x,y,\lambda,\omega)$ in this range. As in
\cite{UV:local} we use a smaller range $|\lambda|<C_2 x$ because of
analytic advantages, namely the ability work in the well-behaved
scattering algebra.
Thus,
for $\chi$ smooth, even, non-negative, of compact support, to be specified, in the function case
\cite{UV:local} considered the operator
$$
Lv(z)=x^{-2}\int \chi(\lambda/x)v(\gamma_{x,y,\lambda,\omega})\,d\lambda\,d\omega,
$$
where $v$ is a (locally, i.e.\ on $\supp\chi$, defined) function on the space of geodesics, here parameterized
by $(x,y,\lambda,\omega)$. (In fact, $L$ had a factor $x^{-1}$
only in \cite{UV:local}, with another $x^{-1}$ placed elsewhere; here we simply combine
these, as was also done in \cite[Section~3]{SUV_localrigidity}. Also,
the particular measure $d\lambda\,d\omega$ is irrelevant; any smooth positive
multiple would work equally well.)
In this paper, with $v$ still a locally defined function on the space of geodesics,
for one-forms we consider the map $L$
\begin{equation}\label{eq:L-forms}
L v(z)=\int \chi(\lambda/x)v(\gamma_{x,y,\lambda,\omega})g_{\scl}(\lambda\,\pa_x+\omega\,\pa_y)\,d\lambda\,d\omega,
\end{equation}
while for 2-tensors
\begin{equation}\label{eq:L-tensors}
L v(z)=x^{2}\int \chi(\lambda/x)v(\gamma_{x,y,\lambda,\omega})g_{\scl}(\lambda\,\pa_x+\omega\,\pa_y)\otimes g_\scl(\lambda\,\pa_x+\omega\,\pa_y)\,d\lambda\,d\omega,
\end{equation}
so in the two cases $L$ maps into one-forms, resp.\ symmetric
2-cotensors,
where $g_{\scl}$ is a {\em scattering metric} used to
convert vectors into covectors --- this is discussed in detail below.

Since it plays a crucial role even in the setup, by giving the bundles
of which our tensors are sections of, as well as the gauge
condition, we need to discuss scattering geometry and the scattering pseudodifferential
algebra, introduced by Melrose in \cite{Melrose-book}, at least briefly. There is a more thorough discussion in
\cite[Section~2]{UV:local}, though the cotangent bundle,
which is crucial here, is suppressed there. Briefly, the scattering
pseudodifferential algebra $\Psisc^{m,l}(X)$ on a manifold with
boundary $X$ is the generalization of the standard pseudodifferential
algebra given by quantizations of symbols $a\in S^{m,l}$, i.e.\ $a\in\CI(\RR^n\times\RR^n)$ satisfying
\begin{equation}\label{eq:sc-symbols}
|D_z^\alpha D_\zeta^\beta a(z,\zeta)|\leq C_{\alpha\beta}\langle z\rangle^{l-|\alpha|}\langle\zeta\rangle^{-|\beta|}
\end{equation}
for all multiindices $\alpha,\beta$ in the same way that on a compact
manifold without boundary $\tilde X$, $\Psi^m(\tilde X)$ arises from (localized)
pseudodifferential operators on $\RR^n$ via considering coordinate
charts. More precisely, $\RR^n$ can be compactified to a ball
$\overline{\RR^n}$, by gluing a sphere at infinity, with the gluing
done via `reciprocal polar coordinates'; see
\cite[Section~2]{UV:local}. One then writes
$\Psisc^{m,l}(\overline{\RR^n})$ for the quantizations of the symbols \eqref{eq:sc-symbols}. Then $\Psisc^{m,l}(X)$ is
defined by requiring that locally in coordinate charts, including
charts intersecting with $\pa X$, the algebra arises from
$\Psisc^{m,l}(\overline{\RR^n})$. (One also has to allow smooth Schwartz
kernels on $X\times X$ which are vanishing to infinite order at
$\pa(X\times X)$, in analogy with the smooth Schwartz kernels on
$\tilde X\times\tilde X$.) Thus, while the compactification is
extremely useful
to package information, the reader should keep in mind that ultimately
almost all of the analysis reduces to uniform analysis on
$\RR^n$. Since we are working with bundles, we also mention that scattering
pseudodifferential operators acting on sections of vector bundles are
defined via local trivializations, in which these operators are given
by matrices of scalar scattering pseudodifferential operators (i.e.\ are given by
the $\RR^n$ definition above if in addition these trivializations are made to be
coordinate charts), up to the same smooth, infinite order vanishing at
$\pa(X\times X)$ Schwartz kernels as in the scalar case.

Concretely, the compactification $\overline{\RR^n}$, away from
$0\in\RR^n\subset\overline{\RR^n}$, is just
$[0,\infty)_x\times\sphere^{n-1}_\omega$, where the identification with
$\RR^n\setminus\{0\}$ is just the `inverse polar coordinate' map $(x,\omega)\mapsto
x^{-1}\omega$, with $r=x^{-1}$ the standard radial variable. Then a straightforward computation shows that
translation invariant vector fields $\pa_{z_j}$ on $\RR^n_z$ lift to
the compactification (via this identification) to generate, over
$\CI(\overline{\RR^n})$, the Lie algebra $\Vsc(\overline{\RR^n})=x\Vb(\overline{\RR^n})$ of
vector fields, where on a manifold with boundary $\Vb(X)$ is the Lie
algebra of smooth vector fields tangent to the boundary of $X$. In
general, if $x$ is a boundary defining function of $X$, we let
$\Vsc(X)=x\Vb(X)$. Then $\Psisc^{1,0}(X)$ contains $\Vsc(X)$,
corresponding to the analogous inclusion on Euclidean space, and the
vector fields in $\Psisc^{1,0}(X)$ are essentially the elements of
$\Vsc(X)$, after a slight generalization of coefficients (since above
$a$ does not have an asymptotic expansion at infinity in $z$, only symbolic
estimates; the expansion would correspond to smoothness of the
coefficients).

Now, a local basis for $\Vsc(X)$, in a coordinate chart
$(x,y_1,\ldots,y_{n-1})$, is
$$
x^2\pa_x,x\pa_{y_1},\ldots,
x\pa_{y_{n-1}}
$$
directly from the definition, i.e.\ $V\in\Vsc(X)$
means exactly that locally, on $U\subset X$
$$
V=a_0 (x^2\pa_x)+\sum a_j (x\pa_{y_j}),\ a_j\in\CI(U).
$$
This gives that elements
of $\Vsc(X)$ are exactly smooth sections of a vector bundle, $\Tsc X$,
with local basis $x^2\pa_x,x\pa_{y_1},\ldots,
x\pa_{y_{n-1}}$. In the case of $X=\overline{\RR^n}$, this simply
means that one is using the local basis $x^2\pa_x=-\pa_r$,
$x\pa_{y_j}=r^{-1}\pa_{\omega_j}$,
where the $\omega_j$ are local coordinates on the sphere. An
equivalent {\em global} basis is just $\pa_{z_j}$, $j=1,\ldots,n$,
i.e.\ $\Tsc\overline{\RR^n}=\overline{\RR^n_z}\times\RR^n$ is a trivial bundle with this
identification.

The dual bundle $\Tsc^*X$ of $\Tsc X$ correspondingly has a local
basis $\frac{dx}{x^2},\frac{dy_1}{x},\ldots,\frac{dy_{n-1}}{x}$, which
in case of $X=\overline{\RR^n}$ becomes $-dr,r\,d\omega_j$, with local
coordinates $\omega_j$ on the sphere. A global version is given by
using the basis $dz_j$, with covectors written as $\sum
\zeta_j\,dz_j$; thus
$\Tsc^*\overline{\RR^n}=\overline{\RR^n_z}\times\RR^n_\zeta$; this is
exactly the same notation as in the description of the symbol class
\eqref{eq:sc-symbols}, i.e.\ one should think of this class as living
on $\Tsc^*\overline{\RR^n}$. Thus, smooth scattering one-forms on $\overline{\RR^n}$, i.e.\
sections of $\Tsc^*\overline{\RR^n}$, are simply smooth one-forms on $\RR^n$
with an expansion at infinity. Similar statements apply to natural
bundles, such as the higher degree differential forms $\Lambdasc^k X$, as well as
symmetric tensors, such as $\Sym^2\Tsc^*X$. The latter give rise to
scattering metrics $g_\scl$, which are positive definite inner products on the
fibers of $\Tsc X$ (i.e.\ positive definite sections of
$\Sym^2\Tsc^*X$) of the form $g_\scl=x^{-4}dx^2+x^{-2}\tilde h$, $\tilde h$
a standard smooth 2-cotensor on $X$ (i.e.\ a section of $\Sym^2
T^*X$). For instance, one can take, in a product decomposition near
$\pa X$, $g_\scl=x^{-4}dx^2+x^{-2}h$, $h$ a metric on the level sets
of $x$.

The principal symbol of a pseudodifferential operator is the
equivalence class of $a$ as in \eqref{eq:sc-symbols} modulo $S^{m-1,l-1}$, i.e.\ modulo additional
decay {\em both in $z$ and in $\zeta$} on $\RR^n\times\RR^n$. In
particular, {\em full ellipticity} is ellipticity in this sense,
modulo $S^{m-1,l-1}$, i.e.\ for a scalar operator lower bounds
$|a(z,\zeta)|\geq c\langle z\rangle^l\langle\zeta\rangle^m$ for
$|z|+|\zeta|>R$, where $R$ is suitably large. This contrasts with (uniform)
ellipticity in the standard sense, which is a similar lower bound, but
only for $|\zeta|>R$. Fully elliptic operators are Fredholm between
the appropriate
Sobolev spaces $\Hsc^{s,r}(X)$ corresponding to the scattering
structure, see \cite[Section~2]{UV:local}; full ellipticity
is needed for this (as shown e.g.\ by taking $\Delta-1$ on $\RR^n$,
$\Delta$ the flat positive Laplacian). If $a$ is matrix valued,
ellipticity can be stated as invertibility for large $(z,\zeta)$,
together with upper bounds for the inverse: $|a(z,\zeta)^{-1}|\leq
c^{-1}\langle z\rangle^{-l}\langle\zeta\rangle^{-m}$; this coincides
with the above definition for scalars.

We mention also that the exterior derivative
$d\in\Diffsc^1(X;\Lambdasc^k,\Lambdasc^{k+1})$ for all $k$. Explicitly,
for $k=0$, in local coordinates, this is the statement that
$$
df=(\pa_x f)\,dx+\sum_j (\pa_{y_j} f)\,dy_j=(x^2\pa_x
f)\,\frac{dx}{x^2}+\sum_j (x\pa_{y_j})\frac{dy_j}{x},
$$
with $x^2\pa_x,x\pa_{y_j}\in\Diffsc^1(X)$, while
$\frac{dx}{x^2},\frac{dy_j}{x}$ are smooth sections of $\Tsc^*X$
(locally, where this formula makes sense). Such a computation also
shows that the principal symbol, in both senses, of $d$, at any point $\xi\,\frac{dx}{x^2}+\sum_j\eta_j\,\frac{dy_j}{x}$,
is wedge product with $\xi\,\frac{dx}{x^2}+\sum_j\eta_j\,\frac{dy_j}{x}$. A similar computation shows
that the gradient with respect to a scattering metric $g_{\scl}$ is a
scattering differential operator (on any of the natural bundles), with
principal symbol given by tensor product with $\xi\,\frac{dx}{x^2}+\sum_j\eta_j\,\frac{dy_j}{x}$,
hence so is the symmetric gradient on one forms, with principal symbol
given by the symmetrized tensor product with $\xi\,\frac{dx}{x^2}+\sum_j\eta_j\,\frac{dy_j}{x}$.
Note that all of these principal symbols are actually {\em 
  independent of the metric $g_{\scl}$}, and $d$ itself is completely
independent of any choice of a metric (scattering or otherwise).

If we instead consider the symmetric differential $d^s$ with respect to a
smooth metric $g$ on $X$, as we are obliged to use in our problem
since its image is what is annihilated by the ($g$-geodesic) X-ray transform $I$, it is a first order differential operator between
sections of bundles $T^*X$ and $\Sym^2 T^*X$. Writing $dx,dy_j$,
resp., $dx^2$, $dx\,dy_j$ and $dy_i\,dy_j$ for the corresponding
bases, this means that we have a matrix of first order differential
operators. Now, as the standard principal symbol of $d^s$ is just tensoring with
the covector at which the principal symbol is evaluated, the first
order terms are the same, modulo zeroth order terms, as when one
considers $d^s_{g_{\scl}}$, and in particular they correspond to a
scattering differential operator acting between section of $\Tsc^*X$
and $\Sym^2\Tsc^*X$. (This can also be checked explicitly using the
calculation done below for zeroth order term, but the above is the
conceptual reason for this.) On the other hand, with $dx^2=dx\otimes
dx$, $dx\,dy_i=\frac{1}{2}(dx\otimes dy_i+dy_i\otimes dx)$, etc., these zeroth order terms form
a matrix with smooth coefficients in the local basis
$$
dx^2\otimes\pa_x,\ dx^2\otimes\pa_{y_j},\ dx\,dy_i\otimes\pa_x,\
dx\,dy_i\otimes\pa_{y_j},\ dy_k\,dy_i\otimes\pa_x,\
dy_k\,dy_i\otimes\pa_{y_j}
$$
of the homomorphism bundle $\hom(T^*X,\Sym^2 T^*X)$. In terms of the local
basis
\begin{equation*}\begin{aligned}
&\frac{dx^2}{x^4}\otimes(x^2\pa_x),\
\frac{dx^2}{x^4}\otimes(x\pa_{y_j}),\
\frac{dx}{x^2}\,\frac{dy_i}{x}\otimes(x^2\pa_x),\\
&\
\frac{dx}{x^2}\,\frac{dy_i}{x}\otimes(x\pa_{y_j}),\ \frac{dy_k\,dy_i}{x^2}\otimes(x^2\pa_x),\
\frac{dy_k\,dy_i}{x^2}\otimes(x\pa_{y_j})
\end{aligned}\end{equation*}
of $\hom(\Tsc^*X,\Sym^2 \Tsc^*X)$, these are all smooth, and vanish at
$\pa X$ to order $2,3,1,2,0,1$ respectively, showing that
$d^s\in\Diffsc^1(X,\Tsc^*X,\Sym^2\Tsc^*X)$, and that the only
non-trivial contribution of these zeroth order terms to the principal
symbol is via the entry corresponding to
$\frac{dy_k\,dy_i}{x^2}\otimes(x^2\pa_x)=dy_k\,dy_i\otimes\pa_x$,
which however is rather arbitrary.

Returning to the choice of gauge, in our case the
solenoidal gauge relative to $g$ would not be a good idea: the metric
on $M$ is an incomplete metric as viewed at the artificial boundary,
and does {\em not} interact well with $LI$. We circumvent this
difficulty by considering instead the adjoint $\delta^s$ relative to a
scattering metric, i.e.\ one of the form $x^{-4}dx^2+x^{-2}h$, $h$ a
metric on the level sets of $x$. While $\delta^s,d^s$ are then
scattering differential operators, unfortunately $\delta^s d^s$ on
functions, or one forms, is not fully elliptic in the scattering
sense (full ellipticity is needed to guarantee Fredholm properties on
Sobolev spaces in
a compact setting), with the problem being at finite points of $\Tsc^*X$,
$X=\{x\geq 0\}$. For instance, in the case of $X$ being the radial compactification
of $\RR^n$, we would be trying to invert the Laplacian on functions or
one-forms, which has issues at the 0-section. However, if we instead
use an exponential weight, which already arose when $LI$ was
discussed, we can make the resulting operator fully elliptic, and
indeed invertible for suitable weights.

Thus, we introduce a Witten-type (in the sense of the Witten Laplacian) solenoidal gauge on the scattering
cotangent bundle, $\Tsc^*X$
or its second symmetric power, $\Sym^2\Tsc^*X$. Fixing $\digamma>0$, our gauge is
$$
e^{2\digamma /x}\delta^s e^{-2\digamma/x}f^s=0,
$$
or {\em the $e^{-2\digamma/x}$-solenoidal gauge}. (Keep in mind here
that $\delta^s$ is the adjoint of $d^s$ relative to a scattering metric.)
We are actually working
with
$$
f_\digamma=e^{-\digamma/x} f
$$
throughout; in terms of this the
gauge is
$$
\delta^s_\digamma f^s_\digamma=0,\qquad \delta^s_\digamma=e^{\digamma/x}\delta^s e^{-\digamma/x}.
$$

\begin{thm}\label{thm:local-linear-intro}(See
  Theorem~\ref{thm:local-linear} for the proof and the formula.)
There exists $\digamma_0>0$ such that for $\digamma\geq\digamma_0$ the
following holds.

For $\Omega=\Omega_c$, $c>0$ small, the geodesic X-ray transform on
{\em $e^{2\digamma/x}$-solenoidal}
one-forms and symmetric 2-tensors $f\in e^{\digamma/x} L_\scl^2(\Omega)$,
i.e.\ ones satisfying $\delta^s (e^{-2\digamma/x} f)=0$, is injective, with a stability
estimate and a reconstruction formula.

In addition, replacing $\Omega_c=\{\tilde x>-c\}\cap M$ by
$\Omega_{\tau,c}=\{\tau>\tilde x>-c+\tau\}\cap M$, $c$ can be taken
uniform in $\tau$ for $\tau$ in a compact set on which the strict
concavity assumption on level sets of $\tilde x$ holds.
\end{thm}

\section{Ellipticity up to gauge}\label{sec:ellipticity}

With $L$ defined in \eqref{eq:L-forms}-\eqref{eq:L-tensors}, the main analytic points are that, first, $LI$ is (after a suitable
exponential conjugation) a scattering
pseudodifferential operator of order $-1$, and second, by choosing an additional appropriate
gauge-related summand, this operator $LI$ is elliptic (again, after the
exponential conjugation). These results are stated in the next two
propositions, with the intermediate
Lemma~\ref{lemma:potential-complement} describing the gauge related summand.

\begin{prop}\label{prop:psdo}
On one forms, resp.\ symmetric 2-cotensors, the operators
$N_{\digamma}=e^{-\digamma/x}L Ie^{\digamma/x}$, lie in
$$
\Psisc^{-1,0}(X;\Tsc^*X,\Tsc^*X),\ \text{resp.}\
\Psisc^{-1,0}(X;\Sym^2\Tsc^*X,\Sym^2\Tsc^*X),
$$
for $\digamma>0$.
\end{prop}

\begin{proof}
The proof of this proposition follows that of the scalar case given in
\cite[Proposition~3.3]{UV:local} and in a modified version of the
scalar case in \cite[Proposition~3.2]{SUV_localrigidity}. For
convenience of the reader, we follow the latter proof very closely,
except that we do not emphasize the continuity statements in terms of
the underlying metric itself,
indicating the modifications.

Thus, recall that the
map
\begin{equation}\label{eq:Gamma+-def}
\Gamma_+:S\tilde M\times[0,\infty)\to[\tilde M\times\tilde M;\diag],\
\Gamma_+(\foliation,\loccoord,\lambda,\omega,t)=((\foliation,\loccoord),\gamma_{\foliation,\loccoord,\lambda,\omega}(t))
\end{equation}
is a local diffeomorphism, and similarly for $\Gamma_-$ in which
$(-\infty,0]$ takes the place of $[0,\infty)$; see the discussion
around \cite[Equation~(3.2)-(3.3)]{UV:local}; indeed this is true for
more general curve families.  Here $[\tilde
M\times\tilde M;\diag]$ is the {\em blow-up} of $\tilde M$ at the
diagonal $z=z'$, which essentially means the introduction of spherical/polar
coordinates, or often more conveniently projective coordinates, about
it. Concretely, writing the (local) coordinates from the two factors of
$\tilde M$ as $(z,z')$,
\begin{equation}\label{eq:blowup-coords}
z,|z-z'|,\frac{z-z'}{|z-z'|}
\end{equation}
give (local) coordinates on this space. Since the statement regarding
the pseudodifferential property of $LI$ is standard away from $x=0$,
we concentrate on the latter region.
Correspondingly, in our coordinates $(\foliation,\loccoord,\lambda,\omega)$,
we write
$$
(\gamma_{\foliation,\loccoord,\lambda,\omega}(t),\gamma'_{\foliation,\loccoord,\lambda,\omega}(t))
=(\Foliation_{\foliation,\loccoord,\lambda,\omega}(t),\Loccoord_{\foliation,\loccoord,\lambda,\omega}(t),\Lambda^\flat_{\foliation,\loccoord,\lambda,\omega}(t),\Omega^\flat_{\foliation,\loccoord,\lambda,\omega}(t))
$$
for
the lifted geodesic $\gamma_{\foliation,\loccoord,\lambda,\omega}(t)$.

Recall from \cite[Section~2]{UV:local} that coordinates on Melrose's
scattering double space, on which the Schwartz kernels of elements of
$\Psisc^{s,r}(X)$ are conormal to the diagonal, near the
lifted scattering diagonal, are (with $\foliation\geq 0$)
$$
\foliation,\ \loccoord,\
X=\frac{\foliation'-\foliation}{\foliation^2},\ Y=\frac{\loccoord'-\loccoord}{\foliation}.
$$
Note that here $X,Y$ are as in \cite{UV:local} around Equation~(3.10),
not as in \cite[Section~2]{UV:local} (where the signs are switched), which means that we need to
replace $(\xi,\eta)$ by $(-\xi,-\eta)$ in the Fourier transform when
computing principal symbols.
Further, it is convenient to write coordinates on $[\tilde M\times
\tilde M;\diag]$ in the region of interest (see the beginning of the paragraph of
Equation~(3.10) in \cite{UV:local}), namely (the lift of)
$|\foliation-\foliation'|<C|\loccoord-\loccoord'|$, as
$$
\foliation,\loccoord,|\loccoord-\loccoord'|,\frac{\foliation'-\foliation}{|\loccoord-\loccoord'|},\frac{\loccoord'-\loccoord}{|\loccoord-\loccoord'|},
$$
with the norms being Euclidean norms,
instead of \eqref{eq:blowup-coords}; we write $\Gamma_\pm$ in terms of
these. Note that these are
$\foliation,\loccoord,\foliation|Y|,\frac{\foliation X}{|Y|},\hat
Y$.
Moreover, by
\cite[Equation(3.10)]{UV:local} and the subsequent equations, combined
also with Equations~(3.14)-(3.15) there, $\lambda,\omega,t$ are given in
terms of $x,x',y,y'$ as
\begin{equation*}\begin{aligned}
&(\Lambda\circ\Gamma_\pm^{-1})\Big(\foliation,\loccoord,\foliation|Y|,\frac{\foliation
  X}{|Y|},\hat Y\Big)\\
&\qquad=
\foliation\frac{X-\alpha(x,y,x|Y|,\frac{xX}{|Y|},\hat Y))|Y|^2}{|Y|}+x^2\tilde\Lambda_\pm\Big(\foliation,\loccoord,\foliation|Y|,\frac{\foliation
  X}{|Y|},\hat Y\Big)
\end{aligned}\end{equation*}
with $\tilde\Lambda_\pm$ smooth,
$$
(\Omega\circ\Gamma_\pm^{-1}) \Big(\foliation,\loccoord,\foliation|Y|,\frac{\foliation  
  X}{|Y|},\hat Y\Big)=\hat Y+\foliation|Y|\tilde\Omega_\pm\Big(\foliation,\loccoord,\foliation|Y|,\frac{\foliation  
  X}{|Y|},\hat Y\Big)  
$$ 
with $\tilde\Omega_\pm$ smooth and
$$
\pm (T\circ\Gamma_\pm^{-1}) \Big(\foliation,\loccoord,\foliation|Y|,\frac{\foliation 
  X}{|Y|},\hat Y\Big)=x|Y|+\foliation^2|Y|^2\tilde T_\pm\Big(\foliation,\loccoord,\foliation|Y|,\frac{\foliation 
  X}{|Y|},\hat Y\Big) 
$$
with $\tilde T$ smooth.

In particular,
\begin{equation*}\begin{aligned}
&(\Lambda\circ\Gamma_\pm^{-1})\,\pa_x+(\Omega\circ\Gamma_\pm^{-1})\,\pa_y\\
&\qquad=\Big(\foliation\frac{X-\alpha(x,y,x|Y|,\frac{xX}{|Y|},\hat Y))|Y|^2}{|Y|}+\foliation^2\tilde\Lambda_\pm\Big(\foliation,\loccoord,\foliation|Y|,\frac{\foliation
  X}{|Y|},\hat Y\Big)\Big)\,\pa_x\\
&\qquad\qquad+\Big(\hat Y+\foliation|Y|\tilde\Omega_\pm\Big(\foliation,\loccoord,\foliation|Y|,\frac{\foliation
  X}{|Y|},\hat Y\Big)\Big)\,\pa_y.
\end{aligned}\end{equation*}
Thus, a smooth metric $g_0=dx^2+h$ applied to this yields
\begin{equation}\begin{aligned}\label{eq:g0-of-gamma-prime}
&(\Lambda\circ\Gamma_\pm^{-1})\,dx+(\Omega\circ\Gamma_\pm^{-1})\,h(\pa_y)=x\Big(x(\Lambda\circ\Gamma_\pm^{-1})\,\frac{dx}{x^2}+(\Omega\circ\Gamma_\pm^{-1})\,\frac{h(\pa_y)}{x}\Big)\\
&=x\Big(x^2\Big(\frac{X-\alpha(x,y,x|Y|,\frac{xX}{|Y|},\hat Y))|Y|^2}{|Y|}+x\tilde\Lambda_\pm\Big(\foliation,\loccoord,\foliation|Y|,\frac{\foliation
  X}{|Y|},\hat Y\Big)\Big)\,\frac{dx}{x^2}\\
&\qquad\qquad+\Big(\hat Y+\foliation|Y|\tilde\Omega_\pm\Big(\foliation,\loccoord,\foliation|Y|,\frac{\foliation
  X}{|Y|},\hat Y\Big)\Big)\,\frac{h(\pa_y)}{x}\Big),
\end{aligned}\end{equation}
while
so $g_{\scl}$ applied to this yields
\begin{equation}\begin{aligned}\label{eq:gsc-of-gamma-prime}
&x^{-1}\Big(x^{-1}(\Lambda\circ\Gamma_\pm^{-1})\,\frac{dx}{x^2}+(\Omega\circ\Gamma_\pm^{-1})\,\frac{h(\pa_y)}{x}\Big)\\
&=x^{-1}\Big(\Big(\frac{X-\alpha(x,y,x|Y|,\frac{xX}{|Y|},\hat Y))|Y|^2}{|Y|}+x\tilde\Lambda_\pm\Big(\foliation,\loccoord,\foliation|Y|,\frac{\foliation
  X}{|Y|},\hat Y\Big)\Big)\,\frac{dx}{x^2}\\
&\qquad\qquad+\Big(\hat Y+\foliation|Y|\tilde\Omega_\pm\Big(\foliation,\loccoord,\foliation|Y|,\frac{\foliation
  X}{|Y|},\hat Y\Big)\Big)\,\frac{h(\pa_y)}{x}\Big).
\end{aligned}\end{equation}
Notice that on the right hand side of \eqref{eq:gsc-of-gamma-prime} the singular factor of $x^{-1}$ in
front of $\frac{dx}{x^2}$ disappears due to the factor $x$ in
$\Lambda$, while on the right hand side of
\eqref{eq:g0-of-gamma-prime} correspondingly $\frac{dx}{x^2}$ has a
vanishing factor $x^2$. This means, as we see below, that the $\frac{dx}{x^2}$ component
behaves trivially at the level of the boundary principal symbol of the
operator $N_{\digamma,0}$ defined like $N_\digamma$ but with $g_0$ in
place of $g_{\scl}$, so in
fact one can never have full ellipticity in this case; this is the
reason we must use $g_{\scl}$ in the definition of $N_{\digamma}$.

One also needs to have $\Lambda^\flat_{x,y,\lambda,\omega}(t),\Omega^\flat_{x,y,\lambda,\omega}(t)$
evaluated at $(x',y')$, since this is the tangent vector
$\lambda'\pa{x'}+\omega'\pa_{y'}$ with which
our tensors are contracted as they are being integrated along the geodesic. In order to compute this efficiently, we
recall from \cite[Equation~(3.14)]{UV:local} that
$$
x'=x+\lambda t+\alpha(x,y,\lambda,\omega)t^2+O(t^3),\ y'=y+\omega t+O(t^2),
$$
with the $O(t^3)$, resp.\ $O(t^2)$ terms having smooth coefficients in
terms of $(x,y,\lambda,\omega)$. Correspondingly,
$$
\lambda'=\frac{dx}{dt}=\lambda+2\alpha(x,y,\lambda,\omega)t+O(t^2),\
\omega'=\frac{dy}{dt}=\omega+O(t).
$$
This gives that in terms of $x,y,x',y'$, $\lambda'$ is given by
$$
\Lambda'\circ\Gamma_\pm^{-1}=\Lambda\circ\Gamma_\pm^{-1}+2\alpha(x,y,\Lambda\circ\Gamma_\pm^{-1},\Omega\circ\Gamma_\pm^{-1})(T\circ\Gamma_\pm^{-1})+(T\circ\Gamma_\pm^{-1})^2\tilde\Lambda'\circ\Gamma_\pm^{-1},
$$
with $\tilde\Lambda'$ smooth in terms of
$x,y,\Lambda\circ\Gamma_\pm^{-1},\Omega\circ\Gamma_\pm^{-1},T\circ\Gamma_\pm^{-1}$. Substituting
these in yields
\begin{equation*}\begin{aligned}
\Lambda'\circ\Gamma_\pm^{-1}&=x\frac{X-\alpha(x,y,x|Y|,\frac{xX}{|Y|},\hat
  Y)|Y|^2}{|Y|}+2 x|Y|\alpha(x,y,x|Y|,\frac{xX}{|Y|},\hat
  Y)\\
&\qquad\qquad\qquad\qquad+x^2|Y|^2\tilde\Lambda'(x,y,x|Y|,\frac{xX}{|Y|},\hat
  Y)\\
&=x\frac{X+\alpha(x,y,x|Y|,\frac{xX}{|Y|},\hat
  Y)|Y|^2}{|Y|}+x^2|Y|^2\tilde\Lambda'(x,y,x|Y|,\frac{xX}{|Y|},\hat
  Y)
\end{aligned}\end{equation*}
while
$$
\Omega'\circ\Gamma_\pm^{-1}=\hat Y+x\tilde\Omega'(x,y,x|Y|,\frac{xX}{|Y|},\hat
  Y).
$$
Correspondingly,
\begin{equation*}\begin{aligned}
&(\Lambda'\circ\Gamma_\pm^{-1})\,\pa_x+(\Omega'\circ\Gamma_\pm^{-1})\,\pa_y=
x^{-1}\Big(x^{-1}(\Lambda'\circ\Gamma_\pm^{-1})\,x^2\pa_x+(\Omega'\circ\Gamma_\pm^{-1})\,x\pa_y\Big)\\
&\qquad=
x^{-1}\Big(\Big(\frac{X+\alpha(x,y,x|Y|,\frac{xX}{|Y|},\hat
  Y)|Y|^2}{|Y|}+\foliation|Y|^2\tilde\Lambda'_\pm\Big(\foliation,\loccoord,\foliation|Y|,\frac{\foliation
  X}{|Y|},\hat Y\Big)\Big)\,x^2\pa_x\\
&\qquad\qquad+\Big(\hat Y+\foliation|Y|\tilde\Omega'_\pm\Big(\foliation,\loccoord,\foliation|Y|,\frac{\foliation
  X}{|Y|},\hat Y\Big)\Big)\,x\pa_y\Big).
\end{aligned}\end{equation*}

Then, similarly, near the boundary as
in \cite[Equation~(3.13)]{UV:local}, one obtains the Schwartz kernel
of $N_{\digamma}$ on one forms:
\begin{equation}\begin{aligned}\label{eq:sc-SK-near-ff}
&K^\flat(\foliation,\loccoord,X,Y)\\
&=\sum_\pm e^{-\digamma
  X/(1+\foliation X)}\chi\Big(\frac{X-\alpha(x,y,x|Y|,\frac{xX}{|Y|},\hat
  Y)|Y|^2}{|Y|}+x\tilde\Lambda_\pm\Big(\foliation,\loccoord,\foliation|Y|,\frac{\foliation|X|}{|Y|},\hat
Y\Big)\Big)\\
&\qquad\qquad\Big(x^{-1}(\Lambda\circ\Gamma_\pm^{-1})\,\frac{dx}{x^2}+(\Omega\circ\Gamma_\pm^{-1})\,\frac{h(\pa_y)}{x}\Big)
\Big(x^{-1}(\Lambda'\circ\Gamma_\pm^{-1})\,x^2\pa_x+(\Omega'\circ\Gamma_\pm^{-1})\,x\pa_y\Big)\\
&\qquad\qquad\qquad|Y|^{-n+1}J_\pm\Big(\foliation,\loccoord,\frac{X}{|Y|},|Y|,\hat Y\Big),
\end{aligned}\end{equation}
with the density factor $J$ smooth, positive, $=1$ at
$\foliation=0$; there is a similar formula for 2-tensors. Note that
the factor $x^{-1}$ in \eqref{eq:gsc-of-gamma-prime}, as well as
another $x^{-1}$ from writing
$$
(\Lambda'\circ\Gamma_\pm^{-1})\,\pa_x+(\Omega'\circ\Gamma_\pm^{-1})\,\pa_y =x^{-1}\Big(x^{-1}(\Lambda'\circ\Gamma_\pm^{-1})\,x^2\pa_x+(\Omega'\circ\Gamma_\pm^{-1})\,x\pa_y\Big)
$$
are absorbed into
the definition of $L$, \eqref{eq:L-forms}-\eqref{eq:L-tensors}, hence
the different powers ($-2$ for functions, $0$ on one-forms, $2$ for
2-cotensors) appearing there.
Here
$$
\foliation,\loccoord,|Y|,\frac{X}{|Y|},\hat Y
$$
are valid
coordinates on the blow-up of the scattering diagonal in $|Y|>\ep|X|$,
$\ep>0$, which is the case automatically on the support of the kernel
due to the argument of $\chi$, cf.\ the discussion after \cite[Equation(3.12)]{UV:local}, so the argument of
$\chi$ is smooth {\em on this blown up
  space}. In addition, due to the order $x$ vanishing of $\Lambda$,
$$
x^{-1}(\Lambda\circ\Gamma_\pm^{-1})\,\frac{dx}{x^2}+(\Omega\circ\Gamma_\pm^{-1})\,\frac{h(\pa_y)}{x},\ \text{resp.}
\
x^{-1}(\Lambda\circ\Gamma_\pm^{-1})\,x^2\pa_x+(\Omega\circ\Gamma_\pm^{-1})\,x\pa_y
$$
are smooth sections of $\Tsc^*X$, resp.\ $\Tsc X$, pulled back from
the left, resp.\ right, factor of $X^2$, thus their product defines a
smooth section of the endomorphism bundle of $\Tsc^*X$.

Since this 
homomorphism factor is the only difference from 
\cite[Proposition~3.3]{UV:local}, and we have shown its smoothness 
properties as a bundle endomorphism,
this proves the 
proposition as in \cite[Proposition~3.3]{UV:local}.

If we defined $N_{\digamma,0}$ as $N_\digamma$ but using a smooth
metric $g_0$ in place of $g_{\scl}$, we would have the Schwartz kernel
\begin{equation}\begin{aligned}\label{eq:0-SK-near-ff}
&K^\flat_0(\foliation,\loccoord,X,Y)\\
&=\sum_\pm e^{-\digamma
  X/(1+\foliation X)}\chi\Big(\frac{X-\alpha(x,y,x|Y|,\frac{xX}{|Y|},\hat
  Y)|Y|^2}{|Y|}+x\tilde\Lambda_\pm\Big(\foliation,\loccoord,\foliation|Y|,\frac{\foliation|X|}{|Y|},\hat
Y\Big)\Big)\\
&\qquad\qquad\Big(x(\Lambda\circ\Gamma_\pm^{-1})\,\frac{dx}{x^2}+(\Omega\circ\Gamma_\pm^{-1})\,\frac{h(\pa_y)}{x}\Big)
\Big(x^{-1}(\Lambda'\circ\Gamma_\pm^{-1})\,x^2\pa_x+(\Omega'\circ\Gamma_\pm^{-1})\,x\pa_y\Big)\\
&\qquad\qquad\qquad|Y|^{-n+1}J_\pm\Big(\foliation,\loccoord,\frac{X}{|Y|},|Y|,\hat Y\Big),
\end{aligned}\end{equation}
and
$$
x(\Lambda\circ\Gamma_\pm^{-1})\,\frac{dx}{x^2}+(\Omega\circ\Gamma_\pm^{-1})\,\frac{h(\pa_y)}{x},\ \text{resp.}
\
x^{-1}(\Lambda\circ\Gamma_\pm^{-1})\,x^2\pa_x+(\Omega\circ\Gamma_\pm^{-1})\,x\pa_y
$$
are again smooth sections of $\Tsc^*X$, resp.\ $\Tsc X$, pulled back from
the left, resp.\ right, factor of $X^2$, but, as pointed out earlier,
with the coefficient of $\frac{dx}{x^2}$ vanishing, thus eliminating
the possibility of ellipticity in this case.
\end{proof}

Before proceeding, we compute the principal symbol of the gauge term
$d^s_\digamma\delta^s_\digamma$. For this recall that
$d^s_\digamma=e^{-\digamma/x}d^s e^{\digamma/x}$, with $d^s$ defined
using the background metric $g$ (on one forms; the metric is irrelevant for functions), and
$\delta^s_\digamma$ is its adjoint with respect to the scattering
metric $g_\scl$ (not $g$). In order to give the principal symbols, we
use the basis
$$
\frac{dx}{x^2},\ \frac{dy}{x}
$$
for one forms, with $\frac{dy}{x}$ understood as a short hand for
$\frac{dy_1}{x},\ldots,\frac{dy_{n-1}}{x}$,
while for 2-tensors, we use a decomposition
$$
\frac{dx}{x^2}\otimes\frac{dx}{x^2},\ \frac{dx}{x^2}\otimes\,\frac{dy}{x},\
\frac{dy}{x}\otimes\frac{dx}{x^2},\ \frac{dy}{x}\otimes \frac{dy}{x}.
$$
Note that symmetry of a 2-tensor is the statement that the 2nd
and 3rd (block) entries are the same (up to the standard identification), so for symmetric 2-tensors we
can also
use
$$
\frac{dx}{x^2}\otimes_s\frac{dx}{x^2},\ \frac{dx}{x^2}\otimes_s\,\frac{dy}{x},\ \frac{dy}{x}\otimes_s \frac{dy}{x},
$$
where the middle component is the common
$\frac{dx}{x^2}\otimes\,\frac{dy}{x}$ and
$\frac{dy}{x}\otimes\frac{dx}{x^2}$ component.

\begin{lemma}\label{lemma:potential-complement}
On one forms, the operator $d^s_\digamma\delta^s_\digamma\in\Diffsc^{2,0}(X;\Tsc^*X,\Tsc^*X)$ has principal symbol
$$
\begin{pmatrix}\xi+i\digamma\\\eta\otimes\end{pmatrix}\begin{pmatrix}\xi-i\digamma&\iota_\eta\end{pmatrix}=
\begin{pmatrix}\xi^2+\digamma^2&(\xi+i\digamma)\iota_\eta\\(\xi-i\digamma)\eta\otimes&\eta\otimes\iota_\eta\end{pmatrix}.
$$

On the other hand, on symmetric 2-tensors
$d^s_\digamma\delta^s_\digamma\in\Diffsc^{2,0}(X;\Sym^2\Tsc^*X,\Sym^2\Tsc^*X)$ has
principal symbol
\begin{equation*}\begin{aligned}
&\begin{pmatrix} \xi+i\digamma&0\\\frac{1}{2}\eta\otimes&\frac{1}{2}(\xi+i\digamma)\\a&\eta\otimes_s\end{pmatrix}
\begin{pmatrix}
\xi-i\digamma&\frac{1}{2}\iota_\eta&\langle
a,.\rangle\\0&\frac{1}{2}(\xi-i\digamma)&\iota^s_\eta\end{pmatrix}\\
&=
\begin{pmatrix}\xi^2+\digamma^2&\frac{1}{2}(\xi+i\digamma)\iota_\eta&(\xi+i\digamma)\langle
  a,.\rangle\\
\frac{1}{2}(\xi-i\digamma)\eta\otimes&\frac{1}{4}(\eta\otimes)\iota_\eta+\frac{1}{4}(\xi^2+\digamma^2)&\frac{1}{2}\eta\otimes\langle
a,.\rangle+\frac{1}{2}(\xi+i\digamma)\iota^s_\eta\\
(\xi-i\digamma)a&\frac{1}{2}a\iota_\eta+\frac{1}{2}(\xi-i\digamma)\eta\otimes&a\langle
a,.\rangle+\eta\otimes_s\iota_\eta,
\end{pmatrix}
\end{aligned}\end{equation*}
where $a$ is a suitable symmetric 2-tensor.
\end{lemma}

\begin{proof}
This is an algebraic symbolic computation, so in particular it can be
done pointwise. Since one can
arrange that the metric $g_{\scl}$ used to compute adjoints is of the form $x^{-4}dx^2+x^{-2}dy^2$, where
$dy^2$ is the flat metric, at the point in question, one can simply
use this in the computation. With our coordinates at the point in question, trivializing the inner 
product, $g_{\scl}$, the inner product on one-forms is given by the matrix 
$$
\begin{pmatrix}1&0\\0&\Id\end{pmatrix}
$$
while on 2-tensors by 
$$
\begin{pmatrix}1&0&0&0\\0&\Id&0&0\\0&0&\Id&0\\0&0&0&\Id\end{pmatrix}. 
$$

First consider one-forms. Recall from Section~\ref{sec:idea}
that the full principal symbol of $d$, in
$\Diffsc^1(X;\underline{\Cx},\Tsc^*X)$, with $\underline{\Cx}$ the
trivial bundle,
is, as a map from functions to one-forms,
\begin{equation}\label{eq:g-d-symbol}
\begin{pmatrix} \xi\\ \eta\otimes\end{pmatrix}.
\end{equation}
Thus the symbol of $d^s_\digamma=e^{-\digamma/x}d^s e^{\digamma/x}$,
which conjugation effectively replaces $\xi$ by $\xi+i\digamma$ (as
$e^{-\digamma/x} x^2D_x e^{\digamma/x}=x^2D_x+i\digamma$), is
\begin{equation*}
\begin{pmatrix} \xi+i\digamma\\ \eta\otimes\end{pmatrix}.
\end{equation*}
Hence $\delta^s_\digamma$ has symbol given by the adjoint of that of
$d^s_\digamma$ with respect to the inner product of $g_\scl$, which is
$$
\begin{pmatrix}
\xi-i\digamma&\iota_\eta\end{pmatrix}.
$$
Thus, the principal symbol of $d^s_\digamma\delta^s_\digamma$ is the
product,
$$
\begin{pmatrix}\xi^2+\digamma^2&(\xi+i\digamma)\iota_\eta\\(\xi-i\digamma)\eta\otimes&\eta\otimes\iota_\eta\end{pmatrix},
$$
proving the lemma for one forms.

We now turn to symmetric 2-tensors.
Again, recall from Section~\ref{sec:idea}
that the full principal symbol of the gradient relative to $g$, in $\Diffsc^1(X;\Tsc^*X;\Tsc^*X\otimes\Tsc^*X)$,
is, as a map from one-forms to 2-tensors (which we
write in the four block form as before) is
\begin{equation}\label{eq:g-grad-symbol}
\begin{pmatrix} \xi&0\\\eta\otimes&0\\0&\xi\\b&\eta\otimes\end{pmatrix},
\end{equation}
where $b$ is a 2-tensor on $Y=\pa X$, and thus
that of $d^s$ (with symmetric 2-tensors considered as a subspace of
2-tensors) is
$$
\begin{pmatrix} \xi&0\\\frac{1}{2}\eta\otimes&\frac{1}{2}\xi\\\frac{1}{2}\eta\otimes&\frac{1}{2}\xi\\a&\eta\otimes_s\end{pmatrix},
$$
with $a$ a symmetric 2-tensor (the symmetrization of $b$). (Notice
that $a,b$ only play a role in the principal symbol at the boundary,
not in the standard principal symbol, i.e.\ as $(\xi,\eta)\to\infty$.) Here $a$ arises due to the
treatment of $d^s$, which is defined using a standard metric $g$, as
an element of $\Diffsc(X;\Tsc^*X,\Sym^2\Tsc^*X)$; it is acting on the
one-dimensional space $\Span\{\frac{dx}{x^2}\}$ by multiplying the
coefficient of $\frac{dx}{x^2}$ to produce a symmetric 2-tensor on $Y$.
Note that here the lower right 
block has $(ijk)$ entry (corresponding to the $(ij)$ entry of the 
symmetric 2-tensor and the $k$ entry of the one-form) given by $\frac{1}{2}(\eta_i\delta_{jk}+\eta_j\delta_{ik})$.
Thus the symbol of $d^s_\digamma=e^{-\digamma/x}d^s e^{\digamma/x}$,
which conjugation effectively replaces $\xi$ by $\xi+i\digamma$ (as
$e^{-\digamma/x} x^2D_x e^{\digamma/x}=x^2D_x+i\digamma$), is
$$
\begin{pmatrix} \xi+i\digamma&0\\\frac{1}{2}\eta\otimes&\frac{1}{2}(\xi+i\digamma)\\\frac{1}{2}\eta\otimes&\frac{1}{2}(\xi+i\digamma)\\a&\eta\otimes_s\end{pmatrix}.
$$

Thus, $\delta^s_\digamma$ has symbol given by the adjoint of that of
$d^s_\digamma$ with respect to this inner product, which is
$$
\begin{pmatrix}
\xi-i\digamma&\frac{1}{2}\iota_\eta&\frac{1}{2}\iota_\eta&\langle a,.\rangle\\0&\frac{1}{2}(\xi-i\digamma)&\frac{1}{2}(\xi-i\digamma)&\iota^s_\eta\end{pmatrix}.
$$
Here the lower right block has $(\ell ij)$ entry given by
$\frac{1}{2}(\eta_i\delta_{\ell j}+\eta_j\delta_{i\ell})$. Here the
inner product $\langle a,.\rangle$ as well as $\iota_\eta$ are with
respect to the identity because of the trivialization of the inner
product; invariantly they with respect to the inner product induced by $h$.
Correspondingly, the product, in the more concise notation for
symmetric tensors, has the symbol as stated, proving the lemma.
\end{proof}

The
proof of the next proposition, on ellipticity, relies on the subsequently stated two lemmas, whose proofs in 
turn take up the rest of this section. 

\begin{prop}\label{prop:elliptic}
First consider the case of one forms. Let $\digamma>0$.
Given $\tilde\Omega$, a neighborhood of $X\cap M=\{x\geq 0,\ \rho\geq 0\}$ in $X$,
for suitable choice of the cutoff $\chi\in\CI_c(\RR)$ and of $M\in\Psisc^{-3,0}(X)$, the operator
$$
A_\digamma=N_\digamma+d^s_\digamma
M\delta^s_\digamma,\qquad N_\digamma=e^{-\digamma/x}LIe^{\digamma/x},\qquad d^s_\digamma=e^{-\digamma/x}d^s e^{\digamma/x},
$$
is elliptic in $\Psisc^{-1,0}(X;\Tsc^*X,\Tsc^*X)$ in $\tilde\Omega$.

On the other hand, consider the case of symmetric 2-tensors. Then
there exists $\digamma_0>0$ such that for $\digamma>\digamma_0$ the
following holds.
Given $\tilde\Omega$, a neighborhood of $X\cap M=\{x\geq 0,\ \rho\geq 0\}$ in $X$,
for suitable choice of the cutoff $\chi\in\CI_c(\RR)$ and of $M\in\Psisc^{-3,0}(X;\Tsc^*X,\Tsc^*X)$, the operator
$$
A_\digamma=N_\digamma+d^s_\digamma
M\delta^s_\digamma,\qquad N_\digamma=e^{-\digamma/x}LIe^{\digamma/x},\qquad d^s_\digamma=e^{-\digamma/x}d^s e^{\digamma/x},
$$
is elliptic in $\Psisc^{-1,0}(X;\Sym^2\Tsc^*X,\Sym^2\Tsc^*X)$ in $\tilde\Omega$.
\end{prop}

\begin{proof}
The proof of this proposition is straightforward given the two lemmas
we prove below. Indeed, as we prove below in
Lemma~\ref{lemma:infinity-elliptic}, provided $\chi\geq 0$,
$\chi(0)>0$, the operator
$e^{-\digamma/x}LIe^{\digamma/x}$ has positive definite principal symbol at
fiber infinity in the scattering cotangent bundle {\em  when
restricted to the subspace of $\Tsc^*X$ or $\Sym^2\Tsc^*X$ given by
the kernel of
the symbol of $\delta^s_\digamma$}, where the
inner product is that of the scattering metric we consider (with
respect to which $\delta^s$ is computed); in
Lemma~\ref{lemma:finite-elliptic} we show a similar statement for the
principal symbol at finite points under the assumption that $\chi$ is
sufficiently close, in a suitable sense, to an even positive Gaussian,
with the complication that for
2-tensors we need to assume $\digamma>0$ sufficiently large. Thus, if
we add $d^s_\digamma M\delta^s_\digamma$ to it, where $M$ has positive
principal symbol, and is of the correct order, we obtain an elliptic
operator, completing the proof of Proposition~\ref{prop:elliptic}.
\end{proof}

We are thus reduced to proving the two lemmas we used.

\begin{lemma}\label{lemma:infinity-elliptic}
Both on one-forms and on symmetric 2-tensors, $N_\digamma$ is elliptic 
at fiber infinity in $\Tsc^*X$ when restricted 
to the kernel of the principal symbol of $\delta^s_\digamma$. 
\end{lemma}

\begin{proof}
This is very similar to the scalar
setting.
With
$$
S=\frac{X-\alpha(\hat Y)|Y|^2}{|Y|},\ \hat Y=\frac{Y}{|Y|},
$$
the Schwartz kernel of $N_{\digamma}$ at the  
scattering front face $x=0$ is, by \eqref{eq:sc-SK-near-ff},
given by
\begin{equation*}\begin{aligned}
&e^{-\digamma X}|Y|^{-n+1}\chi(S)\Big(\Big(S\frac{dx}{x^2}+\hat
Y\cdot\, \frac{dy}{x}\Big) \Big((S+2\alpha|Y|)(x^2\pa_x)+\hat Y\cdot(x\pa_{y})\Big) \Big)
\end{aligned}\end{equation*}
on one forms, respectively
\begin{equation*}\begin{aligned}
&e^{-\digamma X}|Y|^{-n+1}\chi(S)\\
&\qquad \Big(\Big(\Big(S\frac{dx}{x^2}+\hat
Y\cdot\, \frac{dy}{x}\Big) \otimes \Big(\Big(S\frac{dx}{x^2}+\hat
Y\cdot\, \frac{dy}{x}\Big)\Big)\Big)
\Big)\\
&\qquad\qquad\Big(\Big((S+2\alpha|Y|) (x^2\pa_x)+\hat Y\cdot(x\pa_{y})\Big)\otimes\Big((S+2\alpha|Y|) (x^2\pa_x)+\hat Y\cdot(x\pa_{y})\Big) \Big)
\end{aligned}\end{equation*}
on 2-tensors,
where $\hat Y$ is regarded as a tangent vector which acts on
covectors, and where $(S+2\alpha|Y|) (x^2\pa_x)+\hat Y\cdot(x\pa_{y})$ maps
one forms to scalars, thus
$$
\Big((S+2\alpha|Y|) (x^2\pa_x)+\hat
Y\cdot(x\pa_{y})\Big)\otimes \Big((S+2\alpha|Y|) (x^2\pa_x)+\hat Y\cdot(x\pa_{y})\Big)
$$
maps symmetric 2-tensors to scalars, while
$S\frac{dx}{x^2}+\hat Y\cdot\, \frac{dy}{x}$ maps scalars to one
forms, so
$$
\Big(S\frac{dx}{x^2}+\hat Y\cdot\, \frac{dy}{x}\Big)\otimes
\Big(S\frac{dx}{x^2}+\hat Y\cdot\, \frac{dy}{x}\Big)
$$
maps scalars to symmetric 2-tensors.
In order to make the notation less confusing, we employ a matrix notation,
\begin{equation*}\begin{aligned}
&\Big(S\frac{dx}{x^2}+\hat Y\cdot\, \frac{dy}{x}\Big)
\Big((S+2\alpha|Y|) (x^2\pa_x)+\hat
Y\cdot(x\pa_{y})\Big)\\
&\qquad=\begin{pmatrix}S (S+2\alpha|Y|)&S\langle\hat
  Y,\cdot\rangle\\\hat Y (S+2\alpha|Y|)&\hat Y \langle\hat
  Y,\cdot\rangle\end{pmatrix},
\end{aligned}\end{equation*}
with the first column and row corresponding to $\frac{dx}{x^2}$, resp.\
$x^2\pa_x$, and the second column and row to the (co)normal vectors.
For 2-tensors, as before, we use a decomposition
$$
\frac{dx}{x^2}\otimes\frac{dx}{x^2},\ \frac{dx}{x^2}\otimes\,\frac{dy}{x},\
\frac{dy}{x}\otimes\frac{dx}{x^2},\ \frac{dy}{x}\otimes \frac{dy}{x},
$$
where the symmetry of the 2-tensor is the statement that the 2nd
and 3rd (block) entries are the same. For the actual endomorphism
we write
\begin{equation}\begin{aligned}\label{eq:2-tensor-S}
&\begin{pmatrix}S^2\\S\langle\hat
  Y,\cdot\rangle_1\\S\langle\hat
  Y,\cdot\rangle_2\\\langle \hat Y,\cdot\rangle_1 \langle \hat
  Y,\cdot\rangle_2\end{pmatrix}
\begin{pmatrix}(S+2\alpha|Y|)^2\hat Y_1\hat Y_2&(S+2\alpha|Y|)\hat
  Y_1\hat Y_2 \langle \hat Y,\cdot\rangle_1&(S+2\alpha|Y|)\hat Y_1\hat
  Y_2 \langle \hat Y,\cdot\rangle_2&\hat Y_1\hat Y_2 \langle \hat
  Y,\cdot\rangle_1 \langle \hat Y,\cdot\rangle_2\end{pmatrix}\\
&=\begin{pmatrix}S^2(S+2\alpha|Y|)^2&S^2(S+2\alpha|Y|)\langle\hat
  Y,\cdot\rangle_1&S^2(S+2\alpha|Y|)\langle\hat
  Y,\cdot\rangle_2&S^2\langle \hat Y,\cdot\rangle_1
  \langle \hat Y,\cdot\rangle_2\\
S(S+2\alpha|Y|)^2Y_1&S(S+2\alpha|Y|)\hat Y_1\langle \hat
Y,.\rangle_1&S(S+2\alpha|Y|)\hat Y_1\langle \hat
Y,.\rangle_2&S\hat Y_1 \langle \hat Y,\cdot\rangle_1
  \langle \hat Y,\cdot\rangle_2\\
S(S+2\alpha|Y|)^2Y_2&S(S+2\alpha|Y|)\hat Y_2\langle \hat
Y,.\rangle_1&S(S+2\alpha|Y|)\hat Y_2\langle \hat
Y,.\rangle_2&S\hat Y_2 \langle \hat Y,\cdot\rangle_1
  \langle \hat Y,\cdot\rangle_2\\
(S+2\alpha|Y|)^2\hat Y_1\hat Y_2&(S+2\alpha|Y|)\hat Y_1\hat Y_2 \langle \hat Y,\cdot\rangle_1&(S+2\alpha|Y|)\hat Y_1\hat Y_2 \langle \hat Y,\cdot\rangle_2&\hat Y_1\hat Y_2 \langle \hat Y,\cdot\rangle_1 \langle \hat Y,\cdot\rangle_2\end{pmatrix}.
\end{aligned}\end{equation}
Here we write subscripts $1$ and $2$ for clarity on $\hat Y$ to denote
whether it is acting on the first or the second factor, though this
also immediately follows from its position within the matrix.

Now, the standard principal symbol is that of the conormal singularity
at the diagonal, i.e.\ $X=0$, $Y=0$. Writing $(X,Y)=Z$,
$(\xi,\eta)=\zeta$, we would need
to evaluate the $Z$-Fourier transform as $|\zeta|\to\infty$. This
was discussed in \cite{UV:local} around Equation~(3.8); the leading
order behavior of the Fourier
transform as $|\zeta|\to\infty$ can be obtained by working on the
blown-up space of the diagonal, with coordinates $|Z|,\hat
Z=\frac{Z}{|Z|}$ (as well as $z=(x,y)$), and integrating the
restriction of the Schwartz kernel to the front face, $|Z|^{-1}=0$,
after removing the singular factor $|Z|^{-n+1}$, along the {\em
  equatorial sphere} corresponding to $\zeta$, and given by $\hat
Z\cdot\zeta=0$. Now, concretely in our setting, in view of the
infinite order vanishing, indeed compact support, of the Schwartz
kernel as $X/|Y|\to\infty$ (and $Y$ bounded),
we may work in semi-projective coordinates, i.e.\ in spherical
coordinates in $Y$, but $X/|Y|$ as the normal variable; the equatorial
sphere then becomes $(X/|Y|)\xi+\hat Y\cdot\eta=0$ (with the integral
of course relative to an appropriate positive density). With $\tilde
S=X/|Y|$, keeping in mind that terms with extra vanishing factors at
the front face, $|Y|=0$ can be dropped, we thus need to integrate
\begin{equation}\label{eq:one-form-standard-symbol}
\begin{pmatrix}\tilde S^2&\tilde S\langle\hat  
  Y,\cdot\rangle\\ \tilde S\hat Y&\hat Y \langle\hat  
  Y,\cdot\rangle\end{pmatrix}\chi(\tilde S)=\begin{pmatrix} \tilde S\\
  \hat Y\end{pmatrix}\otimes\begin{pmatrix}\tilde S&\hat Y\end{pmatrix}\chi(\tilde S),
\end{equation} 
on this equatorial sphere in the case of one-forms, and the analogous
expression in the case of symmetric 2-tensors. Now, for $\chi\geq 0$ this matrix is a
positive multiple of the projection to the span of $(\tilde S,\hat
Y)$. As $(\tilde S,\hat Y)$ runs through the $(\xi,\eta)$-equatorial
sphere, we are taking a positive (in the sense of non-negative) linear
combination of the projections to the span of the vectors in this
orthocomplement, with the weight being strictly positive as long as
$\chi(\tilde S)>0$ at the point in question. But by
Lemma~\ref{lemma:potential-complement}, the kernel of the standard
principal symbol of $\delta^s_\digamma$ consists of covectors of the
form $v=(v_0,v')$ with $\xi v_0+\eta\cdot v'=0$. Hence, if we show that for
each such non-zero vector $(v_0,v')$ there is at least one $(\tilde S,\hat Y)$
with $\chi(\tilde S)>0$ and $\xi\tilde S+\eta\cdot \hat Y=0$ and
$\tilde S v_0+\hat Y \cdot v'\neq 0$, we
conclude that the integral of the projections is positive, thus the principal
symbol of our operator is elliptic, on the  kernel of the standard
principal symbol of $\delta^s_\digamma$. But this is straightforward
if $\chi(0)>0$:
\begin{enumerate}
\item
if $v'=0$ then $\xi=0$ (since $v_0\neq 0$), one may
take $\tilde S\neq 0$ small, $\hat Y$ orthogonal to $\eta$ (such
$\hat Y$ exists as $\eta\in\RR^{n-1}$, $n\geq 3$),
\item
if $v'\neq 0$ and $v'$ is not a multiple of $\eta$, then take $\hat Y$
orthogonal to $\eta$ but not to $v'$, $\tilde S=0$,
\item
if $v'=c\eta$ with $v'\neq 0$ (so $c$ and $\eta$ do not vanish) then
$\xi v_0+c|\eta|^2=0$ so with $\hat Y$ still to be chosen if we let
$\tilde S=-\frac{\eta\cdot\hat Y}{\xi}$, then $\tilde Sv_0+\hat Y\cdot
v'=c(\hat Y\cdot\eta)(1+\frac{|\eta|^2}{\xi^2})$ which is non-zero as
long as $\hat Y\cdot\eta\neq 0$; this can be again arranged, together
with $\hat Y\cdot\eta$ being sufficiently small  (such
$\hat Y$ exists again as $\eta\in\RR^{n-1}$, $n\geq 3$), so that $\tilde S$ is
small enough in order to ensure $\chi(\tilde S)>0$.
\end{enumerate}
This shows that the principal symbol is positive definite on the
kernel of the symbol of $\delta^s_\digamma$.

In the case of symmetric 2-tensors, the matrix
\eqref{eq:one-form-standard-symbol} is replaced by
\begin{equation}\label{eq:tensor-symbol}
\begin{pmatrix} \tilde S^2\\ \tilde S\hat Y_1\\
  \tilde S\hat Y_2\\ \hat Y_1\otimes\hat Y_2\end{pmatrix}
\otimes\begin{pmatrix} \tilde S^2&\tilde S\langle \hat
  Y,\cdot\rangle_1&\tilde S\langle \hat Y,\cdot\rangle_2&\langle \hat
  Y\otimes\hat Y,\cdot\rangle\end{pmatrix}\chi(\tilde S),
\end{equation}
which again is a non-negative multiple of a projection.
For a symmetric
2-tensor of the form $v=(v_{NN},v_{NT},v_{NT},v_{TT})$ in the kernel
of the principal symbol of $\delta^s_\digamma$, we have by
Lemma~\ref{lemma:potential-complement} that
\begin{equation}\begin{aligned}\label{eq:deltasF-princ-kernel}
&\xi v_{NN}+\eta\cdot v_{NT}=0,\\
&\xi  v_{NT}+\frac{1}{2}(\eta_1+\eta_2)\cdot v_{TT}=0,
\end{aligned}\end{equation}
where $\eta_1$ resp.\ $\eta_2$ denoting that the inner product is
taken in the first, resp.\ second, slots. Taking the inner product of
the second equation with $\eta$ gives
$$
\xi \eta\cdot v_{NT}+(\eta\otimes\eta) v_{TT}=0.
$$
Substituting this into the first equation yields
$$
\xi^2 v_{NN}=(\eta\otimes\eta)v_{TT}.
$$
We now consider two cases, $\xi=0$ and $\xi\neq 0$.

If $\xi\neq 0$, then for a symmetric 2-tensor being in the kernel of
the principal symbol of $\delta^s_\digamma$  at fiber infinity and of
\eqref{eq:tensor-symbol} for $(\tilde S,\hat Y)$ satisfying $\xi\tilde
S+\eta\cdot \hat Y=0$, i.e.\ $\tilde S=-\frac{\eta}{\xi}\cdot\hat Y$ is equivalent to
\begin{equation}\begin{aligned}\label{eq:2-tensor-kernel}
&v_{NN}=\xi^{-2}(\eta\otimes\eta)v_{TT},\\
&v_{NT}=-\frac{1}{2\xi}(\eta_1+\eta_2)\cdot v_{TT}\\
&\Big(\Big(\frac{\eta\cdot\hat
  Y}{\xi}\Big)^2\frac{\eta\otimes\eta}{\xi^2}+\frac{\eta\cdot\hat
  Y}{\xi^2}\Big(\frac{\eta}{\xi}\otimes \hat Y+\hat Y\otimes
\frac{\eta}{\xi}\Big)+\hat Y\otimes\hat Y\Big)\cdot v_{TT}=0,
\end{aligned}\end{equation}
and the last equation is equivalent to
$$
\Big(\Big(\frac{\eta\cdot\hat Y}{\xi}\frac{\eta}{\xi}+\hat
Y\Big)\otimes \Big(\frac{\eta\cdot\hat Y}{\xi}\frac{\eta}{\xi}+\hat
Y\Big)\Big)\cdot v_{TT}=0.
$$
If $\eta=0$, the first two equations say directly that $v_{NN}$ and
$v_{NT}$ vanish, while the last one states that $(\hat Y\otimes\hat
Y)\cdot v_{TT}=0$ for all $\hat Y$ (we may simply take $\tilde S=0$);
but symmetric 2-tensors of the form $\hat Y\otimes\hat
Y$ span the space of all symmetric 2-tensors (as
$w_1\otimes w_2+w_2\otimes
w_1=(w_1+w_2)\otimes(w_1+w_2)-w_1\otimes w_1-w_2\otimes w_2$), so we
conclude that $v_{TT}=0$, and thus $v=0$ in this case.
On the other hand, if $\eta\neq 0$ then
taking $\hat Y=\ep\hat\eta+(1-\ep^2)^{1/2}\hat Y^\perp$ and
substituting into this equation yields
$$
\Big(\Big(1+\frac{|\eta|^2}{\xi^2}\Big)^2\ep^2\hat\eta\otimes\hat\eta+\Big(1+\frac{|\eta|^2}{\xi^2}\Big)\ep(1-\ep^2)^{1/2}(\hat
\eta\otimes\hat Y^\perp+\hat Y^\perp\otimes\hat \eta)+(1-\ep^2)\hat
Y^\perp\otimes\hat Y^\perp\Big)\cdot v_{TT}=0.
$$
Note that $\tilde S=-\ep\frac{|\eta|}{\xi}$, so $|\tilde S|$ is small when $|\ep|$ is
sufficiently small.
Substituting in $\ep=0$ yields $(\hat Y^\perp\otimes\hat Y^\perp)\cdot
v_{TT}=0$; since cotensors of the form $\hat Y^\perp\otimes\hat
Y^\perp$ span $\eta^\perp\otimes\eta^\perp$ ($\eta^\perp$ being the
orthocomplement of $\eta$), we conclude that $v_{TT}$ is orthogonal to
every element of $\eta^\perp\otimes\eta^\perp$. Next, taking the
derivative in $\ep$ at $\ep=0$ yields $(\hat
\eta\otimes\hat Y^\perp+\hat Y^\perp\otimes\hat \eta)\cdot v_{TT}=0$
for all $\hat Y^\perp$; symmetric tensors of this form, together with
$\eta^\perp\otimes\eta^\perp$, span all tensors in
$(\eta\otimes\eta)^\perp$. Finally taking the second derivative at
$\ep=0$ shows that $(\hat\eta\otimes\hat\eta)\cdot v_{TT}=0$, this in
conclusion $v_{TT}=0$. Combined with the first two equations of
\eqref{eq:2-tensor-kernel}, one concludes that $v=0$, thus the desired
ellipticity follows.

On the other hand, if $\xi=0$ (and so $\eta\neq 0$),
then for a symmetric 2-tensor being in the kernel of
the principal symbol of $\delta^s_\digamma$  at fiber infinity and of
\eqref{eq:tensor-symbol} for $(\tilde S,\hat Y)$ satisfying $\xi\tilde
S+\eta\cdot \hat Y=0$, i.e.\ $\eta\cdot\hat Y=0$ is equivalent to
\begin{equation}\begin{aligned}\label{eq:2-tensor-kernel-equator}
&\eta\cdot v_{NT}=0,\\
&(\eta_1+\eta_2)\cdot v_{TT}=0,\\
&\tilde S^2 v_{NN}+2\tilde S\hat Y\cdot v_{NT}+(\hat Y\otimes\hat Y)\cdot v_{TT}=0.
\end{aligned}\end{equation}
Since there are no constraints on $\tilde S$ (apart from $|\tilde S|$
small), we can differentiate the last equation up to two times and
evaluate the result at $0$ to
conclude that $v_{NN}=0$, $\hat Y\cdot v_{NT}=0$ and $(\hat Y\otimes\hat
Y)\cdot v_{TT}=0$. Combined with the first two equations of
\eqref{eq:2-tensor-kernel-equator}, this shows $v=0$, so again the
desired ellipticity follows.

Thus, in summary, both on one forms and on symmetric 2-tensors the
principal symbol at fiber infinity is elliptic on the kernel of that
of $\delta^s_\digamma$, proving the lemma.
\end{proof}

\begin{lemma}\label{lemma:finite-elliptic}
For $\digamma>0$ on one forms 
$N_\digamma$ is elliptic at finite points of $\Tsc^*X$ when restricted 
to the kernel of the principal symbol of $\delta^s_\digamma$. On the 
other hand, there exists $\digamma_0>0$ such that on symmetric 
2-tensors 
$N_\digamma$ is elliptic at finite points of $\Tsc^*X$ when restricted 
to the kernel of the principal symbol of $\delta^s_\digamma$. 
\end{lemma}

\begin{proof}
Again this is similar to, but technically much more involved than, the
scalar setting. We recall from \cite{UV:local} that the kernel is based on using a compactly supported 
$\CI$ localizer, $\chi$, but for the actual computation it is 
convenient to use a Gaussian instead $\chi_0$ instead. One recovers the result by 
taking $\phi\in\CI_c(\RR)$, $\phi\geq 0$, identically $1$ near $0$, and considering an approximating sequence
$\chi_k=\phi(./k)\chi_0$. Then the Schwartz kernels at the front face still converge in
the space of distributions conormal to the diagonal, which means that
the principal symbols (including at finite points) also converge,
giving the desired ellipticity for sufficiently large $k$.

Recall that the scattering principal symbol is the
Fourier transform of the Schwartz kernel at the front face, so
we now need to compute this Fourier transform. We start with the
one form case. Taking
$\chi(s)=e^{-s^2/(2\nu(\hat Y))}$ as in the scalar
case considered in \cite{UV:local} for the computation (in the scalar case we took
$\nu=\digamma^{-1}\alpha$; here we leave it unspecified for now,
except demanding $0<\nu<2\digamma^{-1}\alpha$ as needed for the
Schwartz kernel to be rapidly decreasing at infinity on the front face), we can compute the $X$-Fourier transform
exactly as before, keeping in mind that this needs to be evaluated at
$-\xi$ (just like the $Y$ Fourier transform needs to be evaluated at
$-\eta$) due to our definition of $X$:
\begin{equation*}\begin{aligned}
&|Y|^{2-n}e^{-i\alpha(-\xi-i\digamma)|Y|^2}\begin{pmatrix}D_\sigma^2-2\alpha|Y| D_\sigma&-D_\sigma\langle\hat
  Y,\cdot\rangle\\ \hat Y (-D_\sigma+2\alpha|Y|)&\hat Y \langle\hat
  Y,\cdot\rangle\end{pmatrix}
\hat\chi((-\xi-i\digamma)|Y|)\\
&=c\sqrt{\nu}|Y|^{2-n}e^{i\alpha(\xi+i\digamma)|Y|^2}\begin{pmatrix}D_\sigma^2-2\alpha|Y| D_\sigma&-D_\sigma\langle\hat
  Y,\cdot\rangle\\ \hat Y (-D_\sigma+2\alpha|Y|)&\hat Y \langle\hat
  Y,\cdot\rangle\end{pmatrix}
 e^{-\nu (\xi+i\digamma)^2|Y|^2/2}
\end{aligned}\end{equation*}
with $c>0$, and with $D_\sigma$ differentiating the argument of $\hat\chi$.
One is left with computing the $Y$-Fourier
transform, which in polar coordinates takes the form
\begin{equation*}\begin{aligned}
&\int_{\sphere^{n-2}}\int_{[0,\infty)}e^{i|Y|\hat
  Y\cdot\eta}|Y|^{2-n}e^{i\alpha(\xi+i\digamma)|Y|^2}\\
&\qquad\qquad\qquad\begin{pmatrix}-D_\sigma(-D_\sigma+2\alpha|Y|)&-D_\sigma\langle\hat
  Y,\cdot\rangle\\ \hat Y (-D_\sigma+2\alpha|Y|)&\hat Y \langle\hat
  Y,\cdot\rangle\end{pmatrix}
\hat\chi(-(\xi+i\digamma)|Y|)
|Y|^{n-2}\,d|Y|\,d\hat Y,
\end{aligned}\end{equation*}
and the factors $|Y|^{\pm(n-2)}$ cancel as in the scalar case.
Explicitly evaluating the derivatives, writing
$$
\phi(\xi,\hat Y)=\nu(\hat Y)(\xi+i\digamma)^2-2i\alpha(\hat 
  Y)(\xi+i\digamma),
$$
yields
\begin{equation}\begin{aligned}\label{eq:one-form-exp-1}
\int_{\sphere^{n-2}}\int_0^\infty e^{i|Y|\hat
  Y\cdot\eta}&\begin{pmatrix}i\nu(\xi+i\digamma)(i\nu(\xi+i\digamma)+2\alpha)
  |Y|^2+\nu&i\nu(\xi+i\digamma) |Y|\langle\hat Y,\cdot\rangle\\
\hat Y (i\nu(\xi+i\digamma)+2\alpha) |Y|&\hat Y\langle\hat Y,\cdot\rangle \end{pmatrix}\\
&\qquad\qquad\qquad\qquad\qquad\times  e^{-\phi|Y|^2/2}\,d|Y|\,d\hat Y.
\end{aligned}\end{equation}

We extend
the integral in $|Y|$ to $\RR$, replacing it by a variable $t$, and
using that the integrand is invariant under the joint change of
variables $t\to -t$ and $\hat Y\to -\hat Y$. This gives
\begin{equation*}\begin{aligned}
&\frac{1}{2}\int_{\sphere^{n-2}}\int_{\RR}e^{it\hat
  Y\cdot\eta}\\
&\qquad\begin{pmatrix}i\nu(\xi+i\digamma)(i\nu(\xi+i\digamma)+2\alpha)
  t^2+\nu&i\nu(\xi+i\digamma) t\langle\hat Y,\cdot\rangle\\
\hat Y (i\nu(\xi+i\digamma)+2\alpha) t&\hat Y\langle\hat Y,\cdot\rangle \end{pmatrix}\\
&\qquad\qquad\qquad\qquad\qquad \times e^{-\phi t^2/2}\,dt\,d\hat Y.
\end{aligned}\end{equation*}
Now the $t$ integral is a Fourier transform evaluated at $-\hat Y\cdot\eta$, under which
multiplication by $t$ becomes $D_{\hat Y\cdot\eta}$. Since the
Fourier transform of $e^{-\phi(\xi,\hat Y)t^2/2}$ is a constant multiple of
\begin{equation}\label{eq:basic-FT}
\phi(\xi,\hat Y)^{-1/2}e^{-(\hat
  Y\cdot\eta)^2/(2\phi(\xi,\hat Y))},
\end{equation}
we are left with
\begin{equation*}\begin{aligned}
\int_{\sphere^{n-2}}\phi(\xi,\hat Y)^{-1/2}
&\begin{pmatrix}i\nu(\xi+i\digamma)(i\nu(\xi+i\digamma)+2\alpha)
  D_{\hat Y\cdot\eta}^2+\nu&i\nu(\xi+i\digamma) \langle\hat Y,\cdot\rangle D_{\hat Y\cdot\eta}\\
\hat Y (i\nu(\xi+i\digamma)+2\alpha) D_{\hat Y\cdot\eta}&\hat Y\langle\hat Y,\cdot\rangle \end{pmatrix}\\
&\qquad\qquad\qquad \times e^{-(\hat
  Y\cdot\eta)^2/(2\phi(\xi,\hat Y))}
\,d\hat Y,
\end{aligned}\end{equation*}
which explicitly gives
\begin{equation}\begin{aligned}\label{eq:one-form-exp-2}
&\int_{\sphere^{n-2}}\phi(\xi,\hat Y)^{-1/2}\\
&\qquad\begin{pmatrix}i\nu(\xi+i\digamma)(i\nu(\xi+i\digamma)+2\alpha)
  \Big(-\frac{(\hat
    Y\cdot\eta)^2}{\phi(\xi,\hat Y)^2}+\frac{1}{\phi(\xi,\hat Y)}\Big)
  +\nu&i\nu(\xi+i\digamma) \langle\hat Y,\cdot\rangle i\frac{\hat
    Y\cdot\eta}{\phi(\xi,\hat Y)}\\
\hat Y (i\nu(\xi+i\digamma)+2\alpha) i\frac{\hat
    Y\cdot\eta}{\phi(\xi,\hat Y)}&\hat Y\langle\hat Y,\cdot\rangle \end{pmatrix}\\
&\qquad\qquad\qquad\qquad\qquad\qquad \times e^{-(\hat
  Y\cdot\eta)^2/(2\phi(\xi,\hat Y))}\,d\hat Y.
\end{aligned}\end{equation}
Now observe that the top left entry of the matrix is exactly
$$
-\nu\phi(\xi,\hat Y) \Big(-\frac{(\hat
    Y\cdot\eta)^2}{\phi(\xi,\hat Y)^2}+\frac{1}{\phi(\xi,\hat Y)}\Big)
  +\nu=\frac{\nu(\hat
    Y\cdot\eta)^2}{\phi(\xi,\hat Y)}=\nu(\xi+i\digamma)(\nu(\xi+i\digamma)-2i\alpha) \frac{(\hat
    Y\cdot\eta)^2}{\phi(\xi,\hat Y)^2}.
$$
Thus, the matrix in the integrand is
\begin{equation*}
\begin{pmatrix}-\frac{\nu(\xi+i\digamma)}{\phi}(\hat Y\cdot\eta)\\\hat
  Y\end{pmatrix}\otimes\begin{pmatrix}-\frac{(\nu(\xi+i\digamma)-2i\alpha)}{\phi}(\hat
  Y\cdot\eta)&\langle\hat
  Y,\cdot\rangle\end{pmatrix}.
\end{equation*}
Now, if we take
$$
\nu=\digamma^{-1}\alpha
$$
as in the scalar case in \cite{UV:local}, then
$$
\nu(\xi+i\digamma)-2i\alpha=\nu(\xi-i\digamma),
$$
while
$$
\phi=(\xi+i\digamma)(\nu(\xi+i\digamma)-2i\alpha)=\nu(\xi^2+\digamma^2)
$$
is real,
so the matrix, with this choice of $\nu$, is orthogonal projection to
the span of $(-\frac{\nu(\xi+i\digamma)}{\phi}(\hat Y\cdot\eta),\hat  Y)$.
The expression \eqref{eq:one-form-exp-2} becomes
\begin{equation}\begin{aligned}\label{eq:one-form-exp-3}
&(\xi^2+\digamma^2)^{-1/2}\\
&\int_{\sphere^{n-2}}\nu^{-1/2}
\begin{pmatrix}-\frac{\nu(\xi+i\digamma)}{\xi^2+\digamma^2}(\hat Y\cdot\eta)\\\hat
  Y\end{pmatrix}\otimes\begin{pmatrix}-\frac{\nu(\xi-i\digamma)}{\xi^2+\digamma^2}(\hat
  Y\cdot\eta)&\langle\hat
  Y,\cdot\rangle\end{pmatrix}
e^{-(\hat
  Y\cdot\eta)^2/(2\nu(\xi^2+\digamma^2))}\,d\hat Y,
\end{aligned}\end{equation}
which
is thus a superposition of positive (in the sense of non-negative)
operators, which is thus itself positive. Further, if a vector
$(v_0,v')$ lies in the kernel of the principal symbol of
$\delta^s_\digamma$, i.e.\ $(\xi-i\digamma)v_0+\iota_\eta v'=0$, then
orthogonality to $(-\frac{\nu(\xi+i\digamma)}{\xi^2+\digamma^2}(\hat Y\cdot\eta),\hat
  Y)$ for any particular $\hat Y$ would mean
$$
0=-\frac{\nu(\xi-i\digamma)}{\xi^2+\digamma^2}(\hat
  Y\cdot\eta)v_0+\hat Y\cdot v'=\frac{\nu}{\xi^2+\digamma^2} (\eta\cdot v')(\hat
  Y\cdot\eta)+\hat Y\cdot v'.
$$
Note that $\sphere^{n-2}$ is at least one dimensional (i.e.\ is the sphere
in at least a 2-dimensional vector space). Consider $v'\neq 0$; this
would necessarily be the case of interest since
$v_0=-(\xi-i\digamma)^{-1}(\eta\cdot v')$. If $\eta=0$, picking $\hat
Y$ parallel to $v'$ shows that there is at least one choice of $\hat
Y$ for which this equality does not hold. If $\eta\neq 0$, and
$v'$ is not a multiple of $\eta$, we can take $\hat Y$ orthogonal
to $\eta$ and not orthogonal to $v'$, which again gives a choice
of $\hat Y$ for the equality above does not hold. Finally, if $v'$ is a
multiple of $\eta$, the expression at hand is just $\frac{\nu|\eta|^2}{\xi^2+\digamma^2}(\hat
  Y\cdot v')+\hat Y\cdot v'$, so choosing any $\hat Y$ not orthogonal
  to $v'$ again gives a $\hat Y$ for which the equality does not
  hold. Therefore, \eqref{eq:one-form-exp-3} is actually positive
  definite when restricted to the kernel of the symbol of $\delta^s_\digamma$, as claimed.

We now turn to the 2-tensor version.
With $B_{ij}$ corresponding to the terms with $i$ factors of $S$ and
$j$ factors of $S+2\alpha|Y|$ prior to the Fourier transform,
the analogue of \eqref{eq:one-form-exp-1} is
\begin{equation}\begin{aligned}\label{eq:tensor-exp-1}
\int_{\sphere^{n-2}}\int_0^\infty e^{i|Y|\hat
  Y\cdot\eta}&
\begin{pmatrix}B_{22}&B_{21}\langle\hat
  Y,\cdot\rangle_1&B_{21}\langle\hat
  Y,\cdot\rangle_2&B_{20}\langle \hat Y,\cdot\rangle_1
  \langle \hat Y,\cdot\rangle_2\\
B_{12}Y_1&B_{11}\hat Y_1\langle \hat
Y,.\rangle_1&B_{11}\hat Y_1\langle \hat
Y,.\rangle_2&B_{10}\hat Y_1 \langle \hat Y,\cdot\rangle_1
  \langle \hat Y,\cdot\rangle_2\\
B_{12}Y_2&B_{11}\hat Y_2\langle \hat
Y,.\rangle_1&B_{11}\hat Y_2\langle \hat
Y,.\rangle_2&B_{10}\hat Y_2 \langle \hat Y,\cdot\rangle_1
  \langle \hat Y,\cdot\rangle_2\\
B_{02}\hat Y_1\hat Y_2&B_{01}\hat Y_1\hat Y_2 \langle \hat Y,\cdot\rangle_1
&B_{01}\hat Y_1\hat Y_2 \langle \hat Y,\cdot\rangle_2&B_{00}\hat Y_1\hat Y_2 \langle \hat Y,\cdot\rangle_1 \langle \hat Y,\cdot\rangle_2\end{pmatrix}\\
&\qquad\qquad\qquad\qquad\qquad\times  e^{-\phi|Y|^2/2}\,d|Y|\,d\hat Y,
\end{aligned}\end{equation}
with
\begin{equation*}\begin{aligned}
B_{00}&=1,\\
B_{10}&=i\nu(\xi+i\digamma)|Y|,\\
B_{20}&=-\nu^2(\xi+i\digamma)^2|Y|^2+\nu,\\
B_{01}&=i(\nu(\xi+i\digamma)-2i\alpha)|Y|,\\
B_{11}&=-\nu(\xi+i\digamma)(\nu(\xi+i\digamma)-2i\alpha)|Y|^2+\nu,\\
B_{21}&=-i\nu^2(\xi+i\digamma)^2(\nu(\xi+i\digamma)-2i\alpha)|Y|^3+(3i\nu^2(\xi+i\digamma)+2\alpha\nu)|Y|,\\
B_{02}&=-(\nu(\xi+i\digamma)-2i\alpha)^2|Y|^2+\nu,\\
B_{12}&=-i\nu(\xi+i\digamma)(\nu(\xi+i\digamma)-2i\alpha)^2|Y|^3,\\
B_{22}&=\nu^2(\xi+i\digamma)^2(\nu(\xi+i\digamma)-2i\alpha)^2|Y|^4+\nu(-6\nu^2(\xi+i\digamma)^2+12i\nu\alpha(\xi+i\digamma)+4\alpha^2)|Y|^2+3\nu^2.
\end{aligned}\end{equation*}
Note that the leading term of $B_{jk}$, in terms of the power of $|Y|$
involved, is simply $(i\nu(\xi+i\digamma)|Y|)^j
(i(\nu(\xi+i\digamma)-2i\alpha)|Y|)^k$; this arises by all derivatives
in \eqref{eq:2-tensor-S} arising by Fourier transforming in $S$ (which
gives a derivative $-D_\sigma$ in the dual variable $\sigma$)
falling on the exponential, $e^{-\nu\sigma^2/2}$, which is then
evaluated at $\sigma=-(\xi+i\digamma)|Y|$. However, for the full scattering principal
symbol all terms are relevant.

Next, we extend the $|Y|$ integral to $\RR$, writing the corresponding
variable as $t$ and do the Fourier transform in $t$ (with a minus
sign, i.e.\ evaluated at $-\hat Y\cdot\eta$) as in the one-form
setting. This replaces $t$ by $D_{\hat Y\cdot\eta}$, as above, and in
view of \eqref{eq:basic-FT}, explicitly evaluating the derivatives, we
obtain the following analogue of \eqref{eq:one-form-exp-2}
\begin{equation}\begin{aligned}\label{eq:tensor-exp-2}
&\int_{\sphere^{n-2}}
\begin{pmatrix}C_{22}&C_{21}\langle\hat
  Y,\cdot\rangle_1&C_{21}\langle\hat
  Y,\cdot\rangle_2&C_{20}\langle \hat Y,\cdot\rangle_1
  \langle \hat Y,\cdot\rangle_2\\
C_{12}Y_1&C_{11}\hat Y_1\langle \hat
Y,.\rangle_1&C_{11}\hat Y_1\langle \hat
Y,.\rangle_2&C_{10}\hat Y_1 \langle \hat Y,\cdot\rangle_1
  \langle \hat Y,\cdot\rangle_2\\
C_{12}Y_2&C_{11}\hat Y_2\langle \hat
Y,.\rangle_1&C_{11}\hat Y_2\langle \hat
Y,.\rangle_2&C_{10}\hat Y_2 \langle \hat Y,\cdot\rangle_1
  \langle \hat Y,\cdot\rangle_2\\
C_{02}\hat Y_1\hat Y_2&C_{01}\hat Y_1\hat Y_2 \langle \hat Y,\cdot\rangle_1
&C_{01}\hat Y_1\hat Y_2 \langle \hat Y,\cdot\rangle_2&C_{00}\hat Y_1\hat Y_2 \langle \hat Y,\cdot\rangle_1 \langle \hat Y,\cdot\rangle_2\end{pmatrix}\\
&\qquad\qquad\qquad\qquad\qquad\times \phi(\xi,\hat Y)^{-1/2}
 e^{-(\hat
  Y\cdot\eta)^2/(2\phi(\xi,\hat Y))}\,d\hat Y,
\end{aligned}\end{equation}
where, with $\rho=\hat Y\cdot\eta$,
\begin{equation*}\begin{aligned}
C_{00}&=1,\\
C_{10}&=-\nu(\xi+i\digamma)\phi^{-1}\rho,\\
C_{20}&=\nu^2(\xi+i\digamma)^2\phi^{-2}\rho^2-2i\alpha\nu\phi^{-1}(\xi+i\digamma),\\
C_{01}&=-(\nu(\xi+i\digamma)-2i\alpha)\phi^{-1}\rho,\\
C_{11}&=\nu(\xi+i\digamma))(\nu(\xi+i\digamma)-2i\alpha)\phi^{-2}\rho^{2},\\
C_{21}&=-\nu^2(\xi+i\digamma)^2(\nu(\xi+i\digamma)-2i\alpha)\phi^{-3}\rho^3+2\alpha\nu
i\phi^{-1}\rho,\\
C_{02}&=(\nu(\xi+i\digamma)-2i\alpha)^2\phi^{-2}\rho^2+\phi^{-1}(\nu(\xi+i\digamma)-2i\alpha)2i\alpha,\\
C_{12}&=-\nu(\xi+i\digamma)(\nu(\xi+i\digamma)-2i\alpha)^2\phi^{-3}\rho^{3}-2i\alpha\nu\phi^{-1}\rho,\\
C_{22}&=\nu^2(\xi+i\digamma)^2 (\nu(\xi+i\digamma)-2i\alpha)^2\phi^{-4}\rho^{4}-4\alpha^2\nu\phi^{-2}\rho^{2}+4\alpha^2\nu\phi^{-1}.
\end{aligned}\end{equation*}
Note again that the highest order term, in terms of the power of
$\rho$, of $C_{jk}$ is
$(\nu(\xi+i\digamma))^j(\nu(\xi+i\digamma)-2i\alpha)^k(-1)^{j+k}\phi^{-j-k}$,
corresponding to all derivatives $D_\rho$ falling on the exponential
$e^{-\rho^2/(2\phi)}$, evaluated at $\rho=\hat Y\cdot\eta$.

Notice that $C_{11}$ is exactly the $(1,1)$ entry in the one-form
calculation, \eqref{eq:one-form-exp-2}, while $C_{10}$, resp.\
$C_{01}$, are the factors in
the $(1,2)$ and $(2,1)$ entries, for similar reasons. Now, it is easy
to check that the matrix in \eqref{eq:tensor-exp-2} is
\begin{equation}\begin{aligned}\label{eq:2-tensor-proj}
&\begin{pmatrix} C_{20}\\\hat Y_1 C_{10}\\\hat Y_2 C_{10}\\ \hat
  Y_1\hat Y_2\end{pmatrix}
\otimes\begin{pmatrix}C_{02}&C_{01}\langle\hat
  Y,\cdot\rangle_1&C_{01}\langle Y,\cdot\rangle_2&\langle\hat
  Y,\cdot\rangle_1\langle\hat Y,\cdot\rangle_2\end{pmatrix}.
\end{aligned}\end{equation}
Letting $\nu=\digamma^{-1}\alpha$ as in the one-form setting, the
second factor here is the adjoint (involving of complex conjugates) of
the first, in particular (with $\rho=\hat Y\cdot\eta$)
$$
C_{01}=-\nu(\xi-i\digamma)\phi^{-1}\rho,\
C_{02}=\nu^2(\xi-i\digamma)^2\phi^{-2}\rho^2+ 2i\alpha
\nu(\xi-i\digamma) \phi^{-1},\ \phi=\nu(\xi^2+\digamma^2),
$$
so \eqref{eq:2-tensor-proj} is just a positive multiple of projection to the
span of $(C_{20},\hat Y_1 C_{10},\hat Y_2 C_{10}, \hat
  Y_1\hat Y_2)$. Thus,
as in the one form setting, we have a superposition of positive (in
the sense of non-negative) operators, so it remains to check that as $\hat Y$
varies, these vectors span the kernel of $\delta^s_\digamma$.

For a symmetric
2-tensor of the form $v=(v_{NN},v_{NT},v_{NT},v_{TT})$ in the kernel
of the principal symbol of $\delta^s_\digamma$, we have by
Lemma~\ref{lemma:potential-complement} that
\begin{equation}\begin{aligned}\label{eq:deltasF-kernel}
&(\xi-i\digamma)v_{NN}+\eta\cdot v_{NT}+a\cdot v_{TT}=0,\\
&(\xi-i\digamma) v_{NT}+\frac{1}{2}(\eta_1+\eta_2)\cdot v_{TT}=0,
\end{aligned}\end{equation}
where $\eta_1$ resp.\ $\eta_2$ denoting that the inner product is
taken in the first, resp.\ second, slots. Taking the inner product of
the second equation with $\eta$ gives
$$
(\xi-i\digamma) \eta\cdot v_{NT}+(\eta\otimes\eta)\cdot v_{TT}=0.
$$
Substituting this into the first equation yields
$$
(\xi-i\digamma)^2 v_{NN}+((\xi-i\digamma)a-\eta\otimes\eta)\cdot v_{TT}=0,
$$
so
$$
v_{NN}=(\xi-i\digamma)^{-2}(\eta\otimes\eta-(\xi-i\digamma)a)\cdot v_{TT},\ v_{NT}=-2^{-1}(\xi-i\digamma)^{-1}(\eta_1+\eta_2)\cdot v_{TT}.
$$
For a fixed $\hat Y$ for $v$ in the kernel of the symbol of
$\delta^s_\digamma$ to be in the kernel of the projection
\eqref{eq:2-tensor-proj} means that
$$
\Big(C_{02}(\xi-i\digamma)^{-2}(\eta\otimes\eta-(\xi-i\digamma)a)-C_{01}(\xi-i\digamma)^{-1}(\eta\otimes
\hat Y+\hat Y\otimes\eta)+\hat Y\otimes\hat Y\Big)\cdot v_{TT}=0,
$$
so recalling $\nu=\digamma^{-1}\alpha$, $\phi=\nu(\xi^2+\digamma^2)$,
\begin{equation*}\begin{aligned}
&\Big((\xi+i\digamma)^{-1}(\xi^2+\digamma^2)^{-1}(\hat Y\cdot\eta)^2+ 2i\alpha (\xi^2+\digamma^2)^{-1})
((\xi-i\digamma)^{-1}(\eta\otimes\eta)-a)\\
&\qquad +(\xi^2+\digamma^2)^{-1} (\hat Y\cdot\eta) (\eta\otimes
\hat Y+\hat Y\otimes\eta) +\hat Y\otimes\hat Y\Big)\cdot v_{TT}=0.
\end{aligned}\end{equation*}

Now, it is convenient to rewrite this in terms of `semiclassical' (in
$h=\digamma^{-1}$) variables
\begin{equation*}
\xi_\digamma=\xi/\digamma,\ \eta_\digamma=\eta/\digamma.
\end{equation*}
It becomes
\begin{equation*}\begin{aligned}
&\Big((\xi_\digamma+i)^{-1}(\xi_\digamma^2+1)^{-1}(\hat Y\cdot\eta_\digamma)^2+ 2i\digamma^{-1}\alpha (\xi_\digamma^2+1)^{-1})
((\xi_\digamma-i)^{-1}(\eta_\digamma\otimes\eta_\digamma)-\digamma^{-1}a)\\
&\qquad +(\xi_\digamma^2+1)^{-1} (\hat Y\cdot\eta_\digamma) (\eta_\digamma\otimes
\hat Y+\hat Y\otimes\eta_\digamma) +\hat Y\otimes\hat Y\Big)\cdot v_{TT}=0.
\end{aligned}\end{equation*}
Letting $\digamma^{-1}=h\to 0$, one obtains
\begin{equation*}\begin{aligned}
&\Big((\xi_\digamma+i)^{-1}(\xi_\digamma^2+1)^{-1}(\hat Y\cdot\eta_\digamma)^2
(\xi_\digamma-i)^{-1}(\eta_\digamma\otimes\eta_\digamma)\\
&\qquad +(\xi_\digamma^2+1)^{-1} (\hat Y\cdot\eta_\digamma) (\eta_\digamma\otimes
\hat Y+\hat Y\otimes\eta_\digamma) +\hat Y\otimes\hat Y\Big)\cdot v_{TT}=0,
\end{aligned}\end{equation*}
i.e.
$$
\Big(\Big((\xi_\digamma^2+1)^{-1}(\hat 
Y\cdot\eta_\digamma)\eta_\digamma+\hat Y\Big) \otimes \Big((\xi_\digamma^2+1)^{-1}(\hat 
Y\cdot\eta_\digamma)\eta_\digamma+\hat Y\Big) \Big)\cdot v_{TT}=0.
$$
One can see that this last equation, when it holds for all $\hat Y$,
implies the vanishing of $v_{TT}$ just as for the principal symbol at
fiber infinity. Indeed,
if $\eta_\digamma=0$ then we have $(\hat Y\otimes\hat
Y)\cdot v_{TT}=0$ for all $\hat Y$,
and symmetric 2-tensors of the form $\hat Y\otimes\hat
Y$ span the space of all symmetric 2-tensors (as
$w_1\otimes w_2+w_2\otimes
w_1=(w_1+w_2)\otimes(w_1+w_2)-w_1\otimes w_1-w_2\otimes w_2$), so we
conclude that $v_{TT}=0$, and thus $v=0$ in this case.
On the other hand, if $\eta_\digamma\neq 0$ then
taking $\hat Y=\ep\hat\eta_\digamma+(1-\ep^2)^{1/2}\hat Y^\perp$ and
substituting into this equation yields
\[
\begin{split}
&\Big(\Big(1+\frac{|\eta_\digamma|^2}{\xi_\digamma^2+1}\Big)^2\ep^2\hat\eta_\digamma\otimes\hat\eta_\digamma+\Big(1+\frac{|\eta_\digamma|^2}{\xi_\digamma^2+1}\Big)\ep(1-\ep^2)^{1/2}(\hat
\eta_\digamma\otimes\hat Y^\perp+\hat Y^\perp\otimes\hat \eta_\digamma)\\
&\qquad +(1-\ep^2)\hat
Y^\perp\otimes\hat Y^\perp\Big)\cdot v_{TT}=0.
\end{split} 
\]
Substituting in $\ep=0$ yields $(\hat Y^\perp\otimes\hat Y^\perp)\cdot
v_{TT}=0$; since cotensors of the form $\hat Y^\perp\otimes\hat
Y^\perp$ span $\eta_\digamma^\perp\otimes\eta_\digamma^\perp$ ($\eta_\digamma^\perp$ being the
orthocomplement of $\eta_\digamma$), we conclude that $v_{TT}$ is orthogonal to
every element of $\eta_\digamma^\perp\otimes\eta_\digamma^\perp$. Next, taking the
derivative in $\ep$ at $\ep=0$ yields $(\hat
\eta_\digamma\otimes\hat Y^\perp+\hat Y^\perp\otimes\hat \eta_\digamma)\cdot v_{TT}=0$
for all $\hat Y^\perp$; symmetric tensors of this form, together with
$\eta_\digamma^\perp\otimes\eta_\digamma^\perp$, span all tensors in
$(\eta_\digamma\otimes\eta_\digamma)^\perp$. Finally taking the second derivative at
$\ep=0$ shows that $(\hat\eta_\digamma\otimes\hat\eta_\digamma)\cdot v_{TT}=0$, this in
conclusion $v_{TT}=0$. Combined with the first two equations of
\eqref{eq:2-tensor-kernel}, one concludes that $v=0$. Correspondingly
one concludes that for sufficiently large $\digamma>0$ one has
ellipticity at all finite points, which proves the lemma.
\end{proof}

{\em As already explained, this lemma completes the proof of
Proposition~\ref{prop:elliptic}.}

\section{The gauge condition and the proof of the main results}\label{sec:gauge}

The still remaining analytic issue is to
check that we can arrange the gauge condition,
$\delta^s_\digamma f_\digamma=0$. We do this by considering various
regions $\Omega_j$, which are manifolds with
corners: they have the artificial boundary, $\pa X$, which is `at infinity' in the scattering
calculus sense, as well as the `interior' boundary
$\pa_{\inter}\Omega_j$, which could be $\pa M$, or another (farther
away) hypersurface.

Recall that our gauge freedom is that we can add
to $f$ (without changing $If$) any tensor of the form $d^s v$, with $v$ vanishing at $\pa M$
or on a hypersurface further away, such as $\pa_{\inter}\Omega_j$, i.e.\ to
$f_\digamma=e^{-\digamma/x} f$ (without changing $I e^{\digamma/x}
f_\digamma$) any
tensor of the form $d^s_\digamma v_\digamma=e^{-\digamma/x} d
e^{\digamma/x} v_\digamma$ with a similar vanishing condition. If we let
$\Delta_{\digamma,s}=\delta^s_\digamma d^s_\digamma$ be the
`solenoidal Witten Laplacian', and we impose
Dirichlet boundary condition on $\pa_{\inter}\Omega_j$ (to get the
desired vanishing for $v_\digamma$), and we show that
$\Delta_{\digamma,s}$ is invertible (with this boundary condition) on
suitable function spaces, then
\begin{equation*}\begin{aligned}
&\cS_{\digamma,\Omega_j}\phi=\phi^s_{\digamma,\Omega_j}=\phi-d^s_\digamma\Delta_{\digamma,s,\Omega_j}^{-1}\delta^s_\digamma \phi,\\
&\cP_{\digamma,\Omega_j}\phi=d^s_\digamma Q_{\digamma,\Omega_j} \phi,\qquad Q_{\digamma,\Omega_j} \phi=\Delta_{\digamma,s,\Omega_j}^{-1}\delta^s_\digamma \phi,
\end{aligned}\end{equation*}
are the solenoidal ($\cS$), resp.\ potential ($\cP$) projections of
$\phi$ on $\Omega_j$. Notice that $\cP_{\digamma,\Omega_j}\phi$ is
indeed in the range of $d^s_\digamma$ applied to a function or
one-form vanishing at $\pa_{\inter}\Omega_j$ thanks to the boundary
condition for $\Delta_{\digamma,s}$, which means that
$Q_{\digamma,\Omega_j}$ maps to such functions or tensors. Thus
$\cS_{\digamma,\Omega_j}\phi$ differs from $\phi$ by such a tensor, so
$Ie^{\digamma/x}
f_\digamma=Ie^{\digamma/x}\cS_{\digamma,\Omega_j}f_\digamma$. Further,
$$
\delta^s\cS_{\digamma,\Omega_j}\phi=\delta^s_\digamma\phi-\delta^s_\digamma d^s_\digamma\Delta_{\digamma,s,\Omega_j}^{-1}\delta^s_\digamma \phi=0,
$$
so
$\delta^s_\digamma f_\digamma=0$, i.e.\ the gauge condition we want to
impose is in fact satisfied.

Thus, it remains to check the
invertibility of $\Delta_{\digamma,s}$ with the desired boundary
condition. Before doing this we remark:

\begin{lemma}\label{lemma:potential-elliptic}
For $\digamma>0$, the operator $\Delta_{\digamma,s}=\delta^s_\digamma d^s_\digamma$ is (jointly)
elliptic in $\Diffsc^{2,0}(X)$ on functions.

On the other hand, there exists $\digamma_0>0$ such that for
$\digamma\geq\digamma_0$  the operator $\Delta_{\digamma,s}=\delta^s_\digamma d^s_\digamma$ is (jointly)
elliptic in $\Diffsc^{2,0}(X;\Tsc^*X,\Tsc^*X)$ on one forms. In fact,
on one forms (for all $\digamma>0$)
\begin{equation}\label{eq:symmetric-deriv-exp}
\delta^s_\digamma d^s_\digamma=\frac{1}{2}\nabla_\digamma^*\nabla_\digamma
+\frac{1}{2} d_\digamma\delta_\digamma +A+R,
\end{equation}
where $R\in x\Diffsc^1(X;\Tsc^*X,\Tsc^*X)$,
$A\in\Diffsc^1(X;\Tsc^*X;\Tsc^*X)$ is independent of $\digamma$ and where
$\nabla^\digamma=e^{-\digamma/x}\nabla e^{\digamma/x}$, with $\nabla$
gradient relative to $g_\scl$ (not $g$),
$d_\digamma=e^{-\digamma/x}de^{\digamma/x}$ the exterior derivative on
functions, while $\delta_\digamma$ is its adjoint on one-forms.
\end{lemma}

\begin{proof}
Most of the computations for this lemma have been performed in
Lemma~\ref{lemma:potential-complement}.
In particular, the symbolic
computation is algebraic, and can be
done pointwise, where one arranges that $g_{\scl}$ is as in Lemma~\ref{lemma:potential-complement}.  Since the function case is 
simpler, we consider one-forms. Thus
the full principal symbol of $d^s_\digamma$ (with symmetric 2-tensors considered as a subspace of
2-tensors) is
$$
\begin{pmatrix} \xi+i\digamma&0\\\frac{1}{2}\eta\otimes&\frac{1}{2}(\xi+i\digamma)\\\frac{1}{2}\eta\otimes&\frac{1}{2}(\xi+i\digamma)\\a&\eta\otimes_s\end{pmatrix},
$$
that of $\delta^s_\digamma$ is
$$
\begin{pmatrix}
\xi-i\digamma&\frac{1}{2}\iota_\eta&\frac{1}{2}\iota_\eta&\langle a,.\rangle\\0&\frac{1}{2}(\xi-i\digamma)&\frac{1}{2}(\xi-i\digamma)&\iota^s_\eta\end{pmatrix}.
$$
with the lower right block having $(\ell ij)$ entry given by
$\frac{1}{2}(\eta_i\delta_{\ell j}+\eta_j\delta_{i\ell})$.
Correspondingly, the product, $\Delta^s_\digamma$, has symbol
\begin{equation}\label{eq:Delta-s-symbol}
\begin{pmatrix}
  \xi^2+\digamma^2+\frac{1}{2}|\eta|^2&\frac{1}{2}(\xi+i\digamma)\iota_\eta\\\frac{1}{2}(\xi-i\digamma)\eta\otimes&\frac{1}{2}(\xi^2+\digamma^2)+\iota^s_\eta\eta\otimes_s\end{pmatrix}+\begin{pmatrix}\langle
  a,.\rangle a&\langle
  a,.\rangle\eta\otimes_s\\\iota_\eta^s a&0\end{pmatrix},
\end{equation}
with the lower right block having $\ell k$ entry
$\frac{1}{2}(\xi^2+\digamma^2)\delta_{\ell
  k}+\frac{1}{2}|\eta|^2\delta_{\ell k}+\frac{1}{2}\eta_\ell\eta_k$,
and where we separated out the $a$ terms.

Now ellipticity is easy to see if $a=0$, with a $\digamma$-dependent
lower bound then, and this can be used to absorb the $a$ term by
taking $\digamma>0$ sufficiently large.

To make this more explicit, however, we note that, similarly, the
principal symbol of the gradient relative to $g_{\scl}$ is
$$
\begin{pmatrix} \xi&0\\\eta\otimes&0\\0&\xi\\0&\eta\otimes\end{pmatrix},
$$
with no non-zero entry in the lower left hand corner unlike for the
$g$-gradient in \eqref{eq:g-grad-symbol}, and thus
the
adjoint $\nabla_\digamma^*$ of $\nabla_\digamma$ has principal symbol
$$
\begin{pmatrix}\xi-i\digamma&\iota_\eta&0&0\\0&0&\xi-i\digamma&\iota^s_\eta\end{pmatrix}.
$$
Correspondingly, $\nabla^*_\digamma\nabla_\digamma$ has symbol
\begin{equation}\label{eq:Delta-symbol}
\begin{pmatrix} \xi^2+\digamma^2+|\eta|^2&0\\0&\xi^2+\digamma^2+|\eta|^2\end{pmatrix},
\end{equation}
which is certainly elliptic (including at finite points in
$\Tsc^*_{\pa X} X$!), and indeed is simply $\xi^2+\digamma^2+|\eta|^2$
times the identity matrix. Now, $d=d^s$ going from functions to one-forms
has symbol
$\begin{pmatrix}\xi\\\eta\end{pmatrix}$, so its conjugate
$e^{-\digamma/x}de^{\digamma/x}$ has symbol $\begin{pmatrix}\xi+i\digamma\\\eta\end{pmatrix}$, its adjoint, $\delta_\digamma$ has
symbol $\begin{pmatrix}\xi-i\digamma&\iota_\eta\end{pmatrix}$, and now
$d_\digamma\delta_\digamma$ has symbol
$$
\begin{pmatrix}\xi^2+\digamma^2&(\xi+i\digamma)\iota_\eta\\ (\xi-i\digamma)\eta&\eta\otimes\iota_\eta\end{pmatrix}.
$$
Combining these, we see that the first term in
\eqref{eq:Delta-s-symbol}, i.e.\ in the principal symbol of
$\delta^s_\digamma d^s\digamma$, is the same as
$\frac{1}{2}\nabla^*_\digamma\nabla_\digamma+\frac{1}{2}d_\digamma\delta_\digamma$,
with both terms non-negative, and the first actually positive
definite, with a lower bound $\xi^2+\digamma^2+|\eta|^2$ times the
identity. This proves \eqref{eq:symmetric-deriv-exp}, with the
principal symbol of $A$ given
by the second term in \eqref{eq:Delta-s-symbol}, which is in
particular independent of $\digamma$. Since with $C$ a bound for $a$, the symbol of $A$ is
bounded by $C^2+2C|\eta|\leq C^2(1+\ep^{-1})+\ep|\eta|^2$ for any
$\ep>0$, in particular $\ep<1$, this shows that the principal symbol
of $\delta^s_\digamma d^s\digamma$ is positive definite if
$\digamma>0$ is chosen large enough, completing the proof of the
lemma.
\end{proof}

We now turn to the invertibility question. Let $\Hscd^{m,l}(\Omega_j)$ be the subspace of
$\Hsc^{m,l}(X)$ consisting  of distributions supported in
$\overline{\Omega_j}$, and let $\Hscb^{m,l}(\Omega_j)$ the space of
restrictions of elements of $\Hsc^{m,l}(X)$ to $\Omega_j$. Thus,
$\Hscd^{m,l}(\Omega_j)^*=\Hscb^{-m,-l}(\Omega_j)$. Here we shall be mostly
interested in $m=1$, $l=0$; then at $\pa_{\inter}\Omega_j$, away from
$\pa X$, $\Hscd^{1,0}(\Omega_j)$ is the standard $H^1_0$-space (which is
$\dot H^1$ in H\"ormander's notation, which we adopt), while
$\Hscb^{-1,0}$ is $H^{-1}$ there (which is $\bar H^{-1}$ in
H\"ormander's notation). Further, $\dCI(\Omega_j)$, with the dot
denoting infinite order vanishing at all boundary hypersurfaces, or
indeed $\CI_c(\Omega_j)$ (compact support), are dense in $\Hsc^{1,0}$, so
$\Hscd^{1,0}(\Omega_j)$ is the completion of these spaces in the
$\Hsc^{1,0}(X)$-norm. In addition, the norm on $\Hsc^{1,0}(X)$ is
equivalent to $\|\nabla u\|^2_{L^2}+\|u\|^2_{L^2}$, where the norms
are with respect to any scattering metric, and $\nabla$ is any
differential operator with principal symbol given by $d$, such as the
gradient relative to any (perhaps different from the one giving the
norm) scattering metric. For $L^2_\scl=\Hsc^{0,0}$, or for the weighted
$L^2$-spaces $\Hsc^{0,l}$, the dots and bars do not make any
difference (do not change the space) as usual. Further, the inclusion map $\Hscd^{1,1}\to L^2$ (or indeed
even $\Hscb^{1,1}\to L^2$) is compact. As usual, all these spaces can
be defined for sections of vector bundles, such as
$\Tsc^*_{\Omega_j}X$, by local trivializations. The norm on
$\Hscd^{1,0}(\Omega_j,\Tsc^*\Omega_j)$ is still induced by a gradient
$\nabla$ with respect to any scattering differential operator the same way.

\begin{lemma}\label{lemma:potential-invertible}
The operator on functions $\Delta_{\digamma,s}=\delta^s_\digamma d^s_\digamma$,
considered as a map $\Hscd^{1,0}\to (\Hscd^{1,0})^*=\Hscb^{-1,0}$ is
invertible for all $\digamma>0$.

On the other hand, there exists $\digamma_0>0$ such that for
$\digamma\geq\digamma_0$, the operator
$\Delta_{\digamma,s}=\delta^s_\digamma d^s_\digamma$ on one forms is
invertible.
\end{lemma}

\begin{rem}\label{rem:g-vs-gsc}
The reason for having some $\digamma_0>0$, and requiring $\digamma\geq\digamma_0$, in
the one form case (rather than merely $\digamma>0$) is that $d^s$ is
relative to a standard metric $g$, not a scattering metric. The proof
given below in fact shows that if $d^s$ is replaced by $d^s_{g_\scl}$,
relative to any scattering metric $g_{\scl}$, then one may simply
assume $\digamma>0$.
\end{rem}

\begin{proof}
The following considerations apply to both the function case and the
one-form case.
Relative to the scattering metric with respect to which $\delta^s$ is
defined, the quadratic form of $\Delta_{\digamma,s}$ is
$\langle\Delta_{\digamma,s} u,v\rangle=\langle d^s_\digamma u,
d^s_\digamma v\rangle$. So in particular
$$
\| d^s_\digamma u\|^2_{L^2}\leq \|\Delta_{\digamma,s}
u\|_{\Hscb^{-1,0}}\|u\|_{\Hscd^{1,0}}\leq \ep^{-1}\|\Delta_{\digamma,s}
u\|_{\Hscb^{-1,0}}^2+\ep\| u\|_{\Hscd^{1,0}}^2.
$$
Correspondingly, if one has an estimate
\begin{equation}\label{eq:invertible-est}
\| u\|_{\Hscd^{1,0}}\leq C\|d^s_\digamma u\|_{L^2},
\end{equation}
or equivalently (for a different $C$)
$$
\|\nabla u\|_{L^2}+\|u\|_{L^2}\leq C\|d^s_\digamma u\|_{L^2},
$$
then for small $\ep>0$, one can absorb $\ep\| u\|_{\Hscd^{1,0}}^2$
into the left hand side above, giving
$$
\|u\|_{\Hscd^{1,0}}\leq C\|d^s_\digamma u\|_{L^2}\leq C' \|\Delta_{\digamma,s}
u\|_{\Hscb^{-1,0}}.
$$
This in turn gives
invertibility in the sense
discussed in the statement of the theorem since $\Delta_{\digamma,s}$
is formally (and as this shows, actually) self-adjoint, so one has the
same estimates for the formal adjoint.

On the other hand, if one has an estimate
\begin{equation}\label{eq:Fredholm-est}
\| u\|_{\Hscd^{1,0}}\leq C\|d^s_\digamma u\|_{L^2}+C\|u\|_{\Hscd^{0,-1}},
\end{equation}
or equivalently
$$
\|\nabla u\|_{L^2}+\|u\|_{L^2}\leq C\|d^s_\digamma u\|_{L^2}+C\|u\|_{\Hscd^{0,-1}},
$$
then for $\ep>0$ small one gets
$$
\|u\|_{\Hscd^{1,0}}\leq C\|d^s_\digamma u\|_{L^2}+C\|u\|_{\Hscd^{0,1}}\leq C' \|\Delta_{\digamma,s}
u\|_{\Hscb^{-1,0}}+C'\|u\|_{\Hscd^{0,-1}}.
$$
Again, by formal self-adjointness, one gets the same statement for the
adjoint, which implies that
$\Delta_{\digamma,s}$  is Fredholm (by virtue of the compactness of the inclusion $\Hscd^{1,0}\to\Hscd^{0,-1}$), and further that the invertibility is
equivalent to the lack of kernel on $\Hscd^{1,0}$ (since the cokernel
statement follows by formal self-adjointness). Note that \eqref{eq:Fredholm-est}
follows quite easily from Lemma~\ref{lemma:potential-elliptic} (and is
standard on functions as $d^s=\nabla$ then), in the form case using
the Dirichlet boundary condition to apply
\eqref{eq:symmetric-deriv-exp} to $u$ and pair with $u$ but we discuss
invertibility, taking advantage of
Lemma~\ref{lemma:potential-elliptic} later.

Now, on functions, $d^s=\nabla$, and as $d^s_\digamma$ differs from
$d^s$ by a $0$th order operator, $\|\nabla u\|_{L^2}\leq C\|d^s_\digamma
u\|_{L^2}+C\|u\|_{L^2}$ automatically. In particular,
\eqref{eq:invertible-est} follows if one shows $\|u\|_{L^2}\leq C\|d^s_\digamma
u\|_{L^2}$ for $u\in\Hscd^{1,0}$, or equivalently (by density) for
$u\in\CI_c(\Omega_j)$, which is a Poincar\'e inequality.

To prove this Poincar\'e inequality, notice that $\|e^{-\digamma/x}(x^2 D_x)e^{\digamma/x}
u\|_{L^2}\leq C\|d^s_\digamma
u\|_{L^2}$ certainly, so it suffices to estimate $\|u\|_{L^2}$ in
terms of the $L^2$ norm of
$$
e^{-\digamma/x}(x^2 D_x)e^{\digamma/x}
u=(x^2D_x+i\digamma)u.
$$
But for any operator $P$, writing
$P_R=(P+P^*)/2$ and $P_I=(P-P^*)/(2i)$ for the symmetric and
skew-symmetric parts,
$$
\|Pu\|^2=\|P_R u\|^2+\|P_I u\|^2+\langle i[P_R,P_I]u,u\rangle.
$$
It is convenient here to use a metric
$\frac{dx^2}{x^4}+\frac{h}{x^2}$ where $h$ is a metric, independent of
$x$, on the level sets of $x$, using some product decomposition. For
then the metric density is $x^{-(n+1)}\,|dx|\,|dh|$, so with
$P=x^2D_x+i\digamma$,
$P^*=x^2D_x+i(n-1)x-i\digamma$, so
$$
P_R=x^2D_x+i\frac{n-1}{2} x,\ P_I=\digamma-\frac{n-1}{2} x,\ i[P_R,P_I]=i
x^2\frac{n-1}{2},
$$
we have
\begin{equation*}\begin{aligned}
\|(x^2
D_x+i\digamma)u\|^2_{L^2}&=\Big\|\Big(x^2D_x+i\frac{n-1}{2}\Big)u\Big\|^2_{L^2}\\
&\qquad\qquad+\Big\|\Big(\digamma-\frac{n-1}{2}
x\Big) u\Big\|^2_{L^2}-\frac{n-1}{2}\langle x^2 u,u\rangle\\
&=\Big\|\Big(x^2D_x+i\frac{n-1}{2}\Big)u\Big\|^2_{L^2}\\
&\qquad\qquad+\Big\langle \Big((\digamma-\frac{n-1}{2}
x)^2-\frac{n-1}{2} x^2\Big)u,u\Big\rangle.
\end{aligned}\end{equation*}
Now, if $\Omega_j\subset\{x\leq x_0\}$, as long as $x_0>0$ is
sufficiently small so that
$$
\Big(\digamma-\frac{n-1}{2}
x\Big)^2-\frac{n-1}{2} x^2
$$
is positive (and thus bounded below by a
positive constant) on $[0,x_0]$, which is automatic for sufficiently
small $x_0$, or indeed for bounded $x_0$ and sufficiently large
$\digamma$, one obtains that
$\|u\|_{L^2}^2\leq C \|(x^2
D_x+i\digamma)u\|^2_{L^2}$, and thus in summary that
$$
\|u\|_{L^2}\leq C\|d^s_\digamma u\|_{L^2},
$$
as desired. This proves the lemma for functions, at least in the case
of sufficiently small $x_0$.

This actually suffices for our
application, but in fact one can do better by noting that in fact even
in general this gives us the estimate
$$
\|u\|_{L^2}\leq C\|d^s_{\digamma} u\|+C\|u\|_{L^2(\{x_1\leq x\leq x_0\})}
$$
for suitable small $x_1>0$. But by the standard Poincar\'e inequality,
using the vanishing at $x=x_0$,
one can estimate the last term in terms of $C'\|d^s_\digamma u\|$,
which gives the general conclusion for functions. Here, to place us
properly in the standard Poincar\'e setting, we note that with
$\phi=e^{\digamma/x} u$, the last required estimate is equivalent to the weighted estimate
$\|e^{-\digamma/x}\phi\|_{L^2(\{x_1\leq x\leq x_0\})}\leq
C\|e^{-\digamma/x} d\phi\|_{L^2(\{x_1\leq x\leq x_0\})}$, and now the
weights are bounded, so can be dropped completely.

It remains to deal with one-forms. For this we use that
\eqref{eq:symmetric-deriv-exp} and \eqref{eq:Delta-symbol} give that
\begin{equation}\label{eq:symmetric-deriv-mod}
\delta^s_\digamma d^s_\digamma=\frac{1}{2}\nabla^*\nabla+\frac{1}{2}\digamma^2
+\frac{1}{2} d_\digamma\delta_\digamma +A+\tilde R,
\end{equation}
where $A\in\Diffsc^1(X)$ is independent of $\digamma$ and $\tilde R\in
x\Diffsc^1(X)$; this follows by rewriting
$\nabla_\digamma^*\nabla_\digamma$ using \eqref{eq:Delta-symbol},
which modifies $R$ in \eqref{eq:symmetric-deriv-exp} to give \eqref{eq:symmetric-deriv-mod}.
Thus, in
fact
\begin{equation}\label{eq:symmetric-deriv-identity}
\|d^s_\digamma u\|^2=\frac{1}{2}\|\nabla 
u\|^2+\frac{1}{2}\digamma^2\|u\|^2+\frac{1}{2}\|\delta^s_\digamma
u\|^2+\langle Au,u\rangle+\langle \tilde Ru,u\rangle.
\end{equation}
Since $A\in\Diffsc^1(X)$, $|\langle Au,u\rangle|\leq
C\|u\|_{\Hscd^{1,0}}\|u\|_{L^2}$, and there is a similar estimate for
the last term.
This gives an estimate, for sufficiently large $\digamma$,
\begin{equation}\label{eq:nabla-in-terms-of-ds}
\|\nabla  u\|^2+\digamma^2\|u\|^2\leq C\|d^s_\digamma u\|^2+C\| x^{1/2} u\|^2,
\end{equation}
with the constant $C$ on the right hand side depending on $\digamma$,
and thus
$$
\langle (1-Cx)u,u\rangle\leq C\|d^s_\digamma u\|^2.
$$
Again, if $x_0$ is sufficiently small, this gives
$$
\|u\|\leq C\|d^s_\digamma u\|,
$$
and thus the invertibility, while if $x_0$ is larger, this still gives
$$
\|u\|_{L^2}\leq C\|d^s_\digamma u\|_{L^2}+C\|u\|_{L^2(\{x_1\leq x\leq x_0\})}.
$$
One can then finish the proof as above, using the standard Poincar\'e
inequality for one forms, see \cite[Section~6,
Equation~(28)]{SU-Duke}.
\end{proof}

A slight modification of the argument gives:

\begin{lemma}\label{lemma:potential-invertible-weight}
The operator on functions $\Delta_{\digamma,s}=\delta^s_\digamma d^s_\digamma$,
considered as a map $\Hscd^{1,r}\to \Hscb^{-1,r}$ is
invertible for all $\digamma>0$ and all $r\in\RR$.

On the other hand, there exists $\digamma_0>0$ such that for
$\digamma\geq\digamma_0$, the operator
$\Delta_{\digamma,s}=\delta^s_\digamma d^s_\digamma$ on one forms is
invertible as a map $\Hscd^{1,r}\to \Hscb^{-1,r}$ for all $r\in\RR$.
\end{lemma}

\begin{proof}
Since the function case is completely analogous, we consider one forms
to be definite. Also note that (full) elliptic regularity would
automatically give this result if not for $\pa_\inter\Omega_j$.

An isomorphism estimate
$\Delta_{\digamma,s}:\Hscd^{1,r}\to \Hscb^{-1,r}$ is equivalent to an
isomorphism estimate
$x^{-r}\Delta_{\digamma,s}x^r:\Hscd^{1,0}\to \Hscb^{-1,0}$. But the
operator on the left is $\Delta_{\digamma,s}+F$, where $F\in
x\Diffsc^1$. Thus, $x^{-r}\Delta_{\digamma,s}x^r$ is of the form
\eqref{eq:symmetric-deriv-mod}, with only $\tilde R$ changed. The rest
of the proof then immediately goes through.
\end{proof}

Before proceeding with the analysis of the Dirichlet Laplacian, we
first discuss the analogue of Korn's inequality that will be useful
later.

\begin{lemma}\label{lemma:Korn}
Suppose $\Omega_j$ is a domain in $X$ as above. For
$\digamma>0$ and $r\in\RR$,
$$
\|u\|_{\Hscb^{1,r}(\Omega_j)}\leq C(\|x^{-r}d^s_\digamma
u\|_{L^2_\scl(\Omega_j)}+\|u\|_{x^{-r} L^2_\scl(\Omega_j)}).
$$
for one-forms $u\in\Hscb^{1,r}(\Omega_j)$.
\end{lemma}

\begin{proof}
First note that if one lets $\tilde u=x^{-r} u$, then
$\|u\|_{\Hscb^{1,r}(\Omega_j)}$ is equivalent to $\|\tilde
u\|_{\Hscb^{1,0}(\Omega_j)}$, and $\|x^{-r}d^s_\digamma
u\|_{L^2_\scl(\Omega_j)}+\|u\|_{x^{-r} L^2_\scl(\Omega_j)}$ is
equivalent to $\|d^s_\digamma
\tilde u\|_{L^2_\scl(\Omega_j)}+\|\tilde u\|_{L^2_\scl(\Omega_j)}$
since the commutator term through $d^s_\digamma$ can be absorbed into
a sufficiently large multiple of $\|u\|_{x^{-r}
  L^2_\scl(\Omega_j)}=\|\tilde u\|_{L^2_\scl(\Omega_j)}$. Thus, one is
reduced to proving the case $r=0$.

Let $\tilde\Omega_j$ be a domain in $X$ with $\CI$ boundary,
transversal to $\pa X$, containing
$\overline{\Omega_j}$. We claim that there is a continuous extension
map $E:\Hscb^{1,0}(\Omega_j)\to \Hscd^{1,0}(\tilde\Omega_j)$ such that
\begin{equation}\label{eq:E-ds}
\|d^s_\digamma
Eu\|_{L^2_\scl(\tilde\Omega_j)}+\|Eu\|_{L^2_\scl(\tilde\Omega_j)}\leq
C(\|d^s_\digamma
u\|_{L^2_\scl(\Omega_j)}+\|u\|_{L^2_\scl(\Omega_j)}),\qquad u\in\Hscb^{1,0}(\Omega_j),
\end{equation}
i.e.\ $Eu$ is also continuous when on both sides the gradient is
replaced by the symmetric gradient in the definition of an $H^1$-type space.
Once this is proved, the lemma can be shown in the following
manner. By
\eqref{eq:symmetric-deriv-exp} of
Lemma~\ref{lemma:potential-elliptic} any
$v\in\Hscd^{1,0}(\tilde\Omega_j)$, in particular $v=Eu$,
satisfies, for any $\ep>0$,
\begin{equation*}\begin{aligned}
\|\nabla
v\|_{L^2_\scl(\tilde\Omega_j)}^2+\|v\|_{L^2_\scl(\tilde\Omega_j)}^2&\leq
2\|d^s_\digamma
v\|_{L^2_\scl(\tilde\Omega_j)}^2+\|v\|_{L^2_\scl(\tilde\Omega_j)}^2+C\|v\|_{L^2_\scl(\tilde\Omega_j)}\|v\|_{\Hscb^{1,0}(\tilde\Omega_j)}\\
&\leq 2\|d^s_\digamma
v\|_{L^2_\scl(\tilde\Omega_j)}^2+C'\|v\|_{L^2_\scl(\tilde\Omega_j)}^2+\ep\|v\|_{\Hscb^{1,0}(\tilde\Omega_j)}^2
\end{aligned}\end{equation*}
and now for $\ep>0$ small, the last term on the right hand side can be
absorbed into the left hand side. Using this with $v=Eu$, noting that
$E$ is an extension map so
$$
\|\nabla
u\|_{L^2_\scl(\Omega_j)}^2+\|u\|_{L^2_\scl(\Omega_j)}^2\leq \|\nabla E
u\|_{L^2_\scl(\tilde\Omega_j)}^2+\|Eu\|_{L^2_\scl(\tilde\Omega_j)}^2,
$$
we deduce, using \eqref{eq:E-ds} in the last step, that
$$
\|u\|_{\Hscb^{1,0}(\Omega_j)}\leq C(\|d^s_\digamma
Eu\|_{L^2_\scl(\tilde\Omega_j)}+\|Eu\|_{L^2_\scl(\tilde\Omega_j)})\leq C'(\|d^s_\digamma
u\|_{L^2_\scl(\Omega_j)}+\|u\|_{L^2_\scl(\Omega_j)}),
$$
completing the proof of the lemma.

Thus, it remains to construct $E$. By a partition of unity, this can
be reduced to a local extension, local on $X$. Since $\pa\Omega_j$ is
transversal to $\pa X$, near points on $\pa X\cap\pa\Omega_j$ one can
arrange that locally (in a model in which a neighborhood of $p$ is
identified with an open set in $\overline{\RR^n}$) $\pa\Omega_j$ is
the hypersurface $x_n=0$, $\Omega_j$ is $x_n>0$; the analogous
arrangement can also be made away from $\pa X$ near points on
$\pa\Omega_j$. Since $\Hsc^{s,r}(X)$, $\Hscd^{s,r}(\tilde\Omega_j)$,
$\Hscb^{s,r}(\Omega_j)$, are locally, and also for compactly supported
elements in the chart, are preserved by local diffeomorphisms of $X$
to $\overline{\RR^n}$ in the sense that $X$ is replaced by
$\overline{\RR^n}$, $\overline{\Omega_j}$ by $\overline{\RR^n_+}$ (by virtue of these spaces are well defined on
manifolds with boundary, without additional information on metrics,
etc., up to equivalence of norms), it suffices to prove that there is
a local
extension map $E_1$ that has the desired properties.

Let $\Phi_k(x',x_n)=(x',-kx_n)$ for $x_n<0$, and consider a variation
of the standard
construction of an $H^1(\RR^n_+)$ extension map on one-forms as
follows. (Note that the usual extension map is given by trivialization
of a bundle, in this case using $dx_j$ as a local basis of sections,
and extending the coefficients using the extension map on functions.) Let $E_1$ given by
$$
(E_1 \sum_j u_j\,dx_j)(x',x_n)=\sum_{k=1}^3 c_k\Phi_k^* (\sum
u_j\,dx_j),\ x_n<0,
$$
and
$$
(E_1 \sum_j u_j\,dx_j)(x',x_n)=\sum
u_j\,dx_j,\ x_n\geq 0,
$$
with $c_k$ chosen so that $E_1:C^1(\overline{\RR^n_+})\to
C^1(\RR^n)$. We can achieve this mapping property as follows. We have,
with $\pa_j$ acting as derivatives on the components, or equivalently
but invariantly
as Lie derivatives in this case,
\begin{equation*}\begin{aligned}
\Phi_k^*u_j\,dx_j&=u_j(x',-kx_n)\,dx_j,\ j\neq n,\\
\Phi_k^*u_n\,dx_n&=-k
u_n(x',-kx_n)\,dx_n,\\
\pa_i\Phi_k^*u_j\,dx_j&=(\pa_i
u_j)(x',-kx_n)\,dx_j,\ i,j\neq n,\\
\pa_i\Phi_k^*u_n\,dx_n&=-k(\pa_i
u_n)(x',-kx_n)\,dx_n,\ i\neq n,\\
\pa_n\Phi_k^*u_j\,dx_j&=-k(\pa_n
u_j)(x',-kx_n)\,dx_j,\ j\neq n,\\
\pa_n\Phi_k^*u_n\,dx_n&=k^2(\pa_n
u_j)(x',-kx_n)\,dx_n,
\end{aligned}\end{equation*}
so the requirements for matching the derivatives at $x_n=0$, which
gives the $C^1$ property, are, for $j\neq n$,
\begin{equation*}\begin{aligned}
c_1+c_2+c_3&=1,\\
-c_1-2c_2-3c_3&=1,
\end{aligned}\end{equation*}
while for $j=n$
\begin{equation*}\begin{aligned}
-c_1-2c_2-3c_3&=1,\\
c_1+4c_2+9c_3&=1,
\end{aligned}\end{equation*}
which gives a 3-by-3 system
$$
\begin{pmatrix}1&1&1\\-1&-2&-3\\1&4&9\end{pmatrix}\begin{pmatrix}c_1\\c_2\\c_3\end{pmatrix}=\begin{pmatrix}1\\1\\1\end{pmatrix}.
$$
The matrix on the right is a Vandermonde matrix, and is thus
invertible, so one can find $c_k$ with the desired properties.
With this, $E_1:C^1_c(\overline{\RR^n_+})\to C^1_c(\RR^n)$
has the property that $\|E_1 u\|_{H^1(\RR^n)}\leq
C\|u\|_{H^1(\RR^n_+)}$, since each term in the definition of
  $E_1$ has derivatives $\pa_i$ satisfying $\|\pa_i\Phi_k^*u\|_{L^2(\RR^n)}\leq
C\|\pa_i u\|_{L^2(\RR^n_+)}$, and since $E_1 u\in
  C^1_c(\RR^n)$ assures that the distributional derivative satisfies
  $\pa_i E_1 u\in L^2(\RR^n)$, whose square norm can be calculated as
  the sum of the squared norms over $\RR^n_+=\{x_n>0\}$ and  $\RR^n_-=\{x_n<0\}$.
Correspondingly, $E_1$ extends continuously, in a unique manner, to a map
$H^1(\RR^n_+)\to H^1(\RR^n)$.

Before proceeding we note that with this choice of coefficients, $E_1$
defined as the analogous map on functions, is actually the standard
$H^2$ extension map. However, on one-forms the same choice, defined in
terms of pull-backs, i.e.\ natural operations, as above, rather than
trivializing the form bundle, does not extend continuously to
$H^2$. On the other hand, if one trivializes the bundle and uses the
$H^2$ extension map, one does not have the desired property
\eqref{eq:E-ds} for symmetric differentials, a property that we check
below with our choice of extension map.

Notice that, with $\Phi_k^*$ acting on 2-tensors as usual, for all $i,j$,
\begin{equation*}\begin{aligned}
dx_i\otimes(\pa_i \Phi_k^*u_j\,dx_j)+(\pa_j \Phi_k^*u_i\,dx_i)\otimes dx_j&=\Phi_k^*((\pa_i
u_j+\pa_j u_i)dx_i\otimes dx_j),
\end{aligned}\end{equation*}
as follows from a direct calculation, or indeed from the naturality of the symmetric gradient $d^s=d^s_{g_0}$ for
a translation invariant Riemannian metric $g_0$: the
two sides are the $ij$ component of $2d^s \Phi_k^*$, resp.\ $2\Phi_k^*
d^s$, as for such a metric the symmetric gradient is actually
independent of the choice of the metric (in this class). Since, summed
over $i,j$, the
left hand side is the symmetric gradient of $\Phi_k^*\sum u_j\,dx_j$
in $x_n<0$,
while the right hand side is the pull-back of the symmetric gradient
from $x_n>0$, this shows that
$$
\|d^s\Phi_k^* u\|_{L^2(\RR^n_-)}\leq C\|d^s u\|_{L^2(\RR^n_+)}.
$$
This proves that one has
$$
\|d^s E_1 u\|_{L^2(\RR^n)}\leq C\|d^s u\|_{L^2(\RR^n_+)}.
$$

Now, using a partition of unity $\{\rho_k\}$ to localize on $\Omega_j$, as
mentioned above, this gives a global extension map from
$H^1(\Omega_j)$: $\sum \psi_k E_{1,k}\rho_k$, where $\psi_k$ is
identically $1$ near $\supp\rho_k$. While $d^s$ depends on the choice
of a metric, the dependence is via the $0$th order term, i.e.\ one has
$d^s_g u=d^s_{g_0} u+Ru$ for an appropriate $0$th order $R$. Using the
Euclidean metric in the local model, this shows that
$$
\|d^s_g \psi_k E_{1,k} \rho_k  u\|_{L^2(\RR^n)}\leq C(\|d^s_{g_0}
\rho_k u\|_{L^2(\RR^n_+)}+\|\rho_k u\|_{L^2(\RR^n_+)}).
$$
Since $d^s_\digamma$ differs from $d^s$ by a $0$th order term, one can
absorb this in the $L^2$ norm (using also the continuity of the
extension map from $L^2$ to $L^2$):
$$
\|d^s_{g,\digamma}\psi_k E_{1,k} \rho_k  u\|_{L^2(\RR^n)}\leq C(\|d^s_{g_0,\digamma}
\rho_k u\|_{L^2(\RR^n_+)}+\|\rho_k u\|_{L^2(\RR^n_+)}).
$$
Summing over $k$ proves \eqref{eq:E-ds}, and thus the lemma.
\end{proof}

We now return to the analysis of the Dirichlet Laplacian.

\begin{cor}\label{cor:potential-psdo}
Let
$\phi\in\CI_c(\overline{\Omega_j}\setminus\pa_\inter\Omega_j)$. Then
on functions, for $\digamma>0$, $k\in\RR$,
the operator $\phi\Delta_{\digamma,s}^{-1}\phi:\Hscb^{-1,k}\to\Hscd^{1,k}$ is
in $\Psisc^{-2,0}(X)$. There is $\digamma_0>0$ such that the analogous
conclusion holds for one forms for $\digamma\geq\digamma_0$.
\end{cor}

\begin{proof}
This follows from the usual parametrix identity. Namely, by
Lemma~\ref{lemma:potential-elliptic}, $\Delta_{\digamma,s}$ has a
parametrix $B\in\Psisc^{-2,0}(X)$ so that
$$
B\Delta_{\digamma,s}=\Id+F_L,\ \Delta_{\digamma,s}B=\Id+F_R,
$$
with
$F_L,F_R\in\Psisc^{-\infty,-\infty}(X)$.
Let $\psi\in
\CI_c(\overline{\Omega_j}\setminus\pa_\inter\Omega_j)$ be identically
$1$ on $\supp\phi$. 
Thus,
$$
\psi=B\Delta_{\digamma,s}\psi-F_L\psi=B\psi\Delta_{\digamma,s}+B[\Delta_{\digamma,s},\psi]-F_L\psi
$$
and
$$
\psi=\psi\Delta_{\digamma,s}B-\psi F_R=\Delta_{\digamma,s}\psi
B+[\psi,\Delta_{\digamma,s}]B-\psi F_R.
$$
Then
\begin{equation*}\begin{aligned}
\psi\Delta_{\digamma,s}^{-1}\psi&=\psi\Delta_{\digamma,s}^{-1}\Delta_{\digamma,s}\psi B
+\psi\Delta_{\digamma,s}^{-1}([\psi,\Delta_{\digamma,s}]B-\psi
F_R)\\
&=\psi^2 B+B\psi\Delta_{\digamma,s}\Delta_{\digamma,s}^{-1}([\psi,\Delta_{\digamma,s}]B-\psi
F_R)\\
&\qquad\qquad+(B[\Delta_{\digamma,s},\psi]-F_L\psi)
\Delta_{\digamma,s}^{-1}([\psi,\Delta_{\digamma,s}]B-\psi F_R)\\
&=\psi^2 B+B\psi([\psi,\Delta_{\digamma,s}]B-\psi
F_R)\\
&\qquad\qquad+(B[\Delta_{\digamma,s},\psi]-F_L\psi)
\Delta_{\digamma,s}^{-1}([\psi,\Delta_{\digamma,s}]B-\psi F_R).
\end{aligned}\end{equation*}
Multiplying from both the left and the right by $\phi$ gives
\begin{equation*}\begin{aligned}
\phi\Delta_{\digamma,s}^{-1}\phi&=\phi B\phi +\phi B\psi([\psi,\Delta_{\digamma,s}]B-\psi
F_R)\phi\\
&\qquad\qquad+\phi (B[\Delta_{\digamma,s},\psi]-F_L\psi)
\Delta_{\digamma,s}^{-1}([\psi,\Delta_{\digamma,s}]B-\psi F_R)\phi.
\end{aligned}\end{equation*}
Now, the first two terms on the right hand side are in
$\Psisc^{-2,0}$, resp.\ $\Psisc^{-\infty,-\infty}$, in the latter case
using the disjointness of $\supp d\psi$ and $\phi$ for
$[\psi,\Delta_{\digamma,s}]B\phi$, resp.\ that
$F_L\in\Psisc^{-\infty,-\infty}$ for $\psi
F_R\phi$. For this reason, $([\psi,\Delta_{\digamma,s}]B-\psi
F_R)\phi$ and $\phi (B[\Delta_{\digamma,s},\psi]-F_L\psi)$ are
smoothing, in the sense that they map $\Hsc^{s,r}(X)$ to
$\Hsc^{s',r'}(X)$ for any $s',r',s,r$, and they also have support so
that they map into functions supported in
$\overline{\Omega_j}\setminus\pa_\inter\Omega_j$, and they also can be
applied to
functions on $\overline{\Omega_j}$. As
$\Delta_{\digamma,s}^{-1}$ is continuous
$\Hscb^{-1,k}(\Omega_j)\to\Hscd^{1,k}(\Omega_j)$, this shows that the
last term is continuous from $\Hsc^{s,r}(X)$ to
$\Hsc^{s',r'}(X)$ for any $s',r',s,r$, which means that it has a
Schwartz (rapidly decaying with all derivatives) Schwartz kernel,
i.e.\ it is in $\Psisc^{-\infty,-\infty}(X)$. This completes the proof.
\end{proof}

\begin{cor}\label{cor:potential-loc}
Let $\phi\in\CI_c(\overline{\Omega_j}\setminus\pa_\inter\Omega_j)$,
$\chi\in\CI(\overline{\Omega_j})$ with disjoint support and with
$\chi$ constant near $\pa_{\inter}\Omega_j$. Let $\digamma$,
$\digamma_0$ as in Corollary~\ref{cor:potential-psdo}. Then the
operator $\chi\Delta_{\digamma,s}^{-1}\phi:\Hscb^{-1,k}(\Omega_j)\to\Hscd^{1,k}(\Omega_j)$
in fact maps $\Hsc^{s,r}(X)\to\Hscd^{1,k}(\Omega_j)$ for all $s,r,k$.

Similarly,
$\phi\Delta_{\digamma,s}^{-1}\chi:\Hscb^{-1,k}(\Omega_j)\to\Hscd^{1,k}(\Omega_j)$
in fact maps $\Hscb^{-1,k}(\Omega_j)\to\Hsc^{s,r}(X)$ for all $s,r,k$.
\end{cor}

\begin{proof}
Since the second statement follows by duality, it suffices to prove
the first.

As $\chi\phi=0$, we can write
$$
\chi\Delta_{\digamma,s}^{-1}\phi=[\chi,\Delta_{\digamma,s}^{-1}]\phi=\Delta_{\digamma,s}^{-1}[\Delta_{\digamma,s},\chi]\Delta_{\digamma,s}^{-1}\phi.
$$
By Corollary~\ref{cor:potential-psdo},
$[\Delta_{\digamma,s},\chi]\Delta_{\digamma,s}^{-1}\phi\in\Psisc^{-\infty,\infty}(X)$
since it is in $\Psisc^{-1,0}(X)$ (this uses $\supp d\chi$ disjoint
from $\pa_{\inter}\Omega_j$) but $d\chi$ and $\phi$ have disjoint
supports. Thus, it maps $\Hsc^{s,r}(X)\to\Hsc^{-1,k}(X)$, and thus, in
view of $\supp d\chi$, to $\Hscb^{-1,k}(\Omega_j)$, giving the conclusion.
\end{proof}

\begin{cor}\label{cor:solenoidal}
Let $\phi\in\CI_c(\overline{\Omega_j}\setminus\pa_\inter\Omega_j)$,
$\chi\in\CI(\overline{\Omega_j})$ with disjoint support and with
$\chi$ constant near $\pa_{\inter}\Omega_j$. Let $\digamma$,
$\digamma_0$ as in Corollary~\ref{cor:potential-psdo}.

Then $\phi\cS_{\digamma,\Omega_j}\phi\in\Psisc^{0,0}(X)$, while
$\chi\cS_{\digamma,\Omega_j}\phi:\Hsc^{s,r}(X)\to x^k L^2_\scl(\Omega_j)$ and
$\phi\cS_{\digamma,\Omega_j}\chi:x^k L^2_\scl(\Omega_j)\to \Hsc^{s,r}(X)$ for
all $s,r,k$.
\end{cor}

\begin{proof}
This is immediate from
$\cS_{\digamma,\Omega_j}=\Id-d^s_\digamma\Delta_{\digamma,s,\Omega_j}^{-1}\delta^s_\digamma$
and the above results concerning $\Delta_{\digamma,s,\Omega_j}^{-1}$,
using that $d^s_\digamma$ and $\delta^s_\digamma$ are differential
operators, and thus preserve supports.
\end{proof}

We also need the Poisson operator associated to 
$\pa_{\inter}\Omega_j$. First note that if $H$ is a (codimension $1$) hypersurface in
$\overline{\Omega_j}$ which intersects $\pa\Omega_j$ away from
$\pa_\inter\Omega_j$, and does so transversally, then the restriction
map
$$
\gamma_H:\dCI(\overline{\Omega_j})\to\dCI(H),
$$
with the dots denoting infinite order vanishing at $\pa\Omega_j$,
resp.\ $\pa H$, as usual, in fact maps, for $s>1/2$,
\begin{equation}\label{eq:restrict-Sobolev}
\gamma_H:\Hsc^{s,r}(\Omega_j)\to\Hsc^{s-1/2,r}(H)
\end{equation}
continuously.
This can be easily seen since the restriction map is local, and
locally in $\overline{\Omega_j}$, one can map a neighborhood of $p\in
\pa H$ to a neighborhood of a point $p'\in\pa\RR^{n-1}$ in
$\overline{\RR^n}$ by a diffeomorphism so that $H$ is mapped to
$\overline{\RR^{n-1}}$, and thus by the diffeomorphism invariance of
the spaces under discussion, the standard $\RR^n$ result with the
usual Sobolev spaces $H^s(\RR^n)=\Hsc^{s,0}(\overline{\RR^n})$, using
that weights commute with the restriction, gives
\eqref{eq:restrict-Sobolev}.
The same argument also shows that there is a continuous extension map
\begin{equation}\label{eq:extend-Sobolev}
e_H:\Hsc^{s-1/2,r}(H)\to\Hsc^{s,r}(\Omega_j),\qquad \gamma_H e_H=\Id,
\end{equation}
since the analogous result on $\RR^n$ is standard, and one can
localize by multiplying by cutoffs without destroying the desired
properties.

Considering $\overline{\Omega_j}$ inside a larger domain
$\overline{\Omega'}$, with $\pa_\inter\Omega_j$ satisfying the
assumptions for $H$, we have a continuous extension map
$\Hscb^{s,r}(\Omega_j)\to\Hscb^{s,r}(\Omega')$ by local reduction to
$\overline{\RR^n}$. Correspondingly, we also obtain restriction and extension maps
$$
\gamma_{\pa_\inter\Omega_j}:\Hscb^{s,r}(\Omega_j)\to\Hsc^{s-1/2,r}(\pa_\inter\Omega_j),\qquad e_{\pa_\inter\Omega_j}:\Hsc^{s-1/2,r}(\pa_\inter\Omega_j)\to \Hscb^{s,r}(\Omega_j).
$$
With this background we have:

\begin{lemma}\label{lemma:Poisson}
Let $\digamma$,
$\digamma_0$ as in Corollary~\ref{cor:potential-psdo}, and let $k\in\RR$.

For $\psi\in \Hsc^{1/2,k}(\pa_{\inter}\Omega_j)$ there is a unique
$u\in\Hscb^{1,k}(\Omega_j)$ such that $\Delta_{\digamma,s}u=0$,
$\gamma_{\pa_{\inter}\Omega_j} u=\psi$.

This defines the Poisson operator
$B_{\Omega_j}:\Hsc^{1/2,k}(\pa_{\inter}\Omega_j)\to\Hscb^{1,k}(\Omega_j)$
solving
$$
\Delta_{\digamma,s} B_{\Omega_j}=0,\ \gamma_{\pa_{\inter}\Omega_j} B_{\Omega_j}=\Id,
$$
which has the property that, for $s>1/2$, and for
$\phi\in\CI(\overline{\Omega_j})$ supported away from $\pa_\inter\Omega_j$,
$\phi B_{\Omega_j}:\Hsc^{s-1/2,r}(\pa_\inter\Omega_j)\to\Hsc^{s,r}(\Omega_j)$.
\end{lemma}

\begin{proof}
The uniqueness follows from the unique solvability of the
Dirichlet problem with vanishing boundary conditions, as we already
discussed, while the existence by taking
$u=e_{\pa_\inter\Omega_j}\psi-\Delta_{\digamma,s}^{-1}\Delta_{\digamma,s}e_{\pa_\inter\Omega_j}\psi$,
where $\Delta_{\digamma,s}^{-1}$ is, as before, the inverse of the
operator with vanishing Dirichlet boundary conditions.
The mapping property also follows from this explicit description, the
mapping properties of $e_{\pa_\inter\Omega_j}$ as well as
Corollary~\ref{cor:potential-loc}, since one can arrange that
$e_{\pa_\inter\Omega_j}$ maps to distributions supported away from $\supp\phi$.
\end{proof}

Let $\Omega_2$ be a larger neighborhood of $\Omega$; all of our
constructions take place in $\Omega_2$. Let
$\tilde\Omega_j=\overline{\Omega_j}\setminus\pa_{\inter}\Omega_j$ (so the
artificial boundary is included, but not the interior one).
Let $G$ be a parametrix for $A_\digamma$ in $\Omega_2$; it is thus a
scattering ps.d.o. with Schwartz kernel compactly supported in
$\tilde\Omega_2\times\tilde\Omega_2$. Then $GA_\digamma=I+E$, where
$\WFsc'(E)$ is disjoint from a neighborhood $\Omega_1$ (compactly
contained in $\Omega_2$) of the original region $\Omega$, and
$E=-\Id$ near $\pa_{\inter}\Omega_2$.
Now one has
$$
G(N_\digamma+d^s_\digamma M\delta^s_\digamma)=I+E,
$$
as operators acting on an appropriate function space on $\Omega_2$.
We now
apply $\cS_{\digamma,\Omega_2}$ from both sides. Then
$$
N_\digamma \cS_{\digamma,\Omega_2}=N_\digamma,
$$
since
$$
N_\digamma \cP_{\digamma,\Omega_2}=N_\digamma d^s_{\digamma}
Q_{\digamma,\Omega_2}=0,
$$
in view of the vanishing boundary condition $Q_{\digamma,\Omega_2}$
imposes. On the other hand,
$$
\delta^s_\digamma \cS_{\digamma,\Omega_2}=\delta^s_\digamma-\delta^s_\digamma d^s_{\digamma} Q_{\digamma,\Omega_2}=0
$$
so
$$
\cS_{\digamma,\Omega_2} GN_\digamma=\cS_{\digamma,\Omega_2}+\cS_{\digamma,\Omega_2} E \cS_{\digamma,\Omega_2}.
$$
In order to think of this as giving operators on $\Omega_1$, let
$e_{12}$ be the extension map from $\Omega_1$ to $\Omega_2$,
extending functions (vector fields) as $0$, and $r_{21}$ be the restriction map. (Note
that $e_{12}$ correspondingly maps into a relatively low regularity
space, such as $L^2$, even if one starts with high regularity data.)
Then, with the understanding that $N_\digamma=N_\digamma e_{12}$,
$$
r_{21}\cS_{\digamma,\Omega_2} GN_\digamma=
r_{21}\cS_{\digamma,\Omega_2}e_{12}+K_1,\qquad K_1=r_{21}\cS_{\digamma,\Omega_2} E
\cS_{\digamma,\Omega_2} e_{12}.
$$
We have:

\begin{lemma}\label{lemma:K1}
Let $\digamma$,
$\digamma_0$ as in Corollary~\ref{cor:potential-psdo}.

The operator $K_1=r_{21}\cS_{\digamma,\Omega_2} E
\cS_{\digamma,\Omega_2} e_{12}$ is a smoothing operator in the sense
that it maps $x^kL^2_\scl(\Omega_1)$ to $\Hscb^{s,r}(\Omega_1)$ for every
$s,r,k$. Further, for $\psi\in\CI(\overline{\Omega_2})$ with support in
$\Omega_1$, $\psi K_1\psi\in\Psisc^{-\infty,-\infty}(X)$.

Further, for any $s,r,k$, given $\ep>0$ there exists $\delta>0$ such that if $e_{\delta 1}$ is the
extension map (by $0$) from $\Omega^\delta=\{x\leq\delta\}\cap\Omega_1$ to $\Omega_1$, then
$\|K_1e_{\delta 1}\|_{\cL(x^kL^2_\scl(\Omega^\delta),\Hscb^{s,r}(\Omega_1))}<\ep$.
\end{lemma}

\begin{proof}
This follows from Corollary~\ref{cor:solenoidal}. Indeed, with
$\chi\equiv 1$ near $\pa_\inter\Omega_2$ but with $E=-\Id$ on
$\supp\chi$, and with $\phi\in\CI(\overline{\Omega_2})$ vanishing near
$\supp\chi$, $\supp\phi\cap\WFsc'(E)=\emptyset$, $\phi\equiv 1$ near
$\overline{\Omega_1}$, and with $T$ defined by the first equality,
\begin{equation*}\begin{aligned}
T=\phi \cS_{\digamma,\Omega_2} E
\cS_{\digamma,\Omega_2}\phi=&\phi \cS_{\digamma,\Omega_2}\chi E\chi
\cS_{\digamma,\Omega_2}\phi\\
&+\phi \cS_{\digamma,\Omega_2}(1-\chi) E\chi \cS_{\digamma,\Omega_2}\phi\\
&+\phi \cS_{\digamma,\Omega_2}\chi E(1-\chi) \cS_{\digamma,\Omega_2}\phi\\
&+\phi \cS_{\digamma,\Omega_2}(1-\chi) E(1-\chi) \cS_{\digamma,\Omega_2}\phi.
\end{aligned}\end{equation*}
Now, $E(1-\chi) \cS_{\digamma,\Omega_2}\phi,\phi
\cS_{\digamma,\Omega_2}(1-\chi)E\in\Psisc^{-\infty,-\infty}(X)$ since
they are in $\Psisc^{0,0}(X)$ and $\WFsc'(E)\cap\supp\phi=\emptyset$,
so they are smoothing. In combination with
Corollary~\ref{cor:solenoidal} this gives that
$T:\Hsc^{s',r'}(X)\to\Hsc^{s,r}(X)$ continuously for all $s,r,s',r'$, so composing
with the extension and restriction maps, noting $r_{21}\phi=r_{21}$,
$\phi e_{12}=e_{12}$, proves the first part of the lemma.

To see the smallness claim, note that
$$
K_1 e_{\delta 1}=r_{21}Te_{\delta_1}=r_{21}(Tx^{-1}) (xe_{\delta 2})
$$
$x e_{\delta 2}:x^k L^2_\scl(\Omega^\delta)\to x^k L^2_\scl(\Omega_1)$ has
norm $\leq \sup_{\Omega_\delta} x\leq\delta$, while
$Tx^{-1}:\Hsc^{0,k}(X)\to\Hsc^{s,r}(X)$ is bounded, with bound
independent of $\delta$, and the same is true for
$r_{21}:\Hsc^{s,r}(X)\to\Hscb^{s,r}(\Omega_2)$, completing the proof.
\end{proof}

Now,
\begin{equation*}\begin{aligned}
\cS_{\digamma,\Omega_1}-r_{21}\cS_{\digamma,\Omega_2}e_{12}&=-d^s_{\digamma}Q_{\digamma,\Omega_1}+r_{21}d^s_{\digamma}Q_{\digamma,\Omega_2}e_{12}\\
&=-d^s_{\digamma}Q_{\digamma,\Omega_1}+d^s_{\digamma} r_{21}Q_{\digamma,\Omega_2}e_{12}\\
&=-d^s_{\digamma}(Q_{\digamma,\Omega_1}- r_{21}Q_{\digamma,\Omega_2}e_{12})
\end{aligned}\end{equation*}
and with $\gamma_{\pa_\inter\Omega_1}$ denoting the restriction operator to
$\pa_{\inter}\Omega_1$ as above,
$$
\gamma_{\pa_\inter\Omega_1}(Q_{\digamma,\Omega_1}-
r_{21}Q_{\digamma,\Omega_2}e_{12})=-\gamma_{\pa_\inter\Omega_1}Q_{\digamma,\Omega_2}e_{12},
$$
so
$$
r_{21}\cS_{\digamma,\Omega_2} GN_\digamma=
\cS_{\digamma,\Omega_1}+d^s_\digamma (Q_{\digamma,\Omega_1}- r_{21}Q_{\digamma,\Omega_2}e_{12})+K_1.
$$
Thus, with $B_{\Omega_1}$ being the Poisson operator for
$\Delta_{\digamma,s}$ on $\Omega_1$ as above, 
\begin{equation*}\begin{aligned}
r_{21}\cS_{\digamma,\Omega_2} GN_\digamma=
\cS_{\digamma,\Omega_1}&+d^s_\digamma (Q_{\digamma,\Omega_1}-
r_{21}Q_{\digamma,\Omega_2}e_{12}+B_{\Omega_1}\gamma_{\pa_\inter\Omega_1}Q_{\digamma,\Omega_2}e_{12})\\
&-d^s_\digamma B_{\Omega_1}\gamma_{\pa_\inter\Omega_1}Q_{\digamma,\Omega_2}e_{12}+K_1,
\end{aligned}\end{equation*}
so
$$
\cS_{\digamma,\Omega_1}r_{21}\cS_{\digamma,\Omega_2} GN_\digamma=
\cS_{\digamma,\Omega_1}-\cS_{\digamma,\Omega_1} d^s_\digamma B_{\Omega_1}\gamma_{\pa_\inter\Omega_1}Q_{\digamma,\Omega_2}e_{12}+\cS_{\digamma,\Omega_1}K_1.
$$
Now we consider applying this to vector fields in
$\Omega=\Omega_0$, writing $e_{0j}$ for the extension map to $\Omega_j$. Composing
from the right,
$$
\cS_{\digamma,\Omega_1}r_{21}\cS_{\digamma,\Omega_2} GN_\digamma=
\cS_{\digamma,\Omega_1}e_{01}-\cS_{\digamma,\Omega_1} d^s_\digamma B_{\Omega_1}\gamma_{\pa_\inter\Omega_1}Q_{\digamma,\Omega_2}e_{02}+\cS_{\digamma,\Omega_1}K_1e_{01}.
$$
Now:

\begin{lemma}\label{lemma:K1p}
Let $\digamma$,
$\digamma_0$ as in Corollary~\ref{cor:potential-psdo}. 

The operator $K_1'=\cS_{\digamma,\Omega_1} d^s_\digamma
B_{\Omega_1}\gamma_{\pa_\inter\Omega_1}Q_{\digamma,\Omega_2}e_{02}$ is
smoothing in the sense that for
$\phi\in\CI_c(\overline{\Omega_1}\setminus\pa_\inter\Omega_1)$,
$$
\phi\cS_{\digamma,\Omega_1} d^s_\digamma
B_{\Omega_1}\gamma_{\pa_\inter\Omega_1}Q_{\digamma,\Omega_2}e_{02}:L^2_\scl(\Omega)\to\Hsc^{s,r}(X)
$$
for all $s,r$, and indeed $\phi\cS_{\digamma,\Omega_1} d^s_\digamma
B_{\Omega_1}\gamma_{\pa_\inter\Omega_1}Q_{\digamma,\Omega_2}\phi\in\Psisc^{-\infty,-\infty}(X)$.

Further, for any $s,r,k$, given $\ep>0$ there exists $\delta>0$ such
that if
$$
\Omega\subset\Omega^\delta=\{x\leq\delta\}\cap\Omega_1,
$$
then
$$
\|K'_1\|_{\cL(x^k L^2_\scl(\Omega),\Hscb^{s,r}(\Omega_1))}<\ep.
$$
\end{lemma}

\begin{proof}
By Corollary~\ref{cor:potential-psdo}, using that $\delta_\digamma^s$ is a
differential operator,
$$
\psi Q_{\digamma,\Omega_2}\phi\in\Psisc^{-\infty,-\infty}(X)
$$
whenever $\psi,\phi\in\CI(\overline{\Omega_2})$ have disjoint
supports, also disjoint from $\pa_\inter\Omega_2$ since this
operator is in $\Psisc^{-1,0}(X)$ directly from the corollary, and
then the disjointness of supports gives the conclusion.
Taking such
$\psi,\phi$, as one may, with $\phi\equiv 1$ near $\overline{\Omega}$,
while $\psi\equiv 1$ near $\pa_{\inter}\Omega_1$, we see that
$\gamma_{\pa_\inter\Omega_1}Q_{\digamma,\Omega_2}e_{02}:x^kL^2_\scl(\Omega)\to
\Hsc^{s,r}(\pa_{\inter}\Omega_1)$ for all $s,r,k$, i.e.\ mapping to $\dCI(\pa_{\inter}\Omega_1)$.
The first part then follows from
$B_{\Omega_1}$ mapping this to $\Hscb^{1,r}(\Omega_1)$ for all $r$,
with the additional
property that $\tilde\phi B_{\Omega_1}$ maps to $\Hsc^{s,r}(\Omega)$
for all $s,r$ if $\tilde\phi$ has properties like $\phi$, and then
Corollary~\ref{cor:solenoidal} completes the argument.

For the smallness, we just need to proceed as in Lemma~\ref{lemma:K1},
writing
$$
\gamma_{\pa_\inter\Omega_1}Q_{\digamma,\Omega_2}e_{02}=\gamma_{\pa_\inter\Omega_1}(\psi
Q_{\digamma,\Omega_2}\phi x^{-1}) (xe_{02}),
$$
where now $\psi
Q_{\digamma,\Omega_2}\phi x^{-1}\in\Psisc^{-\infty,-\infty}(X)$, thus
bounded between all weighted Sobolev spaces, with norm independent of
$\delta$, while $xe_{02}:x^k L^2_\scl(\Omega^\delta)\to x^k L^2_\scl(\Omega_2)$ has norm $\leq\delta$.
\end{proof}

Thus,
\begin{equation}\label{eq:Omega-1-param}
\cS_{\digamma,\Omega_1}r_{21}\cS_{\digamma,\Omega_2} GN_\digamma=
\cS_{\digamma,\Omega_1}e_{01}+K_2,
\end{equation}
with $K_2$ smoothing and small if $\Omega\subset\{x\leq\delta\}$, with
$\delta$ suitably small.
This is exactly Equation~(5.7) of \cite{SU-JAMS},
and from this point on we can follow
the argument of the {\em global} work of
Stefanov and Uhlmann \cite[Section~5]{SU-JAMS}, with the addition of having a {\em small} rather than just
compact error, giving invertibility.

Restricting to $\Omega$ from the left, the key remaining step is to compute
$\cS_{\digamma,\Omega}-r_{10}\cS_{\digamma,\Omega_1}e_{01}$ in terms
of the already existing information. As above,
$$
\cS_{\digamma,\Omega}-r_{10}\cS_{\digamma,\Omega_1}e_{01}
=-d^s_\digamma (Q_{\digamma,\Omega}- r_{10}Q_{\digamma,\Omega_1}e_{01}),
$$
but now we compute $u=(Q_{\digamma,\Omega}-
r_{10}Q_{\digamma,\Omega_1}e_{01})f$ using that it is the solution of
the Dirichlet problem $\Delta_{\digamma,s} u=0$,
$\gamma_{\pa_\inter\Omega}u=-\gamma_{\pa_\inter\Omega}Q_{\digamma,\Omega_1}e_{01}f$,
so
\begin{equation}\label{eq:u-f-Omega-to-Omega-1}
u=-B_\Omega
\gamma_{\pa_\inter\Omega}Q_{\digamma,\Omega_1}e_{01}f,
\end{equation}
and using that one can compute
$\gamma_{\pa_\inter\Omega}Q_{\digamma,\Omega_1}e_{01}f$ from
$d^s_\digamma Q_{\digamma,\Omega_1}e_{01}f$. Concretely, we have the
following lemma on functions:

\begin{lemma}\label{lemma:local-ds-inverse-fn}
Let $\Hscd^{1,0}(\Omega_1\setminus\Omega)$ denote the restriction of
elements of $\Hscd^{1,0}(\Omega_1)$ to $\Omega_1\setminus\overline{\Omega}$
(thus, these need not vanish at $\pa_\inter\Omega$), and let
$\rho_{\Omega_1\setminus\Omega}$ be a defining function of
$\pa_\inter\Omega$ as a boundary of $\Omega_1\setminus\Omega$, i.e.\
it is positive in the latter set.
Suppose that
$\pa_x\rho_{\Omega_1\setminus\Omega}> 0$ at $\pa_\inter\Omega$; note
that this is independent of the choice of
$\rho_{\Omega_1\setminus\Omega}$ satisfying the previous criteria (so this is a statement on $x$ being
increasing as one leaves $\Omega$ at $\pa_\inter\Omega$). Then
on functions, for 
$\digamma>0$, $k\in\RR$,
the map
$$
d^s_\digamma:\Hscd^{1,k}(\Omega_1\setminus\Omega)\to x^kL^2(\Omega_1\setminus\Omega)
$$
is injective, with a continuous left inverse $P_{\Omega_1\setminus\Omega}:x^kL^2(\Omega_1\setminus\Omega)\to \Hscd^{1,k}(\Omega_1\setminus\Omega)$.
\end{lemma}

\begin{proof}
Consider $k=0$ first.

The norm of $d^s_\digamma u$
is certainly equivalent to that of $\nabla u$ in $L^2(\Omega_1\setminus\Omega)$ modulo the $L^2(\Omega_1\setminus\Omega)$ norm
of $u$, so one only needs to prove a local Poincar\'e inequality
\begin{equation}\label{eq:local-Poincare-0}
\|u\|_{L^2(\Omega_1\setminus\Omega)}\leq C\|d^s_\digamma u\|_{L^2(\Omega_1\setminus\Omega)}
\end{equation}
to conclude that
$$
\|u\|_{\Hscd^{1,0}(\Omega_1\setminus\Omega)}\leq C\|d^s_\digamma u\|_{L^2(\Omega_1\setminus\Omega)},
$$
which proves the lemma in this case, since it proves that
$d^s_\digamma$, between these spaces, has closed range and is
injective, so it is an isomorphism between
$\Hscd^{1,0}(\Omega_1\setminus\Omega)$ and its range, and then its
inverse in this sense can
be extended continuously to $L^2(\Omega_1\setminus\Omega)$.

But \eqref{eq:local-Poincare-0} can be proved similarly to
Lemma~\ref{lemma:potential-invertible}, by showing that
\begin{equation}\label{eq:local-Poincare-functions}
\|u\|_{L^2(\Omega_1\setminus\Omega)}\leq C\|(x^2 D_x+i\digamma)u\|_{L^2(\Omega_1\setminus\Omega)}.
\end{equation}
Here we want to use $P=x^2D_x+i\digamma$ and $\|P u\|^2$ again; we
need to be careful at $\pa_\inter\Omega$ since $u$ does not vanish
there. Thus, there is an integration by parts boundary term, which we
express in terms of the characteristic function
$\chi_{\Omega_1\setminus\Omega}$:
\begin{equation*}\begin{aligned}
\|Pu\|^2_{L^2(\Omega_1\setminus\Omega)}&=\langle\chi_{\Omega_1\setminus\Omega}
Pu,Pu\rangle_{L^2(\Omega_1)}
=\langle P^*\chi_{\Omega_1\setminus\Omega}
Pu,u\rangle_{L^2(\Omega_1)}\\
&=\langle P^*Pu,u\rangle_{L^2(\Omega_1\setminus\Omega)}+\langle[P^*,\chi_{\Omega_1\setminus\Omega}] Pu,u\rangle_{L^2(\Omega_1)}.
\end{aligned}\end{equation*}
Similarly,
\begin{equation*}\begin{aligned}
\|P_R u\|^2_{L^2(\Omega_1\setminus\Omega)}=\langle P_R^*P_Ru,u\rangle_{L^2(\Omega_1\setminus\Omega)}+\langle[P_R^*,\chi_{\Omega_1\setminus\Omega}] P_Ru,u\rangle_{L^2(\Omega_1)}.
\end{aligned}\end{equation*}
On the other hand, with $P_I$ being $0$th order, the commutator term
vanishes for it. Correspondingly,
\begin{equation*}\begin{aligned}
\|Pu\|^2_{L^2(\Omega_1\setminus\Omega)}&=\langle
P^*Pu,u\rangle_{L^2(\Omega_1\setminus\Omega)}+\langle[P^*,\chi_{\Omega_1\setminus\Omega}]
Pu,u\rangle_{L^2(\Omega_1)}\\
&=\langle
P_R^*P_Ru,u\rangle_{L^2(\Omega_1\setminus\Omega)}+\langle
P_I^*P_Iu,u\rangle_{L^2(\Omega_1\setminus\Omega)}+\langle
i[P_R,P_I]u,u\rangle_{L^2(\Omega_1\setminus\Omega)}\\
&\qquad\qquad+\langle[P^*,\chi_{\Omega_1\setminus\Omega}]
Pu,u\rangle_{L^2(\Omega_1)}\\
&=\|P_Ru\|^2_{L^2(\Omega_1\setminus\Omega)}+\|P_Iu\|^2_{L^2(\Omega_1\setminus\Omega)}+\langle
i[P_R,P_I]u,u\rangle_{L^2(\Omega_1\setminus\Omega)}\\
&\qquad\qquad+\langle[P^*,\chi_{\Omega_1\setminus\Omega}]
Pu,u\rangle_{L^2(\Omega_1)}-\langle[P_R^*,\chi_{\Omega_1\setminus\Omega}]
P_Ru,u\rangle_{L^2(\Omega_1)}.
\end{aligned}\end{equation*}
Now, as $P-P_R$ is $0$th order,
$[P^*,\chi_{\Omega_1\setminus\Omega}]=[P_R^*,\chi_{\Omega_1\setminus\Omega}]$,
so the last two terms on the right hand side give
\begin{equation}\label{eq:d-inj-annulus}
\langle[P^*,\chi_{\Omega_1\setminus\Omega}]
i P_Iu,u\rangle_{L^2(\Omega_1)}=\langle x^2\pa_x\chi_{\Omega_1\setminus\Omega}(\digamma-\frac{n-1}{2}x)u,u\rangle_{L^2(\Omega_1)},
\end{equation}
which is non-negative, at least if $x$ is sufficiently small (or
$\digamma$ large) on $\pa_\inter\Omega$ since $\chi_{\Omega_1\setminus\Omega}=\chi_{(0,\infty)}\circ\rho_{\Omega_1\setminus\Omega}$. Correspondingly, this term
can be dropped, and one obtains \eqref{eq:local-Poincare-functions} at
least if $x$ is small on $\Omega_1$ just as in the proof of
Lemma~\ref{lemma:potential-invertible}. The case of $x$ not
necessarily small on $\Omega_1$ (though small on $\Omega$) follows
exactly as in Lemma~\ref{lemma:potential-invertible} using the
standard Poincar\'e inequality, and even the case where $x$ is not
small on $\Omega$ can be handled similarly since one now has an extra
term at $\pa_\inter\Omega$, away from $x=0$, which one can control
using the standard Poincar\'e inequality. This gives
$$
\|u\|_{\Hscd^{1,0}(\Omega_1\setminus\Omega)}\leq C\|d^s_\digamma u\|_{L^2 (\Omega_1\setminus\Omega)},
$$
showing the claimed injectivity. Further, this gives a continuous
inverse from the range of $d^s_\digamma $, which is closed in
$L^2(\Omega_1\setminus\Omega)$; one can use an orthogonal projection
to this space to define the left inverse
$P_{\Omega_1\setminus\Omega}$, completing the proof when $k=0$.

For general $k$, one can proceed as in
Lemma~\ref{lemma:potential-invertible-weight}, conjugating
$d^s_\digamma$ by $x^k$, which changes it by $x$ times a smooth
one form; this changes $x^2D_x+i\digamma$ by an element of $x\CI(X)$,
with the only effect of modifying the $x\frac{n-1}{2}$ term in
\eqref{eq:d-inj-annulus}, which does not affect the proof.
\end{proof}

We now turn to one forms.

\begin{lemma}\label{lemma:local-ds-inverse-forms}
Let $\Hscd^{1,0}(\Omega_1\setminus\Omega)$ be as in
Lemma~\ref{lemma:local-ds-inverse-fn}, but with values in one-forms, and let
$\rho_{\Omega_1\setminus\Omega}$ be a defining function of
$\pa_\inter\Omega$ as a boundary of $\Omega_1\setminus\Omega$, i.e.\
it is positive in the latter set. 
Suppose that
$\pa_x\rho_{\Omega_1\setminus\Omega}> 0$ at $\pa_\inter\Omega$; note
that this is independent of the choice of
$\rho_{\Omega_1\setminus\Omega}$ satisfying the previous criteria (so this is a statement on $x$ being
increasing as one leaves $\Omega$ at $\pa_\inter\Omega$). Then for
$r\leq -(n-5)/2$, on one-forms
the map
$$
d^s_\digamma:\Hscd^{1,r}(\Omega_1\setminus\Omega)\to \Hsc^{0,r}(\Omega_1\setminus\Omega)
$$
is injective, with a continuous left inverse $P_{\Omega_1\setminus\Omega}:\Hsc^{0,r}(\Omega_1\setminus\Omega)\to \Hscd^{1,r-2}(\Omega_1\setminus\Omega)$.
\end{lemma}

\begin{rem}
Unfortunately the argument given 
above for functions would give an unfavorable boundary term, so 
instead we proceed proving the local Poincar\'e inequality directly and 
using our generalized Korn's inequality, Lemma~\ref{lemma:Korn}, to 
avoid a loss of derivatives. However, our method still produces a loss 
of weight, essentially because as presented the estimate would be
natural for standard tensors, not scattering tensors, hence the presence of the loss $-2$ in the weight in the statement of the lemma. 
\end{rem}

\begin{proof}
As in the work of the first two authors,
\cite[Section~6]{SU-Duke}, we prove the
Poincar\'e inequality using the identity, see \cite[Chapter~3.3]{Sh-book},
\begin{equation}\label{eq:int-along-curves}
\sum_i [v(\gamma(s))]_i\dot\gamma^i(s)=\int_0^s \sum_{ij}[d^s v(\gamma(t))]_{ij}\dot\gamma^i(t)\dot\gamma^j(t)\,dt,
\end{equation}
where $\gamma$ is a unit speed geodesic of the original metric $g$
(thus not of a scattering metric) with
$\gamma(0)\in\pa_\inter\Omega_1$ (so $v(\gamma(0))$
vanishes) and $\gamma(\tau)\in\pa_\inter\Omega\cup\pa X$, with
$\gamma|_{(0,\tau)}$ in $\Omega_1\setminus\overline{\Omega}$. 
Identity (\ref{eq:int-along-curves}) is just an application of the Fundamental Theorem of Calculus with the $s$-derivative of the l.h.s.\ computed using the rules of covariant differentiation. 
In this formula we use
$[d^s v(\gamma(t))]_{ij}$ for the components in the symmetric
2-cotensors corresponding to the standard cotangent
bundle, and similarly for $[v(\gamma(s))]_i$. Notice that this formula
gives an explicit left inverse for $d^s_\digamma$, as discussed below.

Here we choose $\gamma$ such that $x\circ\gamma$ is strictly
monotone in the sense that $-\frac{\pa}{\pa t}(x\circ\gamma)$ is
bounded below (and above) by a positive constant, thus
$(x\circ\gamma)^2\frac{\pa}{\pa t}(x^{-1}\circ\gamma)$ has the same property. Note that
one can construct a smooth family of such geodesics emanating from
$\pa_\inter\Omega_1$, parameterized by $\pa_\inter\Omega$, in a manner
that, with $d\omega$ a smooth measure on $\pa_\inter\Omega_1$,
$d\omega\,dt$ is equivalent to the volume form $dg$, i.e.\ also to $dx\,dy_1\ldots\,dy_{n-1}$.
Thus, for any $k\geq 0$, using $x(\gamma(s))\leq x(\gamma(t))$ along
the geodesic segment, $t\in[0,s]$,
\begin{equation*}\begin{aligned}
|e^{-\digamma/x(\gamma(s))} &x(\gamma(s))^{k} \sum_i [v(\gamma(s))]_i\dot\gamma^i(s)|^2\\
&=\Big|\int_0^s \sum_{ij}e^{-\digamma/x(\gamma(t))}x(\gamma(t))^{k+1}[d^s
v(\gamma(t))]_{ij}\dot\gamma^i(t)\dot\gamma^j(t)\\
&\qquad\qquad\qquad\times e^{-\digamma(1/x(\gamma(s))-1/x(\gamma(t)))}x(\gamma(t))^{-1}\,dt\Big|^2\\
&\leq n^2\Big(\int_0^\tau \sum_{ij}e^{-2\digamma/x(\gamma(t))}x(\gamma(t))^{2k+2}|[d^s
v(\gamma(t))]_{ij}\dot\gamma^i(t)\dot\gamma^j(t)|^2\,dt\Big)\\
&\qquad\qquad\qquad\times\Big(\int_0^s
e^{-2\digamma(1/x(\gamma(s))-1/x(\gamma(t)))}x(\gamma(t))^{-2}\,dt\Big).
\end{aligned}\end{equation*}
Thus,
\begin{equation*}\begin{aligned}
&|e^{-\digamma/x(\gamma(s))} x(\gamma(s))^{k}\sum_i [v(\gamma(s))]_i\dot\gamma^i(s)|^2\\
&\leq C'\Big(\int_0^\tau \sum_{ij}e^{-2\digamma/x(\gamma(t))}x(\gamma(t))^{2k+2}|[d^s
v(\gamma(t))]_{ij}\dot\gamma^i(t)\dot\gamma^j(t)|^2\,dt\Big)\\
&\qquad\qquad\qquad\times\Big(\int_0^s
e^{-2\digamma(1/x(\gamma(s))-1/x(\gamma(t)))}\big(-\frac{\pa}{\pa
  t}(x^{-1}(\gamma(t)))\big)\,dt\Big)\\
&\leq C'\Big(\int_0^\tau e^{-2\digamma/x(\gamma(t))}x(\gamma(t))^{2k+2}|d^s
v(\gamma(t))|^2_{\ell^2}\,dt\Big) \Big(\int_{r_0}^{x^{-1}(\gamma(s))}
e^{-2\digamma(1/x(\gamma(s))-r)}\,dr\Big)
\end{aligned}\end{equation*}
for suitable $r_0>0$,
where we wrote $r=x^{-1}$, and we used the lower bound for $(x\circ\gamma)^2\frac{\pa}{\pa
  t}(x^{-1}\circ\gamma)$ in the second factor, and that $\gamma$ is
unit speed in the first factor, with $\ell^2$ being the
norm as a symmetric map on $T_p X$. The
second factor on the right hand side is bounded by $(2\digamma)^{-1}$, so can be dropped. Now, as
$\tau\,dx+\zeta\,dy=(x^2\tau)\,\frac{dx}{x^2}+(x\zeta)\,\frac{dy}{x}$,
so e.g.\ the $dx^2$ component of $d^s v$ is $x^{-4}$ times the
$\frac{dx^2}{x^4}$ component in the scattering basis,
we have
\begin{equation}\label{eq:ds-sc-standard-l2-bound}
|d^s v(\gamma(t))|_{\ell^2}\leq C x(\gamma(t))^{-4}|d^s v(\gamma(t))|_{\ell^2_\scl},
\end{equation}
so the right hand side is bounded from above by
$$
C''\digamma^{-1}\int_0^\tau e^{-2\digamma/x(\gamma(t))}x(\gamma(t))^{2k-6}|d^s
v(\gamma(t))|^2_{\ell^2_\scl}\,dt
$$
Integrating in the spatial variable, $\gamma(0)\in\pa_\inter\Omega_1$, and using that the
second factor is $(2\digamma)^{-1}$, gives
\begin{equation*}\begin{aligned}
\|e^{-\digamma/x}x^kv(\gamma')\|^2_{L^2(\Omega_1\setminus\Omega)}\leq C\digamma^{-1}\|x^{k-3}e^{-\digamma/x}d^s
v\|^2_{L^2(\Omega_1\setminus\Omega;\Sym^2\Tsc^*X)}.
\end{aligned}\end{equation*}
Using different families of geodesics with tangent vectors covering
$TX$ over $\Omega_1\setminus\Omega$,
\begin{equation*}\begin{aligned}
\|e^{-\digamma/x}x^kv\|^2_{L^2(\Omega_1\setminus\Omega;T^*X)}\leq C\digamma^{-1}\|x^{k-3}e^{-\digamma/x}d^s
v\|^2_{L^2(\Omega_1\setminus\Omega;\Sym^2\Tsc^*X)}.
\end{aligned}\end{equation*}
Now, similarly to \eqref{eq:ds-sc-standard-l2-bound}, but going the
opposite direction,
$$
\|v(p)\|_{\ell^2_\scl}\leq x(p)\|v(p)\|_{\ell^2},
$$
so
\begin{equation*}\begin{aligned}
\|e^{-\digamma/x}x^{k-1}v\|^2_{L^2(\Omega_1\setminus\Omega;\Tsc^*X)}\leq C\digamma^{-1}\|x^{k-3}e^{-\digamma/x}d^s
v\|^2_{L^2(\Omega_1\setminus\Omega;\Sym^2\Tsc^*X)}.
\end{aligned}\end{equation*}
Changing the volume form as well yields
\begin{equation*}\begin{aligned}
\|e^{-\digamma/x}x^{k-1+(n+1)/2}v\|^2_{L^2_\scl(\Omega_1\setminus\Omega;\Tsc^*X)}\leq C\digamma^{-1}\|x^{k-3+(n+1)/2}e^{-\digamma/x}d^s
v\|^2_{L^2_\scl(\Omega_1\setminus\Omega;\Sym^2\Tsc^*X)}.
\end{aligned}\end{equation*}

With $u=e^{-\digamma/x}v$, this gives, for
$u\in\CI(\overline{\Omega_1\setminus\Omega})$, vanishing at $\pa_\inter\Omega_1$, of compact support,
\begin{equation}\label{eq:local-Poincare-1}
\|u\|^2_{\Hsc^{0,r-2}(\Omega_1\setminus\Omega)}\leq C\digamma^{-1}\|d^s_\digamma
u\|^2_{\Hsc^{0,r}(\Omega_1\setminus\Omega)},
\end{equation}
$r\leq -(n-5)/2$,
which then gives the same conclusion, by density and continuity
considerations for $u\in\Hscd^{1,r}(\Omega_1\setminus\Omega)$,
the desired Poincar\'e estimate.

To obtain the $H^1$ estimate, we use Lemma~\ref{lemma:Korn}, which
gives, even for $u\in\Hscb^{1,r-2}(\Omega_1\setminus\Omega)$,
$$
\|u\|^2_{\Hscb^{1,r-2}(\Omega_1\setminus\Omega)}\leq C(\|d^s_\digamma u\|^2_{\Hsc^{0,r-2}(\Omega_1\setminus\Omega)}+\|u\|^2_{\Hsc^{0,r-2}(\Omega_1\setminus\Omega)}),
$$
which combined with \eqref{eq:local-Poincare-1} proves
$$
\|u\|_{\Hscd^{1,r-2}(\Omega_1\setminus\Omega)}\leq C\|d^s_\digamma
u\|_{\Hsc^{0,r} (\Omega_1\setminus\Omega)},\qquad u\in\Hscd^{1,r}(\Omega_1\setminus\Omega),
$$
where recall that our notation is that membership of
$\Hscd^{1,r}(\Omega_1\setminus\Omega)$ only implies vanishing at
$\pa_\inter\Omega_1$, not at $\pa_\inter\Omega$.

Taking into account the above considerations, namely choosing several
families of geodesics to span the tangent space, and working with
$v=e^{\digamma/x}u$, the formula \eqref{eq:int-along-curves} then also gives an explicit
formula for the left inverse.
\end{proof}

Recall now \eqref{eq:u-f-Omega-to-Omega-1}:
\begin{equation*}
u=-B_\Omega
\gamma_{\pa_\inter\Omega}Q_{\digamma,\Omega_1}e_{01}f.
\end{equation*}
Using Lemmas~\ref{lemma:local-ds-inverse-fn}-\ref{lemma:local-ds-inverse-forms},
we conclude that
$$
u=-B_\Omega
\gamma_{\pa_\inter\Omega} P_{\Omega_1\setminus\Omega}d^s_\digamma Q_{\digamma,\Omega_1}e_{01}f,
$$
and as $e_{01}f$ vanishes on $\Omega_1\setminus\Omega$,
$$
\cS_{\digamma,\Omega_1}e_{01}f|_{\Omega_1\setminus\Omega}=-d^s_\digamma Q_{\digamma,\Omega_1}e_{01}f|_{\Omega_1\setminus\Omega},
$$
so
$$
u=B_\Omega
\gamma_{\pa_\inter\Omega} P_{\Omega_1\setminus\Omega}\cS_{\digamma,\Omega_1}e_{01}f,
$$
and thus
$$
\cS_{\digamma,\Omega}-r_{10}\cS_{\digamma,\Omega_1}e_{01}
=-d^s_\digamma B_\Omega
\gamma_{\pa_\inter\Omega} P_{\Omega_1\setminus\Omega}\cS_{\digamma,\Omega_1}e_{01}.
$$
Using \eqref{eq:Omega-1-param} this gives
$$
r_{10}\cS_{\digamma,\Omega_1}r_{21}\cS_{\digamma,\Omega_2} GN_\digamma=\cS_{\digamma,\Omega}+d^s_\digamma B_\Omega
\gamma_{\pa_\inter\Omega} P_{\Omega_1\setminus\Omega}\cS_{\digamma,\Omega_1}e_{01}+r_{10}K_2.
$$
Using \eqref{eq:Omega-1-param} again to express
$\cS_{\digamma,\Omega_1}e_{01}$ on the right hand side, we get
\begin{equation*}\begin{aligned}
&r_{10}\cS_{\digamma,\Omega_1}r_{21}\cS_{\digamma,\Omega_2}
GN_\digamma\\
&\qquad=\cS_{\digamma,\Omega}+
d^s_\digamma B_\Omega
\gamma_{\pa_\inter\Omega}
P_{\Omega_1\setminus\Omega}(\cS_{\digamma,\Omega_1}r_{21}\cS_{\digamma,\Omega_2}
GN_\digamma-K_2)+r_{10}K_2,
\end{aligned}\end{equation*}
which gives
\begin{equation*}\begin{aligned}
&(r_{10}-d^s_\digamma B_\Omega
\gamma_{\pa_\inter\Omega}
P_{\Omega_1\setminus\Omega})
\cS_{\digamma,\Omega_1}r_{21}\cS_{\digamma,\Omega_2} GN_\digamma\\
&\qquad=\cS_{\digamma,\Omega}+(r_{10}-d^s_\digamma B_\Omega
\gamma_{\pa_\inter\Omega}
P_{\Omega_1\setminus\Omega}) K_2.
\end{aligned}\end{equation*}
We now add $\cP_{\digamma,\Omega}$ to both sides, and use that the
smallness of $K_2$ when $\Omega$ is small enough gives that $\Id+(r_{10}-d^s_\digamma B_\Omega
\gamma_{\pa_\inter\Omega}
P_{\Omega_1\setminus\Omega}) K_2$ is invertible. Here we need to be
careful in the 2-tensor case: while $K_2$ is smoothing, including in the sense of producing
additional decay, so there is no problem with applying
$P_{\Omega_1\setminus\Omega}$ regardless of the weighted space we are
considering, the result will have only a weighted estimate in
$\Hsc^{1,r-2}$, $r\leq -(n-5)/2$,
corresponding to Lemma \ref{lemma:local-ds-inverse-forms}, so the
inversion has to be done in a sufficiently negatively weighted space,
namely $\Hsc^{0,r}(\Omega)$, with $r\leq -(n-1)/2$.
Thus,
\begin{equation*}\begin{aligned}
&(\Id+(r_{10}-d^s_\digamma B_\Omega
\gamma_{\pa_\inter\Omega}
P_{\Omega_1\setminus\Omega}) K_2)^{-1}\\
&\qquad\qquad\qquad\circ\Big((r_{10}-d^s_\digamma B_\Omega
\gamma_{\pa_\inter\Omega}
P_{\Omega_1\setminus\Omega})
\cS_{\digamma,\Omega_1}r_{21}\cS_{\digamma,\Omega_2}
GN_\digamma+\cP_{\digamma,\Omega}\Big)=\Id,
\end{aligned}\end{equation*} 
and so multiplying from $\cS_{\digamma,\Omega}$ from the right yields
\begin{equation}\begin{aligned}\label{eq:almost-final-inv}
&(\Id+(r_{10}-d^s_\digamma B_\Omega
\gamma_{\pa_\inter\Omega}
P_{\Omega_1\setminus\Omega}) K_2)^{-1}\\
&\qquad\qquad\circ(r_{10}-d^s_\digamma B_\Omega
\gamma_{\pa_\inter\Omega}
P_{\Omega_1\setminus\Omega})
\cS_{\digamma,\Omega_1}r_{21}\cS_{\digamma,\Omega_2}
GN_\digamma=\cS_{\digamma,\Omega}.
\end{aligned}\end{equation}
Now recall that $N_\digamma=e^{-\digamma/x}LIe^{\digamma/x}$, and that
for $f\in e^{\digamma/x}L^2_\scl(\Omega)$, $\cP_{\digamma,\Omega}
e^{-\digamma/x}f=0$ amounts to $e^{\digamma/x}\delta^s
e^{-\digamma/x}(e^{-\digamma/x}f)=0$, i.e.\ $\delta^s
(e^{-2\digamma/x} f)=0$. This in particular gives an inversion formula
for the geodesic X-ray transform on $e^{2\digamma/x}$-solenoidal
one-forms and symmetric 2-tensors.

In order to state
the stability estimate it is convenient to consider
$(x,y,\lambda,\omega)\in SX$ to actually lie in $\Ssc X$ via the
identification (multiplying the tangent vector by $x$)
$$
(x,y,\lambda\,\pa_x+\omega\,\pa_y)\mapsto (x,y,(\lambda/x)
(x^2\pa_x)+\omega\,(x\pa_y))
$$
Here $\Ssc X=(\Tsc X\setminus o)/\RR^+$ is the sphere bundle in $\Tsc
X$, and in the relevant open set the fiber over a fixed point $(x,y)$ can be identified with vectors of the
form $\tilde\lambda (x^2\pa_x)+\tilde\omega(x\pa_y)$,
$\tilde\omega\in\sphere^{n-2}$, $\tilde\lambda\in\RR$.
Then the region $|\lambda/x|<M$ in $SX$ corresponds to the region
$|\tilde\lambda|<M$; this is now an open subset of
$\Ssc X$. Note that in particular that the
`blow-down map' $(x,y,\tilde\lambda,\tilde\omega)\mapsto
(x,y,x\tilde\lambda,\tilde\omega)$ is smooth, and the composite map
$(x,y,\tilde\lambda,\tilde\omega,t)\mapsto\gamma_{x,y,x\tilde\lambda,\tilde\omega}(t)$
has surjective differential. In particular, with
$$
U=\{|\tilde\lambda|<M\},
$$
the scattering Sobolev spaces are just
restrictions to a domain with smooth boundary. Note that $U$ lies
within the set of $\Omega$-local geodesics; we choose $M$ so that
$\supp\chi\subset M$.

This discussion, in particular \eqref{eq:almost-final-inv}, proves our main local result, for which we
reintroduce the subscript $c$ for the size of the region $\Omega_c$:

\begin{thm}\label{thm:local-linear}
For one forms, let $\digamma>0$; for symmetric 2-tensors let
$\digamma_0>0$ be the maximum of the two constants, denoted
there by $\digamma_0$, in Proposition~\ref{prop:elliptic} and  Corollary~\ref{cor:potential-psdo}.

For $\Omega=\Omega_c$, $c>0$ small, the geodesic X-ray transform on
{\em $e^{2\digamma/x}$-solenoidal}
one-forms and symmetric 2-tensors $f\in e^{\digamma/x}L^2_\scl(\Omega)$,
i.e.\ ones satisfying $\delta^s (e^{-2\digamma/x} f)=0$, is injective, with a stability
estimate and a reconstruction formula
\begin{equation*}\begin{aligned}
f=e^{\digamma/x} (\Id+(r_{10}-d^s_\digamma B_\Omega
\gamma_{\pa_\inter\Omega}
P_{\Omega_1\setminus\Omega}) K_2)^{-1}&(r_{10}-d^s_\digamma B_\Omega
\gamma_{\pa_\inter\Omega}
P_{\Omega_1\setminus\Omega})\\
&\qquad
\circ\cS_{\digamma,\Omega_1}r_{21}\cS_{\digamma,\Omega_2}
Ge^{-\digamma/x}LIf.
\end{aligned}\end{equation*}
Here stability is in the sense that for $s\geq 0$ there exist $R,R'$
such that for any (sufficiently negative in the case of 2-tensors) $r$ the $e^{\digamma/x}\Hsc^{s-1,r}$ norm of
$f$ on $\Omega$ is controlled by the $e^{\digamma/x}\Hsc^{s,r+R}$ norm of $If$ on
$U$, provided $f$ is a priori in $e^{\digamma/x}\Hsc^{s,r+R'}$.
In addition, replacing $\Omega_c=\{\tilde x>-c\}\cap M$ by
$\Omega_{\tau,c}=\{\tau>\tilde x>-c+\tau\}\cap M$, $c$ can be taken
uniform in $\tau$ for $\tau$ in a compact set on which the strict
concavity assumption on level sets of $\tilde x$ holds.
\end{thm}

\begin{rem}
Notice that the proof below gives in particular, by composing $L$ and $I$,
$LI:e^{\digamma/x}\Hsc^{s,r}(X)\to e^{\digamma/x}\Hsc^{s,r-1-s}(X)$,
$s\geq 0$, even
though Proposition~\ref{prop:psdo} implies the mapping property
$LI:e^{\digamma/x}\Hsc^{s,r}(X)\to e^{\digamma/x}\Hsc^{s+1,r}(X)$
(with values in scattering one-forms or 2-tensors). The loss
in the derivatives by one order and of the decay by order $\geq 1$ is due to the non-sharp
treatment of the scattering Fourier integral operators $L,I$ below.
\end{rem}

\begin{proof}
Given \eqref{eq:almost-final-inv},
we just need to show that for $s\geq 0$ there exist $R_1,R_2$ such that
for $k\in\RR$, $L$ is bounded
$$
e^{\digamma/x}\Hsc^{s,k+R_1}(U)\to
e^{\digamma/x}\Hsc^{s,k}(X),
$$
while $I$ is bounded
$$
e^{\digamma/x}\Hsc^{s,k+R_2}(X)\to e^{\digamma/x}\Hsc^{s,k}(U),
$$
with the function spaces on $X$ with values in either one forms or 2-tensors.
To see these boundedness statements, one proceeds as in \cite[Section~3]{UV:local}, prior to
Proposition~3.3, though we change our point of view slightly, as we
are using the `blown-up space' $\Ssc X$ rather than $SX$ for the
geodesic parameterization.

Concretely, $L$ can be written as the composition of
a multiplication operator $M$, by $x\chi(\tilde\lambda)$, resp.\
$x^3\chi(\tilde\lambda)$, for the one-form, resp.\ 2-tensor, case,
times $x^{-1}$ times a sc-one-form
or $x^{-2}$ times a sc-2-tensor factor, with a $-1$ in the
power of $x$ in the definition of $L$ being absorbed into the $\tilde\lambda$ integral, and
a push-forward in which the $\tilde\lambda,\tilde\omega$ variables are integrated
out. The pushforward maps $L^2(U)=x^{-(2n-1+1)/2}L^2_\scl(U)$ to
$L^2(X)=x^{-(n+1)/2}L^2_\scl(X)$ ($L^2$ spaces without subscripts being relative
to smooth non-degenerate densities) with the weights arising from the scattering
volume forms
being $x^{-2n}$, resp.\ $x^{-n-1}$, times a smooth volume
form. Further, it
commutes with multiplication by functions of $x$, so it maps
$e^{\digamma/x}\Hsc^{0,k}(X)$ to $e^{\digamma/x}\Hsc^{0,k+(n-1)/2}(X)$, and (local) lifts of
scattering vector fields $x^2D_x$, $xD_{y_j}$ are still scattering
vector fields so it also maps $e^{\digamma/x}\Hsc^{s,k}(U)$ to
$e^{\digamma/x}\Hsc^{s,k+(n-1)/2}(X)$ for $s\geq 0$ integer, and then
by interpolation for $s\geq 0$. Also, taking into
account the smoothness of $\chi(\tilde\lambda)$, we see that multiplication by
$x^p\chi(\tilde\lambda)$ maps $e^{\digamma/x}\Hsc^{s,k}(U)\to
e^{\digamma/x}\Hsc^{s,k+p}(U)$ for all $s\geq 0$, so in the one form
case
$$
L:e^{\digamma/x}\Hsc^{s,k}(U)\to
e^{\digamma/x}\Hsc^{s,k+(n-1)/2}(X),
$$
while in the 2-tensor case
$$
L:e^{\digamma/x}\Hsc^{s,k}(U)\to
e^{\digamma/x}\Hsc^{s,k+1+(n-1)/2}(X).
$$

On the other hand, $I$ can be written as a pull-back to the subset $U\times\RR$ of
$\Ssc X\times\RR$ from $X$, after contraction with $\gamma'_{x,y,x\tilde\lambda,\tilde\omega}(t)$,
via the map $\gamma:(x,y,\tilde\lambda,\tilde\omega,t)\mapsto
\gamma_{x,y,x\tilde\lambda,\tilde\omega}(t)$, which has surjective differential,
followed by integration over (a uniformly controlled compact subset
of) the $\RR$ factor.
The integration (push-forward) maps $e^{\digamma/x}\Hsc^{s,k}(U\times\RR)\to
e^{\digamma/x}\Hsc^{s,k+1/2}(U)$, where the $1/2$ shift is due to the
density defining the scattering space, as above; by the same argument
as above. On the other hand, the vector $\gamma'_{x,y,x\tilde\lambda,\tilde\omega}(t)$
is $x^{-1}$ times a scattering tangent vector, as discussed in Proposition~\ref{prop:psdo}. Thus, the boundedness of the pull-back as a map
$$
xL^2(X;\Tsc^*X)\to L^2(U\times\RR),\  \text{i.e.}\ x^{-(n-1)/2}L^2_\scl(X)\to
x^{-(2n+1)/2}L^2_\scl(U\times\RR),
$$
in the one-form case, resp.
$$
x^2L^2(X;\Sym^2\Tsc^*X)\to L^2(U\times\RR),\ \text{i.e.}\ x^{-(n-3)/2}L^2_\scl(X)\to
x^{-(2n+1)/2}L^2_\scl(U\times\RR),
$$
in the 2-tensor case,
follows from the surjectivity of the differential of
$\gamma$. (Concretely here this means that as for fixed
$\tilde\lambda,\tilde\omega,t$, $(x,y)\mapsto
\gamma_{x,y,\tilde\lambda,\tilde\omega}(t)=(x',y')$ is a
diffeomorphism, one can rewrite the integral expressing the squared
$L^2$-norm of the pull-back in terms of the squared $L^2$-norm of the
original function using Fubini's theorem.)  Further, the $x$
coordinate along $\gamma_{x,y,\lambda,\omega}$, denoted by $x'$ in
Proposition~\ref{prop:psdo}, satisfies $x'\geq x-C M^2x^2$ (as
$|\lambda/x|\leq M$ on $U$) due to \cite[Equation~(3.1)]{UV:local},
which means that $e^{-\digamma/x}x^{-k} e^{\digamma/x'} (x')^{k}$ is
bounded on the curves as $-\digamma/x+\digamma/x'-k\log (x/x')$ is
bounded above (with the boundedness for $x'\leq x$, holding thanks to
the lower bound for $x'$, being the important
point; for $x'\geq x$, $-\digamma/x-k\log x$ being monotone for small $x$
can be used). Thus, the
mapping property
$$
e^{\digamma/x}\Hsc^{0,k}(X)\to
e^{\digamma/x}\Hsc^{0,k-n/2-1}(U\times\RR),
$$
resp.
$$
e^{\digamma/x}\Hsc^{0,k}(X)\to
e^{\digamma/x}\Hsc^{0,k-n/2-2}(U\times\RR),
$$
follows by the same argument as the $L^2$ boundedness. Finally, by
the chain rule, using just the smoothness of $\gamma$, we obtain that any derivative of the pull-back with
respect to the standard vector fields $V\in\Vf(U)$ can be expressed in
terms of linear combinations with smooth coefficients of standard
derivatives (with respect to $V'\in\Vf(X)$) of the original
function, so in particular for $P\in\Diff^s(U\times\RR)$ and one-forms, $Pf$ is controlled in
$e^{\digamma/x}\Hsc^{0,k-n/2-1}(U\times\RR)$ in terms of derivatives
of order $\leq s$ of $f$ in $e^{\digamma/x}\Hsc^{0,k}(X)$, with a
similar statement for 2-tensors. Now,
(with the above notation) $x'\geq c
x$ for some $c>0$ (so $x/x'$ is bounded), so that $x$ factors of derivatives like $x^2\pa_x$,
$x\pa_{y_j}$, $x\pa_{\tilde \lambda}$, $x\pa_{\tilde\omega_j}$ being
applied to the pull-back can be
turned into factors of $x'$, so we see that if $P\in\Diffsc^s(X)$, then $Pf$ is controlled in
$e^{\digamma/x}\Hsc^{0,k-n/2-1}(U\times\RR)$ in terms of derivatives
of order $\leq s$ of $f$ with respect to the vector fields
$x'\pa_{x'}$ $x'\pa_{y'}$ in $e^{\digamma/x}\Hsc^{0,k}(X)$. Note here
the presence of $x'\pa_{x'}$ rather than $(x')^2\pa_{x'}$, due to the
fact that when one writes the
pull-back as
$f(\Foliation_{x,y,x\tilde\lambda,\tilde\omega}(t),\Loccoord_{x,y,x\tilde\lambda,\tilde\omega}(t))$,
a derivative like $x\pa_{y}$ hitting it is controllable by
$(x'\pa_{x'} f) (\pa_{y}\Foliation)$ and $(x'\pa_{y'}f) (\pa_{y}\Foliation)$,
with the first of these lacking an extra factor of $x'$.
This means that we need to have an extra decay by order $s$ to get a
bounded map between the scattering spaces (since $x'\pa_{x'}=(x')^{-1}((x')^2\pa_{x'})$), so for $s\geq 0$ integer the mapping
property
$$
e^{\digamma/x}\Hsc^{s,k}(X)\to
e^{\digamma/x}\Hsc^{s,k-s-n/2-1}(U\times\RR),
$$
resp.
$$
e^{\digamma/x}\Hsc^{s,k}(X)\to
e^{\digamma/x}\Hsc^{s,k-s-n/2-2}(U\times\RR),
$$
follows,
and then interpolation gives this for all $s\geq 0$. Thus, in the
one-form case
$$
I:e^{\digamma/x}\Hsc^{s,k}(X)\to
e^{\digamma/x}\Hsc^{s,k-s-n/2-1/2}(U\times\RR),
$$
in the 2-tensor case
$$
I:e^{\digamma/x}\Hsc^{s,k}(X)\to e^{\digamma/x}\Hsc^{s,k-s-n/2-3/2}(U\times\RR),
$$
completing the proof.
\end{proof}

If $f\in x^r e^{\digamma/x} L_\scl^2(\Omega)$ then the map $f\to
If$ factors through
$$
\cS_{\digamma,\Omega} e^{-\digamma/x}f=e^{-\digamma/x}f-\cP_{\digamma,\Omega}
e^{-\digamma/x}f
$$
since
$$
Ie^{\digamma/x}\cP_{\digamma,\Omega}
e^{-\digamma/x}f=Id^s e^{\digamma/x}\Delta_{\digamma,s,\Omega}^{-1}e^{\digamma/x}\delta^se^{-2\digamma/x}f=0.
$$
By
Theorem~\ref{thm:local-linear}, $e^{\digamma/x}\cS_{\digamma,\Omega}
e^{-\digamma/x}f\mapsto I e^{\digamma/x}\cS_{\digamma,\Omega} e^{-\digamma/x}f$ is
injective, with a stability estimate. Since
$$
e^{\digamma/x}\cP_{\digamma,\Omega}
e^{-\digamma/x}f=d^s e^{\digamma/x}\Delta_{\digamma,s,\Omega}^{-1}e^{\digamma/x}\delta^se^{-2\digamma/x}f,
$$
this means that we have recovered $f$ up to a potential term, 
i.e.\ in a gauge-free manner we have:

\begin{cor}\label{cor:local-linear-one-form}
Let $\digamma>0$.
With $\Omega=\Omega_c$ as in Theorem~\ref{thm:local-linear},  $r$
sufficiently negative, $c>0$
small, if $f\in e^{\digamma/x} x^r L_\scl^2(\Omega)$ is a one-form then $f=u+d^s v$,
where $v\in e^{\digamma/x} x^r \Hscd^{1,0}(\Omega)$, while $u\in
e^{\digamma/x} x^r L_\scl^2(\Omega)$ can be stably determined from
$If$.

Again, replacing $\Omega_c=\{\tilde x>-c\}\cap M$ by
$\Omega_{\tau,c}=\{\tau>\tilde x>-c+\tau\}\cap M$, $c$ can be taken
uniform in $\tau$ for $\tau$ in a compact set on which the strict
concavity assumption on level sets of $\tilde x$ holds.
\end{cor}

\begin{cor}\label{cor:local-linear-2-tensor}
Let $\digamma,\digamma_0$ be as in Theorem~\ref{thm:local-linear}.
With $\Omega=\Omega_c$ as in Theorem~\ref{thm:local-linear},  $r$ sufficiently negative, $c>0$
small, if $f\in x^re^{\digamma/x} L_\scl^2(\Omega)$ is a symmetric 2-tensor then $f=u+d^s v$,
where $v\in e^{\digamma/x} \Hscd^{1,r-2}(\Omega)$, while $u\in
e^{\digamma/x} x^{r-2}L_\scl^2(\Omega)$ can be stably determined from
$If$.

Again, replacing $\Omega_c=\{\tilde x>-c\}\cap M$ by
$\Omega_{\tau,c}=\{\tau>\tilde x>-c+\tau\}\cap M$, $c$ can be taken
uniform in $\tau$ for $\tau$ in a compact set on which the strict
concavity assumption on level sets of $\tilde x$ holds.
\end{cor}

This theorem has an easy global consequence. To state this, assume
that $\tilde x$ is a globally defined function with level sets $\Sigma_t$
which are strictly concave from the super-level set for $t\in (-T,0]$,
with $\tilde x\leq 0$ on the manifold with boundary $M$.
Then we have:

\begin{thm}\label{thm:global}
Suppose $M$ is compact.
The geodesic X-ray transform is injective and stable modulo potentials
on the
restriction of
one-forms and symmetric 2-tensors $f$ to $\tilde x^{-1}((-T,0])$ in
the following sense. For all $\tau>-T$ there is $v\in \dot H^1_\loc(\tilde
x^{-1}((\tau,0]))$ such that $f-d^sv\in L^2_\loc(\tilde
x^{-1}((\tau,0]))$ can be
stably recovered from $If$. Here for stability we assume that $s\geq
0$, $f$ is in an
$H^s$-space, the norm on $If$ is an
$H^s$-norm, while the norm for $v$ is an $H^{s-1}$-norm.
\end{thm}

\begin{proof}
For the sake of contradiction,
suppose there is no $v$ as stated on $\tilde x^{-1}((\tau_0,0])$ for
some $0>\tau_0>-T$, $If=0$, and let
$$
\tau=\inf\{t\leq 0:\ \exists v_t\in \dot H^1_{\loc}(\{\tilde x>t\})\ \text{s.t.}\
f=d^sv_t\ \text{on}\ \{\tilde x>t\}\}\geq\tau_0.
$$
Thus, for any $\tau'>\tau$, such as
$\tau'<\tau+c/3$, $c$ as in the uniform part of
Corollaries~\ref{cor:local-linear-one-form}-\ref{cor:local-linear-2-tensor} on the levels $[\tau,0]$, there is $v\in \dot
H^1_{\loc}(\{\tilde x>\tau'\})$ such that $f=d^s v$ on $\{\tilde
x>\tau'\}$. Choosing $\phi\in\CI(M)$ identically $1$ near $\tilde
x\geq\tau+2c/3$, supported in $\tilde x>\tau+c/3$, $f-d^s(\phi v)$ is
supported in $\tilde x\leq\tau+2c/3$. But then by the uniform
statement of Corollaries~\ref{cor:local-linear-one-form}-\ref{cor:local-linear-2-tensor}, there exists $v'\in
\dot H^1_\loc (\{\tau-c/3<\tilde x\leq \tau+2c/3\})$ such
that $f-d^s(\phi v)=d^sv'$ in $\tau-c/3<\tilde x<\tau+2c/3$. Extending
$v'$ as $0$, the resulting function $\tilde v'\in \dot H^1_\loc
(\{\tau-c/3<\tilde x\})$ and $d^s \tilde v'$ is the extension of $d^s
v'$ by $0$. Thus, $f=d^s(\phi
v+\tilde v')$, and this contradicts the choice of $\tau$, completing
the proof.

The stability of the recovery follows from a similar argument: by the
uniform property one can recover $f$ modulo potentials in a finite
number of steps: if $c$ works uniformly on $[\tau,0]$, at most
$|\tau|/c+1$ steps are necessary.
\end{proof}

\bibliographystyle{abbrv}
\bibliography{myreferences}

\begin{thebibliography}{10}

\bibitem{AnikonovR}
Yu.~E. Anikonov and V.~G. Romanov.
\newblock On uniqueness of determination of a form of first degree by its
  integrals along geodesics.
\newblock {\em J. Inverse Ill-Posed Probl.}, 5(6):487--490 (1998), 1997.

\bibitem{BGerver}
I.~N. Bernstein and M.~L. Gerver.
\newblock A problem of integral geometry for a family of geodesics and an
  inverse kinematic seismics problem.
\newblock {\em Dokl. Akad. Nauk SSSR}, 243(2):302--305, 1978.

\bibitem{BQ1}
Jan Boman and Eric~Todd Quinto.
\newblock Support theorems for real-analytic {R}adon transforms.
\newblock {\em Duke Math. J.}, 55(4):943--948, 1987.

\bibitem{Croke_scatteringrigidity}
Christopher Croke.
\newblock {S}cattering rigidity with trapped geodesics.
\newblock peprint, 2012.

\bibitem{CrokeH02}
Christopher Croke and Pillar Herreros.
\newblock Lens rigidity with trapped geodesics in two dimensions.
\newblock {\em preprint}, 2012.
\newblock
  http://www.math.upenn.edu/$\sim$ccroke/dvi-papers/SurfaceLensRigidity.pdf.

\bibitem{Croke04b}
Christopher~B. Croke.
\newblock Conjugacy rigidity for non-positively curved graph manifolds.
\newblock {\em Ergodic Theory Dynam. Systems}, 24(3):723--733, 2004.

\bibitem{Dairbekov}
Nurlan~S. Dairbekov.
\newblock Integral geometry problem for nontrapping manifolds.
\newblock {\em Inverse Problems}, 22(2):431--445, 2006.

\bibitem{Herglotz}
G.~Herglotz.
\newblock {\"U}ber die {E}lastizitaet der {E}rde bei {B}eruecksichtigung ihrer
  variablen {D}ichte.
\newblock {\em Zeitschr. f\"ur Math. Phys.}, 52:275--299, 1905.

\bibitem{Venky09}
Venkateswaran~P. Krishnan.
\newblock A support theorem for the geodesic ray transform on functions.
\newblock {\em J. Fourier Anal. Appl.}, 15(4):515--520, 2009.

\bibitem{SV}
Venkateswaran~P. Krishnan and Plamen Stefanov.
\newblock A support theorem for the geodesic ray transform of symmetric tensor
  fields.
\newblock {\em Inverse Probl. Imaging}, 3(3):453--464, 2009.

\bibitem{Melrose-book}
Richard~B. Melrose.
\newblock {\em Geometric scattering theory}.
\newblock Stanford Lectures. Cambridge University Press, Cambridge, 1995.

\bibitem{Mu2}
R.~G. Muhometov.
\newblock On a problem of reconstructing {R}iemannian metrics.
\newblock {\em Sibirsk. Mat. Zh.}, 22(3):119--135, 237, 1981.

\bibitem{Mu1}
R.~G Mukhometov.
\newblock On the problem of integral geometry ({R}ussian).
\newblock {\em Math. problems of geophysics, Akad. Nauk SSSR, Sibirsk., Otdel.,
  Vychisl., Tsentr, Novosibirsk}, 6(2):212--242, 1975.

\bibitem{PaternainSU_13}
Gabriel~P. Paternain, Mikko Salo, and Gunther Uhlmann.
\newblock Tensor tomography on surfaces.
\newblock {\em Invent. Math.}, 193(1):229--247, 2013.

\bibitem{PestovSh}
L.~N. Pestov and V.~A. Sharafutdinov.
\newblock Integral geometry of tensor fields on a manifold of negative
  curvature.
\newblock {\em Sibirsk. Mat. Zh.}, 29(3):114--130, 221, 1988.

\bibitem{PestovU}
Leonid Pestov and Gunther Uhlmann.
\newblock Two dimensional compact simple {R}iemannian manifolds are boundary
  distance rigid.
\newblock {\em Ann. of Math. (2)}, 161(2):1093--1110, 2005.

\bibitem{RanjanS02}
Akhil Ranjan and Hemangi Shah.
\newblock Convexity of spheres in a manifold without conjugate points.
\newblock {\em Proc. Indian Acad. Sci. Math. Sci.}, 112(4):595--599, 2002.

\bibitem{Sh-book}
V.~A. Sharafutdinov.
\newblock {\em Integral geometry of tensor fields}.
\newblock Inverse and Ill-posed Problems Series. VSP, Utrecht, 1994.

\bibitem{Sh-UW}
V.~A. Sharafutdinov.
\newblock Ray transform on {R}iemannian manifolds, lecture notes,
  {U}{W}--{S}eattle.
\newblock available at:
  http://www.ima.umn.edu/talks/workshops/7-16-27.2001/sharafutdinov/, 1999.

\bibitem{Sh-sibir}
V.~A. Sharafutdinov.
\newblock A problem in integral geometry in a nonconvex domain.
\newblock {\em Sibirsk. Mat. Zh.}, 43(6):1430--1442, 2002.

\bibitem{Sh-2D}
Vladimir Sharafutdinov.
\newblock Variations of {D}irichlet-to-{N}eumann map and deformation boundary
  rigidity of simple 2-manifolds.
\newblock {\em J. Geom. Anal.}, 17(1):147--187, 2007.

\bibitem{S-AIP}
P~Stefanov.
\newblock A sharp stability estimate in tensor tomography.
\newblock {\em Journal of Physics: Conference Series}, 124(1):012007, 2008.

\bibitem{SU-Duke}
Plamen Stefanov and Gunther Uhlmann.
\newblock Stability estimates for the {X}-ray transform of tensor fields and
  boundary rigidity.
\newblock {\em Duke Math. J.}, 123(3):445--467, 2004.

\bibitem{SU-JAMS}
Plamen Stefanov and Gunther Uhlmann.
\newblock Boundary rigidity and stability for generic simple metrics.
\newblock {\em J. Amer. Math. Soc.}, 18(4):975--1003, 2005.

\bibitem{SU-AJM}
Plamen Stefanov and Gunther Uhlmann.
\newblock Integral geometry of tensor fields on a class of non-simple
  {R}iemannian manifolds.
\newblock {\em Amer. J. Math.}, 130(1):239--268, 2008.

\bibitem{SU-lens}
Plamen Stefanov and Gunther Uhlmann.
\newblock Local lens rigidity with incomplete data for a class of non-simple
  {R}iemannian manifolds.
\newblock {\em J. Differential Geom.}, 82(2):383--409, 2009.

\bibitem{SUV_localrigidity}
Plamen Stefanov, Gunther Uhlmann, and Andras Vasy.
\newblock Boundary rigidity with partial data.
\newblock {\em arxiv.1210.2084}, 2013.

\bibitem{UV:local}
Gunther Uhlmann and Andras Vasy.
\newblock The inverse problem for the local geodesic ray transform.
\newblock {\em preprint, arxiv.1210.2084}.

\bibitem{WZ}
E.~Wiechert and K.~Zoeppritz.
\newblock {\"U}ber {E}rdbebenwellen.
\newblock {\em Nachr. Koenigl. Geselschaft Wiss. G\"ottingen}, 4:415--549,
  1907.

\end{thebibliography}

\end{document}